\documentclass[english]{article}
\usepackage[T1]{fontenc}
\usepackage[latin9]{inputenc}
\usepackage{color}
\usepackage{amsmath}
\usepackage{amsthm}
\usepackage{amssymb}

\makeatletter
\@ifundefined{date}{}{\date{}}
\usepackage{fullpage}\usepackage{makeidx}\usepackage{amsfonts}\usepackage{latexsym}

\def\C{{\mathbb{C}}}
\def\Z{{\mathbb{Z}}}

\theoremstyle{definition}
\newtheorem{lemma}{Lemma}[section]
\newtheorem{theorem}[lemma]{Theorem}\newtheorem{proposition}[lemma]{Proposition}\newtheorem{definition}[lemma]{Definition}\newtheorem{remark}[lemma]{Remark}\usepackage{times}

\usepackage[blocks]{authblk}

\title{Extended affine Lie algebras, vertex algebras and equivariant $\phi$-coordinated quasi modules}
\author{Fulin Chen, Shaobin Tan and Nina Yu}

\usepackage{babel}

\numberwithin{equation}{section}

\makeatother

\usepackage{babel}
\begin{document}
\maketitle
\begin{abstract}
For any nullity
$2$ extended affine Lie algebra $\mathcal{E}$ of maximal type and
$\ell\in\mathbb{C}$, we prove that there exist a vertex algebra $V_{\mathcal{E}}(\ell)$ and
an automorphism group $G$ of $V_{\mathcal{E}}(\ell)$ equipped with a linear
character $\chi$, such that the category of restricted
$\mathcal{E}$-modules of level $\ell$ is canonically isomorphic to
the category of $(G,\chi)$-equivariant $\phi$-coordinated quasi
$V_{\mathcal{E}}(\ell)$-modules. Moreover, when $\ell$ is a nonnegative
integer, there is a quotient vertex algebra $L_{\mathcal{E}}(\ell)$
of $V_{\mathcal{E}}(\ell)$ modulo by a $G$-stable ideal, and we
prove that the integrable restricted $\mathcal{E}$-modules of level
$\ell$ are exactly the $(G,\chi)$-equivariant $\phi$-coordinated
quasi $L_{\mathcal{E}}(\ell)$-modules.
\end{abstract}

\section{{\normalsize{}{}{}Introduction}}

\emph{Extended affine Lie algebra} (EALA for short) was introduced
by Hoegh-Krohn and Torresani in \cite{H-KT} with applications to
quantum gauge theory, and since then it has been studied extensively
in literature (see \cite{N} and the references therein). An EALA by definition
is  a complex Lie algebra $\mathcal{E}$, together with
a finite-dimensional ad-diagonalizable subalgebra and a nondegenerate
invariant symmetric  bilinear form, satisfying a list of natural axioms.
The isotropic roots (i.e., roots of length $0$) in $\mathcal{E}$
generate a free abelian group of finite rank called the \emph{nullity} of $\mathcal{E}$. EALAs
of nullity $0$ and $1$ precisely coincide with the finite-dimensional
simple Lie algebras and affine Kac-Moody algebras respectively \cite{ABGP}.
Meanwhile, the structure of EALAs with positive nullity are like affine
Kac-Moody algebras in many ways \cite{ABFP,BGK,CNPY}. It is well-known
that affine Kac-Moody algebras through their restricted modules can be naturally associated with vertex
algebras and (twisted) modules  \cite{FZ,FLM,Li1,Li2}.
This association plays an important role in both affine Kac-Moody
algebra theory and vertex (operator) algebra theory. It is a natural
question to ask whether EALAs of nullity $\ge2$ can be associated
with vertex algebras in a similar way as the affine Kac-Moody algebras.

The main goal of this paper is to associate the nullity $2$ EALAs and their modules to vertex algebras.
We would like to point out that the representation theory of nullity $2$ EALAs is
totally different from that for EALAs with nullity $\ge 3$ (see [ESB] and [CLT2]), and the nullity 2 EALAs arose naturally in the
work of Saito \cite{Sa} and Slodowy \cite{Sl} on simple elliptic singularities
and can be connected with Ringel-Hall algebras \cite{LP}.
The subalgebra of an EALA generated by its nonisotropic root vectors
is called the \emph{core} of the EALA, and the core modulo its center is often called
the \emph{centerless core} of the algebra. It is known that the classification of EALAs can be reduced
to the classification of their centerless cores \cite{N}. And
 the centerless cores of nullity $2$ EALAs were
classified by Allison-Berman-Pianzola in \cite{ABP} (see also \cite{GP}).

In
this paper we deal with the nullity $2$ EALAs of
maximal type in the sense that their cores are centrally closed \cite{BGK}.
The
representation theory of nullity $2$ EALAs of maximal type has been
extensively studied (see \cite{G1,G2,G3,GZ,B2,CLT1,CLT2} for example).
By applying the theory of \emph{equivariant $\phi$-coordinated quasi
modules} for vertex algebras developed by Li (see \cite{Li6,Li8}), we associate all nullity $2$ EALAs of maximal type
and their restricted modules to vertex algebras.

Li introduced the notion of  \emph{$\left(G,\chi\right)$-equivariant
quasi modules} for vertex algebras in \cite{Li3,Li4} to associate  certain infinite-dimensional Lie algebras to vertex algebras, where $G$ is a group
and $\chi$ is a linear character of $G$. Li \cite{Li6,Li8} also developed a theory of \emph{$\left(G,\chi\right)$-equivariant
$\phi$-coordinated quasi modules} for nonlocal vertex algebras to associate quantum affine algebras
with quantum vertex algebras (see also \cite{JKLiT,CLTW}), where
$\phi$ is an \emph{associate} of the $1$-dimensional
additive formal group $F(z,w)=z+w$ \cite{Li6}. In
this paper we assume the \emph{associate} $\phi=ze^{w}$, which is indeed the
associate appearing in the quantum vertex algebra theory \cite{Li6,Li8}.

  If $\dot{\mathfrak{g}}$ is a finite-dimensional simple Lie algebra with a diagram automorphism $\dot{\nu}$,
  we denote by $\widetilde{\mathcal{L}}(\dot{\mathfrak{g}},\dot{\nu})$ the corresponding affine Kac-Moody algebra (i.e. nullity 1 EALA), and $\widehat{\mathcal{L}}(\dot{\mathfrak{g}},\dot{\nu})=[\widetilde{\mathcal{L}}(\dot{\mathfrak{g}},\dot{\nu}),
  \widetilde{\mathcal{L}}(\dot{\mathfrak{g}},\dot{\nu})]$ the derived Lie subalgebra of $\widetilde{\mathcal{L}}(\dot{\mathfrak{g}},\dot{\nu})$.
   We know that the restricted modules for the  affine Lie algebra $\widehat{\mathcal{L}}(\dot{\mathfrak{g}},\dot{\nu})$ of level
$\ell\in\mathbb{C}$ can be associated with the $\dot{\nu}$-twisted modules
for the universal affine vertex  algebra $V_{\widehat{\mathcal{L}}(\dot{\mathfrak{g}}) }(\ell, 0)$ of the untwisted affine Lie algebra
$\widehat{\mathcal{L}}(\dot{\mathfrak{g}})=\widehat{\mathcal{L}}(\dot{\mathfrak{g}},\mathrm{id})$ (see \cite{FZ,Li1,FLM,Li2}).
And the integrable restricted modules for $\widehat{\mathcal{L}}(\dot{\mathfrak{g}},\dot{\nu})$
 of level
$\ell\in\mathbb{N}$ can be associated with the $\dot{\nu}$-twisted modules for the corresponding
simple affine vertex  algebras $L_{\widehat{\mathcal{L}}(\dot{\mathfrak{g}}) }(\ell, 0)$ (see \cite{FZ,Li1,Li2}).

Similar to the association of the affine Lie algebra $\widehat{\mathcal{L}}(\dot{\mathfrak{g}},\dot{\nu})$
 and its  restricted  modules with the affine vertex algebras and their  twisted modules,
 we can apply the equivariant $\phi$-coordinated quasi modules for the affine vertex algebras  to associate with the restricted modules for the affine Kac-Moody algebra $\widetilde{\mathcal{L}}(\dot{\mathfrak{g}},\dot{\nu})$. For this purpose, we first, by using results from \cite{Z} and \cite{Li5}, investigate the natural connections
among equivariant $\phi$-coordinated quasi modules, equivariant quasi
modules and twisted modules for the general vertex operator algebras  (see Proposition
\ref{prop:modisoforvoa}). And then we prove that the category of restricted (resp.\,integrable restricted)
modules for the  affine Kac-Moody algebra  $\widetilde{\mathcal{L}}(\dot{\mathfrak{g}},\dot{\nu})$ of level $\ell$ is isomorphic to the category
of equivariant $\phi$-coordinated quasi modules for the universal (resp.\,simple)
affine vertex algebra $V_{\widehat{\mathcal{L}}(\dot{\mathfrak{g}}) }(\ell, 0)$ (resp. $L_{\widehat{\mathcal{L}}(\dot{\mathfrak{g}}) }(\ell, 0)$) (see Theorem \ref{thm:main1}).

Let $\mathfrak{g}$ be the untwisted affine Kac-Moody algebra $\widetilde{\mathcal{L}}(\dot{\mathfrak{g}},\mathrm{id})$, and $\mu$ a nontransitive diagram automorphism  of $\mathfrak{g}$. We can also define the twisted toroidal EALA $\widetilde{\mathfrak{g}}[\mu] $ similar to the construction of twisted affine Kac-Moody algebra. Allison-Berman-Pianzola proved in \cite{ABP} that a nullity
$2$ EALA of maximal type is either isomorphic to a twisted toroidal EALA $\widetilde{\mathfrak{g}}[\mu]$, or to EALA $\widetilde{\mathfrak{sl}}_{N}(\mathbb{C}_{q})$
of type $A_{N-1}$ coordinated by an irrational quantum torus $\mathbb{C}_{q}$ (see Section 5 for details). It
was discovered by Billig in \cite{B2} that an untwisted toroidal
EALA is in general not a vertex Lie algebra in the
sense of \cite{DLM}. Thus one cannot associate the restricted $\widetilde{\mathfrak{g}}[\mu]$-modules,
as the affine Kac-Moody algebra case, to twisted modules of vertex algebras. It was also pointed out in \cite{Li3} that, since
the generating functions of $\widetilde{\mathfrak{sl}}_{N}(\mathbb{C}_{q})$
are not ``local'' in general, the restricted modules of EALA
$\widetilde{\mathfrak{sl}}_{N}(\mathbb{C}_{q})$ cannot be directly
associated to twisted modules of vertex algebras.

As the main result of this paper, we prove that every nullity $2$ EALA of maximal type
can be associated with a vertex algebra through equivariant $\phi$-coordinated
quasi modules. More explicitly, for any nontransitive
diagram automorphism $\mu$ of an untwisted affine Kac-Moody algebra
$\mathfrak{g}$, we construct a vertex algebra $V_{\widehat{\mathfrak{g}}}(\ell,0)$,
a quotient vertex algebra $L_{\widehat{\mathfrak{g}}}(\ell,0)$ of
$V_{\widehat{\mathfrak{g}}}(\ell,0)$, an automorphism group $G_{\mu}$
of $V_{\widehat{\mathfrak{g}}}(\ell,0)$ and $L_{\widehat{\mathfrak{g}}}(\ell,0)$,
and a linear character $\chi_{\omega}$ of $G_{\mu}$. And then we establish
a module category isomorphism from the category of restricted
(resp.\,integrable restricted) $\widetilde{\mathfrak{g}}[\mu]$-modules
of level $\ell$ to the category of $(G_{\mu},\chi_{\omega})$-equivariant $\phi$-coordinated
quasi modules for $V_{\widehat{\mathfrak{g}}}(\ell,0)$ (resp.\,$L_{\widehat{\mathfrak{g}}}(\ell,0)$)
(see Theorem \ref{thm:main2}).
Meanwhile, for any positive integer $N\ge2$ and generic complex number
$q$, we also construct a universal
(resp.\,simple) affine vertex algebras $V_{\widehat{\mathcal{L}}\left(\mathfrak{sl}_{\infty}\right)}\left(\ell,0\right)$
(resp.\,$L_{\widehat{\mathcal{L}}\left(\mathfrak{sl}_{\infty}\right)}\left(\ell,0\right)$)
associated to $\mathfrak{sl}_{\infty}$, an automorphism group $G_{N}$ of the affine vertex algebras, and a linear character $\chi_{q}$
of $G_{N}$. And then we prove that the category of restricted (resp.\,integrable
restricted) $\widetilde{\mathfrak{sl}}_{N}(\mathbb{C}_{q})$-modules
of level $\ell$ is canonically isomorphic to the category of $(G_{N},\chi_{q})$-equivariant
$\phi$-coordinated quasi modules for $V_{\widehat{\mathcal{L}}\left(\mathfrak{sl}_{\infty}\right)}\left(\ell,0\right)$
(resp.\,$L_{\widehat{\mathcal{L}}\left(\mathfrak{sl}_{\infty}\right)}\left(\ell,0\right)$)
(see Theorem \ref{thm:main3}).



The structure of the paper is given as follows. In Section 2 we
recall the notion of  $(G,\chi)$-equivariant $\phi$-coordinated
quasi module for a vertex algebra introduced in \cite{Li7}, and
consider the $(G,\chi)$-equivariant
$\phi$-coordinated quasi modules for the universal enveloping vertex
algebra of a conformal algebra. After giving
natural connections among equivariant $\phi$-coordinated quasi modules,
equivariant quasi modules and twisted modules for vertex operator
algebras and (general) universal affine vertex algebras, in Section
3 we prove the isomorphism  between the categories of  restricted modules
for affine Kac-Moody algebras and equivariant $\phi$-coordinated
quasi modules for affine vertex algebras. In Section 4, for any diagram automorphism
$\mu$ of an untwisted affine Kac-Moody algebra $\mathfrak{g}$, we   construct
an automorphism $\widetilde{\mu}$ of the toroidal EALA $\widetilde{\mathfrak{g}}$ associated to $\mathfrak{g}$
and study the Lie subalgebra $\widetilde{\mathfrak{g}}[\mu]$ of $\widetilde{\mathfrak{g}}$
fixed by the automorphism $\widetilde{\mu}$.
We recall Allison-Berman-Pianzola's
classification result of  nullity 2 EALAs of maximal type  in Section 5. And then in  Section 6,
we define two vertex algebras $V_{\widehat{\mathfrak{g}}}(\ell,0)$
and $L_{\widehat{\mathfrak{g}}}(\ell,0)$, and  associate restricted (resp.\,integrable restricted) $\widetilde{\mathfrak{g}}[\mu]$-modules
of level $\ell$ with equivariant $\phi$-coordinated quasi modules
for the vertex algebra $V_{\widehat{\mathfrak{g}}}(\ell,0)$ (resp.\,$L_{\widehat{\mathfrak{g}}}(\ell,0)$).
Finally, in Section 7, we associate
restricted (resp.\,integrable restricted) $\widetilde{\mathfrak{sl}}_{N}(\mathbb{C}_{q})$-modules
of level $\ell$ with equivariant $\phi$-coordinated quasi modules
for universal (resp.\,simple) affine vertex algebras associated to
$\mathfrak{sl}_{\infty}$.


In this paper we denote by $\mathbb{Z}$, $\mathbb{Z}^{\ast}$, $\mathbb{N}$, $\mathbb{C}$
and $\mathbb{C}^{\ast}$ respectively the sets  of integers, nonzero integers, nonnegative
integers, complex numbers and nonzero complex numbers. And if $\mathfrak{g}$ is a Lie algebra, we denote by $\mathcal{U}(\mathfrak{g})$ the universal enveloping algebra.

\section{Equivariant $\phi$-coordinated quasi modules for vertex algebras}

In this section we recall  the notion and some basics on equivariant
$\phi$-coordinated quasi modules for vertex algebras (cf. \cite{Li8,JKLiT,CLTW}).

\subsection{Definitions and basic properties}

Throughout this paper, let $z,w,z_{0},z_{1},z_{2},\dots$ be mutually
commuting independent formal variables.
We use the standard  notations and conventions as in
\cite{FHL,LLi}.
For example, for a vector space $U$, $U[[z_1,z_2,\dots,z_r]]$ is the space of formal (possibly doubly
infinite) power series in $z_1,z_2,\dots,z_r$ with coefficients in $U$, and
$U((z_1,z_2,\dots,z_r))$  is the space of lower truncated Laurent power series in
$z_1,z_2,\dots,z_r$ with coefficients in $U$.
 Denote a vertex algebra by $V=(V,Y, \mathbf{1})$ , where ${\mathbf{1}}$ is the vacuum vector and $Y(\cdot,z): V\rightarrow \mathrm{Hom}(V,V((z))),\ v\mapsto \sum_{n\in \Z} v_nz^{-n-1}$ is the vertex operator. And  the canonical derivation  on $V$ defined by $v\mapsto v_{-2}\mathbf{1}$
for $v\in V$ is denoted by
$\mathcal{D}$.


For a subset $\Gamma$ of $\mathbb{C}^{\ast}$, denote by $\mathbb{C}_{\Gamma}\left[z\right]$
the set of all polynomials in $\mathbb{C}\left[z\right]$ whose roots
are contained in $\Gamma$. Let $\phi$ be the formal power series $\phi\left(z_{2},z_{0}\right)=z_{2}e^{z_{0}}$,
which is a particular associate of the one-dimensional additive formal
group $F\left(z,w\right)=z+w$ as defined in \cite{Li3}. Now we recall
the notions of equivariant $\phi$-coordinated quasi modules for a vertex
algebra (see \cite{Li8}).

\begin{definition}\label{def:vaphimod} Let $(V, Y, \mathbf{1})$ be a vertex algebra,
$G$ a group of automorphism on $V$ and $\chi:G\to\mathbb{C}^{\ast}$
a linear character of $G$. \emph{A $\left(G,\chi\right)$-equivariant
$\phi$-coordinated quasi $V$-module $\left(W,Y_{W}^{\phi}\right)$}
is a vector space $W$ equipped with a linear map
\[
Y_{W}^{\phi}\left(\cdot,z\right):V\to\text{Hom}\left(W,W\left(\left(z\right)\right)\right)\subset\left(\text{End}W\right)\left[\left[z,z^{-1}\right]\right]
\]
satisfying the following three conditions:
\begin{enumerate}
\item[(i)] \quad{}$Y_{W}^{\phi}\left(\textbf{1},z\right)=1_{W}$;
\item[(ii)] \quad{}$Y_{W}^{\phi}\left(gv,z\right)=Y_{W}^{\phi}\left(v,\chi\left(g\right)z\right)$
for $g\in G,v\in V$;
\item[(iii)] \quad{}For $u,v\in V,$ there exists $f\left(z\right)\in\mathbb{C}_{\chi\left(G\right)}\left[z\right]$
such that
\begin{gather*}
f\left(z_{1}/z_{2}\right)Y_{W}^{\phi}\left(u,z_{1}\right)Y_{W}^{\phi}\left(v,z_{2}\right)\in\text{Hom}\left(W,W\left(\left(z_{1},z_{2}\right)\right)\right),\\
f\left(e^{z_{0}}\right)Y_{W}^{\phi}\left(Y\left(u,z_{0}\right)v,z_{2}\right)
=\left(f\left(z_{1}/z_{2}\right)Y_{W}^{\phi}\left(u,z_{1}\right)Y_{W}^{\phi}\left(v,z_{2}\right)\right)|_{z_{1}=\phi(z_{2},z_{0})}.
\end{gather*}
\end{enumerate}
Furthermore,
a {\em $\left(G,\chi\right)$-equivariant $\phi$-coordinated quasi $V$-module
$\left(W,Y_{W}^{\phi},d\right)$ }is a $\left(G,\chi\right)$-equivariant
$\phi$-coordinated quasi $V$-module $\left(W,Y_{W}^{\phi}\right)$
 equipped with an endomorphism
$d$ of $W$ such that $[d,Y_{W}^{\phi}\left(v,z\right)]=Y_{W}^{\phi}\left(\mathcal{D}v,z\right)\ \text{for }\ v\in V$.
\end{definition} 


Now we fix a vertex algebra $V$, an automorphism group $G$ of $V$, and a linear character $\chi$ of $G$.
The following two results follow respectively from \cite[Lemma 3.7]{Li6} and \cite[Proposition 5.2]{CLTW}.


\begin{lemma}\label{lem:actofd} For a $\left(G,\chi\right)$-equivariant
$\phi$-coordinated quasi $V$-module $\left(W,Y_{W}^{\phi},d\right)$,
 we have
\begin{align*}
\left[d,Y_{W}^{\phi}\left(v,z\right)\right]=Y_{W}^{\phi}\left(\mathcal{D}v,z\right)=z\frac{d}{dz}Y_{W}^{\phi}\left(v,z\right),\quad\forall\ v\in V.
\end{align*}
\end{lemma}


\begin{proposition}\label{prop:Borcherdscomm} Let $\left(W,Y_{W}^{\phi}\right)$
be a $\left(G,\chi\right)$-equivariant $\phi$-coordinated quasi
$V$-module and $\psi:\chi\left(G\right)\to G$ be a section of $\chi$.
Then for $u,v\in V$,
\begin{align*}
\left[Y_{W}^{\phi}\left(u,z_{1}\right),Y_{W}^{\phi}\left(v,z_{2}\right)\right]=\text{Res}_{z_{0}}\sum_{g\in\psi\left(\chi\left(G\right)\right)}Y_{W}^{\phi}\left(Y\left(gu,z_{0}\right)v,z_{2}\right)e^{z_{0}z_{2}\frac{\partial}{\partial z_{2}}}\left(\chi\left(g\right)\delta\left(\frac{\chi\left(g\right)z_{2}}{z_{1}}\right)\right).\label{equivariant phi coordinate, commutator}
\end{align*}
\end{proposition}

Then, we have:

\begin{lemma}\label{lem:premodnil} Let $\left(W,Y_{W}^{\phi}\right)$
be a $\left(G,\chi\right)$-equivariant $\phi$-coordinated quasi
$V$-module, and $u,v\in V$ be such that $u_{n}v=0$ for $n\ge0$.
Then there exists a polynomial $q\left(z\right)\in\mathbb{C}_{\chi\left(G\right)\setminus\{1\}}\left[z\right]$
such that
\begin{equation}
q\left(z_{1}/z_{2}\right)Y_{W}^{\phi}\left(u,z_{1}\right)Y_{W}^{\phi}\left(v,z_{2}\right)=q\left(z_{1}/z_{2}\right)Y_{W}^{\phi}\left(v,z_{2}\right)Y_{W}^{\phi}\left(u,z_{1}\right).\label{eq:modnil1}
\end{equation}
Furthermore, for any such polynomial $q\left(z\right)$ we have
\begin{equation}
Y_{W}^{\phi}\left(u_{-1}v,z_{2}\right)=\left(q\left(z_{1}/z_{2}\right)Y_{W}^{\phi}\left(u,z_{1}\right)Y_{W}^{\phi}\left(v,z_{2}\right)\right)|_{z_{1}=z_{2}}.\label{eq:modnil2}
\end{equation}
\end{lemma}
\begin{proof}
Since $u_{n}v=0$ for $n\ge0$, from Proposition \ref{prop:Borcherdscomm}
it follows that there exist (possibly same) $\lambda_{1},\dots,\lambda_{r}\in\chi\left(G\right)\setminus\left\{ 1\right\} $
such that
\[
\left(z_{1}/z_{2}-\lambda_{1}\right)\cdots\left(z_{1}/z_{2}-\lambda_{r}\right)\left[Y_{W}^{\phi}\left(u,z_{1}\right),Y_{W}^{\phi}\left(v,z_{2}\right)\right]=0.
\]
This proves \eqref{eq:modnil1} with $q\left(z\right)=\left(z-\lambda_{1}\right)\cdots\left(z-\lambda_{r}\right)$.
Then by definition we have
\[q\left(e^{z_{0}}\right)Y_{W}^{\phi}\left(Y\left(u,z_{0}\right)v,z_{2}\right)
=\left(q\left(z_{1}/z_{2}\right)Y_{W}^{\phi}\left(u,z_{1}\right)Y_{W}^{\phi}\left(v,z_{2}\right)\right)|_{z_{1}=z_{2}e^{z_{0}}}.\]
Note that $Y\left(u,z_{0}\right)v\in V\left[\left[z_{0}\right]\right]$,
one can set $z_{0}=0$ in the above equality, which gives \eqref{eq:modnil2}.
\end{proof}
Furthermore, we have the following results (cf. \cite[Proposition 2.3.6]{Li1}, \cite[Proposition 2.10]{Li2}).

\begin{proposition}\label{prop:modnil} Let $u\in V$ such that $u_{n}u=0$
for $n\ge0$, $\ell$ a positive integer and $\left(W,Y_{W}^{\phi}\right)$
a $\left(G,\chi\right)$-equivariant $\phi$-coordinated quasi $V$-module.
Then there exists a polynomial $q\left(z\right)\in\mathbb{C}_{\chi\left(G\right)\setminus\{1\}}\left[z\right]$
such that
\begin{equation}
\left(\prod_{1\le i<j\le\ell+1}q\left(z_{i}/z_{j}\right)\right)Y_{W}^{\phi}\left(u,z_{1}\right)\cdots Y_{W}^{\phi}\left(u,z_{\ell+1}\right)\in\text{Hom}\left(W,W\left(\left(z_{1},\cdots,z_{\ell+1}\right)\right)\right).\label{eq:modnil3}
\end{equation}
Moreover, if $\left(u_{-1}\right)^{\ell+1}\boldsymbol{1}=0$, then
\begin{equation}
\left(\prod_{1\le i<j\le\ell+1}q\left(z_{i}/z_{j}\right)\right)Y_{W}^{\phi}\left(u,z_{1}\right)\cdots Y_{W}^{\phi}\left(u,z_{\ell+1}\right)|_{z_{1}=\cdots=z_{\ell+1}}=0.\label{eq:modnil4}
\end{equation}
And on the other hand, if $W$ is faithful and \eqref{eq:modnil4} holds,
then $\left(u_{-1}\right)^{\ell+1}\boldsymbol{1}=0.$ \end{proposition}
\begin{proof}
Note that the first part of the proposition follows from \eqref{eq:modnil1}. For the second part, it suffices to prove the following identity
\begin{align}
Y_{W}^{\phi}\left(\left(u_{-1}\right)^{\ell+1}\boldsymbol{1},z_{\ell+1}\right)=\left(\prod_{1\le i<j\le\ell+1}q\left(z_{i}/z_{j}\right)\right)Y_{W}^{\phi}\left(u,z_{1}\right)\cdots Y_{W}^{\phi}\left(u,z_{\ell+1}\right)|_{z_{1}=\cdots=z_{\ell+1}}.\label{eq:modnil5}
\end{align}
We prove this equation by induction on $\ell$. When $\ell=1$, the identity \eqref{eq:modnil5}
follows from \eqref{eq:modnil2}. Now we assume that $\ell>1$ and set
$v=\left(u_{-1}\right){}^{\ell}\boldsymbol{1}$. Then by induction
 and \eqref{eq:modnil3}, we obtain
\[
\begin{split} & q\left(z_{1}/z_{\ell+1}\right)^{\ell}Y_{W}^{\phi}\left(u,z_{1}\right)Y_{W}^{\phi}\left(v,z_{\ell+1}\right)\\
= & q\left(z_{1}/z_{\ell+1}\right)^{\ell}Y_{W}^{\phi}\left(u,z_{1}\right)\cdot\left(\prod_{2\le i<j\le\ell+1}q\left(z_{i}/z_{j}\right)\right)Y_{W}^{\phi}\left(u,z_{2}\right)\cdots Y_{W}^{\phi}\left(u,z_{\ell+1}\right)|_{z_{2}=\cdots=z_{\ell+1}}\\
= & \left(\prod_{1\le i<j\le\ell+1}q\left(z_{i}/z_{j}\right)\right)Y_{W}^{\phi}\left(u,z_{1}\right)\cdots Y_{W}^{\phi}\left(u,z_{\ell+1}\right)|_{z_{2}=\cdots=z_{\ell+1}}\in\text{Hom}\left(W,W\left(\left(z_{1},z_{\ell+1}\right)\right)\right).
\end{split}
\]
Since $u_{n}u=0$ for $n\ge0$, we have $u_{n}v=0$ for $n\ge0$.\textcolor{magenta}{{}
}Then it follows from \eqref{eq:modnil1} that
\begin{align*}
 & Y_{W}^{\phi}\left(\left(u_{-1}\right)^{\ell+1}\boldsymbol{1},z_{\ell+1}\right)=Y_{W}^{\phi}\left(u_{-1}v,z_{\ell+1}\right)=\left(q\left(z_{1}/z_{\ell+1}\right){}^{\ell}Y_{W}^{\phi}\left(u,z_{1}\right)Y_{W}^{\phi}\left(v,z_{\ell+1}\right)\right)|_{z_{1}=z_{\ell+1}}\\
= & \left(\prod_{1\le i<j\le\ell+1}q\left(z_{i}/z_{j}\right)\right)Y_{W}^{\phi}\left(u,z_{1}\right)\cdots Y_{W}^{\phi}\left(u,z_{\ell+1}\right)|_{z_{1}=\cdots=z_{\ell+1}},
\end{align*}
which proves the claim \eqref{eq:modnil5} and hence completes the
proof of the proposition.
\end{proof}

\subsection{$\left(G,\chi\right)$-equivariant $\phi$-coordinated quasi $V_{\mathcal{C}}$-modules}

In this subsection we first recall the notion of conformal algebra,
then study the equivariant $\phi$-coordinated quasi modules for
the universal enveloping vertex algebra constructed from a conformal algebra
\cite{CLTW}.

A \emph{ conformal algebra},
also known as a\emph{ vertex Lie algebra} (see \cite{P,DLM}),  is a vector
space $\mathcal{C}$ equipped with a linear operator $\partial$ and
a linear map
\begin{eqnarray}
Y^{-}:\mathcal{C}  \rightarrow \text{Hom}\left(\mathcal{C},z^{-1}\mathcal{C}\left[z^{-1}\right]\right),\quad
u  \mapsto  Y^{-}\left(u,z\right)=\sum_{n\ge0}u_{n}z^{-n-1}\label{eq:}
\end{eqnarray}
such that for any $u,v\in\mathcal{C}$,
\begin{align}
\left[\partial,Y^{-}\left(u,z\right)\right] & =Y^{-}\left(\partial u,z\right)=\frac{d}{dz}Y^{-}\left(u,z\right),\label{derivation property of Y-}\\
Y^{-}\left(u,z\right)v & =\text{Sing}\left(e^{z\partial}Y^{-}\left(v,-z\right)u\right),\nonumber \\
\left[Y^{-}\left(u,z\right),Y^{-}\left(v,w\right)\right] & =\text{Sing}\left(Y^{-}\left(Y^{-}\left(u,z-w\right)v,w\right)\right),\nonumber
\end{align}
where \text{Sing} stands for the singular part.

It was proved in \cite[Remark 4.2]{P}   that a  conformal algebra structure on a vector space $\mathcal{C}$
amounts to a Lie algebra structure on the following quotient space of $\mathbb{C}\left[t,t^{-1}\right]\otimes\mathcal{C}$:
\[
\widehat{\mathcal{C}}
=\mathbb{C}\left[t,t^{-1}\right]\otimes\mathcal{C}/\left(1\otimes\partial+\frac{d}{dt}\otimes1\right)\left(\mathbb{C}\left[t,t^{-1}\right]\otimes\mathcal{C}\right).
\]

\begin{lemma}\label{lem:hatc} Let $\mathcal{C}$ be a vector space
equipped with a linear operator $\partial$ and a linear map $Y^{-}$
as given in (\ref{eq:}) such that (\ref{derivation property of Y-})
holds. Then $\mathcal{C}$ is a conformal algebra if and only if there
is a Lie algebra structure on $\widehat{\mathcal{C}}$ such that
\begin{align}
\left[u\left(m\right),v\left(n\right)\right]=\sum_{i\ge0}{{m} \choose {i}}\left(u_{i}v\right)\left(m+n-i\right),\label{eq:relationinhatc}
\end{align}
for $u,v\in\mathcal{C}$, $m,n\in\mathbb{Z},$ where $u(m)$ stands for the image of $t^{m}\otimes u$ in $\widehat{\mathcal{C}}$. \end{lemma}

Let $\mathcal{C}$ be a conformal algebra. Set
\begin{align}
\widehat{\mathcal{C}}^{-}=\text{Span}\{u(-m-1)\mid u\in\mathcal{C},m\in\mathbb{N}\}\ \text{\ and\ \ }\widehat{\mathcal{C}}^{+}=\text{Span}\{u(m)\mid u\in\mathcal{C},m\in\mathbb{N}\}.\label{eq:polardecofC}
\end{align}
Then $\widehat{\mathcal{C}}=\widehat{\mathcal{C}}^{+}\oplus\widehat{\mathcal{C}}^{-}$
and both $\widehat{\mathcal{C}}^{+}$ and $\widehat{\mathcal{C}}^{-}$
are subalgebras of the Lie algebra $\widehat{\mathcal{C}}$. Moreover,
the map
\begin{eqnarray}
\mathcal{C}\to\widehat{\mathcal{C}}^{-},\quad u\mapsto u\left(-1\right)\label{eq:cimbedhatc}
\end{eqnarray}
is an isomorphism of vector spaces \cite[Theorem 4.6]{P}. Consider
the induced $\widehat{\mathcal{C}}$-module
\[
V_{\mathcal{C}}=\mathcal{U}\left(\mathcal{\widehat{\mathcal{C}}}\right)\otimes_{\mathcal{U}\left(\widehat{\mathcal{C}}^{+}\right)}\mathbb{C},
\]
where $\mathbb{C}$ is the one dimensional trivial $\widehat{\mathcal{C}}^{+}$-module.
Set $\boldsymbol{1}=1\otimes1\in V_{\mathcal{C}}.$ Identify $\mathcal{C}$
as a subspace of $V_{\mathcal{C}}$ through the linear map $u\mapsto u\left(-1\right)\boldsymbol{1}.$
It was proved in \cite{P} that there exists a unique vertex algebra
structure on $V_{\mathcal{C}}$, called \emph{the universal enveloping
vertex algebra of $\mathcal{C}$}, with $\boldsymbol{1}$ as the vacuum
vector and $Y\left(u,z\right)=u\left(z\right)=\sum_{n\in\mathbb{Z}}u\left(n\right)z^{-n-1}$
for $u\in\mathcal{C}$. The map $-\frac{d}{dt}\otimes1$ (or equivalently,
$1\otimes\partial$) on $\C[t,t^{-1}]\otimes\mathcal{C}$ induces
a derivation $\mathcal{D}$ on $\widehat{\mathcal{C}}$ such that
\begin{align}
\mathcal{D}\left(u\left(m\right)\right)=-mu\left(m-1\right),\quad\forall\ u\in\mathcal{C},\ m\in\mathbb{Z}.
\end{align}
Note that $\mathcal{D}$ preserves the subalgebra $\widehat{\mathcal{C}}^{-}$.
Thus it can be uniquely extended to a derivation on $\mathcal{U}\left(\widehat{\mathcal{C}}^{-}\right)\cong V_{\mathcal{C}}$,
which coincides with the canonical derivation on the vertex algebra
$V_{\mathcal{C}}$.

Recall that an \emph{automorphism} $\varphi$ of the conformal algebra
$\mathcal{C}$ is a linear automorphism such that $\varphi\circ\partial=\partial\circ\varphi$
and $\varphi\left(u_{i}v\right)=\varphi\left(u\right){}_{i}\varphi\left(v\right)$
for $u,v\in\mathcal{C}$ and $i\in\mathbb{N}$ (see \cite{K2}). We have:

\begin{lemma}\label{lem:conformalaut} Let $\varphi$ be a linear
map of the conformal algebra $\mathcal{C}$ such that $\varphi\circ\partial=\partial\circ\varphi$.
Then the linear map
\begin{eqnarray}
\hat{\varphi}:\widehat{\mathcal{C}}\to\widehat{\mathcal{C}},\quad u\left(m\right)\mapsto\varphi\left(u\right)\left(m\right)\quad(u\in\mathcal{C},m\in\mathbb{Z})\label{eq:defhatvarphi}
\end{eqnarray}
on the Lie algebra $\widehat{\mathcal{C}}$ is well-defined. Furthermore,
$\varphi$ is an automorphism of the conformal algebra $\mathcal{C}$
if and only if $\hat{\varphi}$ is an automorphism of the Lie algebra
$\widehat{\mathcal{C}}.$ \end{lemma}
\begin{proof}
The first assertion of the lemma is easy to see. For the second part, let $u,v\in\mathcal{C}$
and $m,n\in\mathbb{Z}$. Then by \eqref{eq:relationinhatc} we have
\begin{align}\label{eq:varphiandhatvarphi1}
\hat{\varphi}\left(\left[u\left(m\right),v\left(n\right)\right]\right)=\sum_{i\ge0}{{m} \choose {i}}\hat{\varphi}\left(\left(u_{i}v\right)\left(m+n-i\right)\right)=\sum_{i\ge0}{{m} \choose {i}}\varphi\left(u_{i}v\right)\left(m+n-i\right),
\end{align}
and
\begin{align}\label{eq:varphiandhatvarphi2}
\left[\hat{\varphi}\left(u\left(m\right)\right),\hat{\varphi}\left(v\left(n\right)\right)\right]=\left[\varphi(u)\left(m\right),\varphi(v)\left(n\right)\right]=\sum_{i\ge0}{{m} \choose {i}}\left((\varphi u)_{i}(\varphi v)\right)\left(m+n-i\right).
\end{align}
This immediately implies that if $\varphi$ is an automorphism of
$\mathcal{C}$, then $\hat{\varphi}$ is a Lie algebra automorphism of $\widehat{\mathcal{C}}$.
On the other hand, let $\hat{\varphi}$ be a Lie algebra automorphism of $\widehat{\mathcal{C}}$.
Assume that $\varphi$ is not an automorphism of $\mathcal{C}$. Then
there exist $u,v\in\mathcal{C}$ such that $\varphi(u_{i}v)\ne(\varphi u)_{i}(\varphi v)$
for some $i\in\mathbb{N}$. Let $i_{0}$ be the maximal one among
such integers. Take $m,n\in\mathbb{Z}$ such that $m+n=i_{0}-1$.
Then we have
\[
\sum_{i=0}^{i_{0}-1}{{m} \choose {i}}\varphi\left(u_{i}v\right)\left(m+n-i\right),\ \sum_{i=0}^{i_{0}-1}{{m} \choose {i}}\left((\varphi u)_{i}(\varphi v)\right)\left(m+n-i\right)\in\widehat{\mathcal{C}}^{+},
\]
and
\[
\sum_{i\ge i_{0}}{{m} \choose {i}}\varphi\left(u_{i}v\right)\left(m+n-i\right),\ \sum_{i\ge i_{0}}{{m} \choose {i}}\left((\varphi u)_{i}(\varphi v)\right)\left(m+n-i\right)\in\widehat{\mathcal{C}}^{-}.
\]
Recall that $\widehat{\mathcal{C}}^{+}\cap\widehat{\mathcal{C}}^{-}=\{0\}$ and $\hat{\varphi}$ is a Lie algebra automorphism.
Then by \eqref{eq:varphiandhatvarphi1}, \eqref{eq:varphiandhatvarphi2} and the maximality of $i_0$
we obtain that $\varphi(u_{i_{0}}v)(-1)=((\varphi u)_{i_{0}}(\varphi v))(-1)$.
This together with
 \eqref{eq:cimbedhatc} gives $\varphi(u_{i_{0}}v)=(\varphi u)_{i_{0}}(\varphi v)$,
a contradiction. Therefore we have finished the proof of the lemma.
\end{proof}

We define a multiplication on the
loop space $\mathbb{C}\left[t,t^{-1}\right]\otimes\mathcal{C}$ by
\begin{equation}
\left[f\left(t\right)\otimes u,g\left(t\right)\otimes v\right]=\sum_{i\ge0}\left(\frac{1}{i!}\left(t\frac{d}{dt}\right)^{i}f\left(t\right)\right)g\left(t\right)\otimes\left(u_{i}v\right)\quad
 \left(f\left(t\right),g\left(t\right)\in\mathbb{C}\left[t,t^{-1}\right],\, u,v\in\mathcal{C}\right).
\label{Lie relation of =00003D00003D00003D00003D00005ChatCphi}
\end{equation}
By \cite[Remark 2.7d]{K2} this multiplication affords a Lie algebra structure
on the quotient space
\[
\overline{\mathcal{C}}=\left(\mathbb{C}\left[t,t^{-1}\right]\otimes\mathbb{\mathcal{C}}\right)/\text{Im}\left(1\otimes\partial+t\frac{d}{dt}\otimes1\right).
\]
Denote by $\mathbf{d}$ the derivation on $\overline{\mathcal{C}}$
induced from $-t\frac{d}{dt}\otimes1\in\mathrm{End}\left(\mathbb{C}\left[t,t^{-1}\right]\otimes\mathcal{C}\right)$.
Form the semi-direct product Lie algebra
\[
\widetilde{\mathcal{C}}=\overline{\mathcal{C}}\rtimes\C\textbf{d}.
\]
For $u\in\mathcal{C}$ and $m\in\mathbb{Z}$, denote by $u\left[m\right]$
the image of $t^{m}\otimes u$ in $\overline{\mathcal{C}}$. Then
the Lie relations on $\widetilde{\mathcal{C}}$ are given by
\begin{align}
\left[u\left[m\right],v\left[n\right]\right]=\sum_{i\ge0}\frac{m^{i}}{i!}\left(u_{i}v\right)\left[m+n\right],\quad\left[\textbf{d},u\left[m\right]\right]=-mu\left[m\right]\label{eq:relationintildec}
\end{align}
for $u,v\in\mathcal{C},$ $m,n\in\mathbb{Z}$.
The following notion was first introduced in \cite{G-KK} (see also
\cite{Li3}).

\begin{definition} A \emph{$G$-conformal algebra} is a conformal
algebra $\mathcal{C}$ together with an automorphism group $G$ of
$\mathcal{C}$ such that for $u\in\mathcal{C}$, $Y^{-}\left(gu,z\right)=0$
for all but finitely many $g\in G$. \end{definition}

Let $\mathcal{C}$ be a $G$-conformal algebra and $\chi:G\to\mathbb{C}^{\ast}$
be a linear character. For any $g\in G$, it is easy to check
that (cf. \cite[Lemma 5.8]{CLTW}) the linear map
\begin{eqnarray*}
\bar{g}:\overline{\mathcal{C}}\to\overline{\mathcal{C}},\quad u\left[m\right]\mapsto\chi\left(g\right)^{m}\left(gu\right)\left[m\right]
\end{eqnarray*}
defines an automorphism of the Lie algebra $\overline{\mathcal{C}}$.
Furthermore, $\overline{g}$ can be extended to be an automorphism $\tilde{g}$
of $\widetilde{\mathcal{C}}$ by
\begin{align}
\tilde{g}|_{\overline{\mathcal{C}}}=\bar{g}\quad\text{and}\quad\tilde{g}\left(\mathbf{d}\right)=\mathbf{d}.\label{eq:deftildeg}
\end{align}
Following \cite[Section 4]{Li6}, we define a new operation on $\bar{\mathcal{C}}$
by
\[
\left(a,b\right)\mapsto\sum_{g\in G}\left[\overline{g}a,b\right]\quad (a,b\in\bar{\mathcal{C}}).
\]
It was proved therein that the quotient space
\begin{equation}
\overline{\mathcal{C}}\left[G\right]=\bar{\mathcal{C}}/\text{Span\ensuremath{\left\{ \bar{g}a-a\mid a\in\overline{\mathcal{C}},g\in G\right\} } }\label{quotient of Lie algebra C_phi}
\end{equation}
is a Lie algebra under the operation. For $u\in\mathcal{C}$ and $n\in\mathbb{Z},$
we denote by $\overline{u\left[n\right]}$ the image of $u\left[n\right]$
under the quotient map $\overline{\mathcal{C}}\to\overline{\mathcal{C}}\left[G\right]$.
Since $g\circ\mathbf{d}=\mathbf{d}\circ g$ for $g\in G$, $\mathbf{d}$
descends to a derivation on $\overline{\mathcal{C}}\left[G\right]$
and hence we have the semi-direct product Lie algebra
\begin{eqnarray}
\widetilde{\mathcal{C}}\left[G\right]=\overline{\mathcal{C}}\left[G\right]\rtimes\mathbb{C}\mathbf{d}.\label{eq:deftcg}
\end{eqnarray}

For $u\in\mathcal{C}$, we form the generating function
$u\left[z\right]=\sum_{n\in\mathbb{Z}}\overline{u\left[n\right]}z^{-n}$, and we
 say that a module $W$ for $\widetilde{\mathcal{C}}\left[G\right]$
or $\overline{\mathcal{C}}\left[G\right]$ is \emph{restricted }if
$u\left[z\right]\in\text{Hom}\left(W,W\left(\left(z\right)\right)\right)$
for $u\in\mathcal{C}$. It is known from \cite{P} that an automorphism
of $\mathcal{C}$ can be lifted uniquely to an automorphism of $V_{\mathcal{C}}$.
In particular, the automorphism group $G$ of $\mathcal{C}$ can be
viewed as an automorphism group of $V_{\mathcal{C}}$. Now we have
the following result.

\begin{proposition} Let $\mathcal{C}$ be a $G$-conformal algebra
and $\chi:G\rightarrow\mathbb{C}^{\ast}$ be an injective linear character.
Then the $\left(G,\chi\right)$-equivariant $\phi$-coordinated quasi
$V_{\mathcal{C}}$-modules $\left(W,Y_{W}^{\phi},d\right)$ are exactly
the restricted $\widetilde{\mathcal{C}}\left[G\right]$-modules $W$
with
\begin{align*}
d=\textbf{d}\quad\text{and}\quad Y_{W}^{\phi}\left(u,z\right)=u\left[z\right]\quad\text{for}\ u\in\mathcal{C}.
\end{align*}
\end{proposition}
\begin{proof}
It was proved in \cite[Theorem 5.12]{CLTW} that the $\left(G,\chi\right)$-equivariant
$\phi$-coordinated quasi $V_{\mathcal{C}}$-modules $\left(W,Y_{W}^{\phi}\right)$
are exactly the restricted $\overline{\mathcal{C}}\left[G\right]$-modules
$W$ with $Y_{W}^{\phi}\left(u,z\right)=u\left[z\right]$ for $u\in\mathcal{C}$.
Note that for $u\in\mathcal{C}$, we have
\begin{align*}
\left[\textbf{d},u\left[z\right]\right]=\sum_{n\in\mathbb{Z}}\left[\textbf{d},\overline{u\left[n\right]}\right]z^{-n}=\sum_{n\in\Z}-n\overline{u\left[n\right]}z^{-n}=z\frac{d}{dz}u\left[z\right].
\end{align*}
This together with Lemma \ref{lem:actofd} proves the proposition.
\end{proof}
\begin{remark}\label{rem:isoofcovandfix} Assume that $G=\left\langle g\right\rangle $
is a cyclic group of finite order $T$, and the linear character $\chi$ is injective. Then it is straightforward
to see that $\overline{\mathcal{C}}\left[G\right]$ is isomorphic
to the subalgebra of $\overline{\mathcal{C}}$ fixed by $\bar{g}$.
And the isomorphism is given by
\[
\overline{u\left[n\right]}\mapsto\sum_{p=0}^{T-1}\left(\bar{g}\right)^{p}\left(u\left[n\right]\right)
\]
for $u\in\mathcal{C}$ and $n\in\mathbb{Z}.$ Furthermore, this isomorphism
can be extended to an isomorphism from $\widetilde{\mathcal{C}}\left[G\right]$
to the subalgebra of $\widetilde{\mathcal{C}}$ fixed by $\tilde{g}$
such that $\mathbf{d}\mapsto\mathbf{d}$. \end{remark}

\section{Equivariant $\phi$-coordinated quasi modules for affine vertex algebras}

In this section, we study the connection between equivariant $\phi$-coordinated
quasi modules for universal (resp.\,simple) affine vertex algebras
and restricted (resp.\,integrable restricted) modules for affine
Kac-Moody algebras.

\subsection{Equivalence of module categories for vertex operator algebras}

We first study connections among equivariant $\phi$-coordinated
quasi modules, equivariant quasi modules and twisted modules for
vertex operator algebras (see \cite{FLM,FHL}). Recall that for any finite order automorphism $\sigma$ of a vertex
algebra $V$, there is (weak) \emph{$\sigma$-twisted
$V$-module} $\left(W,Y_{W}^{t}\right)$, where $W$ is a vector space
and $Y_{W}^{t}\in\mathrm{Hom}\left(W,W\left(\left(z^{1/N}\right)\right)\right)$
with $N$ the order of $\sigma$ (cf. \cite{Li2,FLM}). We also recall that a \textit{$\mathbb{Z}$-graded vertex algebra} is a vertex algebra
$V$ equipped with a $\mathbb{Z}$-grading $V=\oplus_{n\in\mathbb{Z}}V_{(n)}$
such that ${\bf 1}\in V_{(0)}$ and
\begin{eqnarray}
u_{m}V_{(n)}\subset V_{(n+k-m-1)}\ \ \mbox{ for }u\in V_{(k)},\ m,n,k\in\mathbb{Z}.
\end{eqnarray}
Define the linear operator $L\left(0\right)$
on a $\mathbb{Z}$-graded vertex algebra $V$ by $L\left(0\right)v=nv$
for $v\in V_{(n)}$ with $n\in\Z$. Then we have the following definition
(cf. \cite{Li3,Li4}):

\begin{definition}\label{def:vamod} Let $(V, Y, \mathbf{1})$ be a $\mathbb{Z}$-graded
vertex algebra, $G$ an automorphism group of $V$ preserving the
$\mathbb{Z}$-grading of $V$ and $\chi:G\rightarrow\mathbb{C}^{\ast}$
a linear character. \emph{A $\left(G,\chi\right)$-equivariant quasi
$V$-module} $\left(W,Y_{W}\right)$ is a vector space $W$ equipped
with a linear map $Y_{W}\left(\cdot,z\right):V\to\text{Hom}\left(W,W\left(\left(z\right)\right)\right)$
satisfying the following conditions:
\begin{enumerate}
\item[(i)] \quad{}$Y_{W}\left(\textbf{1},z\right)=1_{W}$;
\item[(ii)] \quad{}$Y_{W}\left(\chi\left(g\right){}^{-L(0)}gv,z\right)=Y_{W}\left(v,\chi\left(g\right)z\right)$
for $g\in G,v\in V$;
\item[(iii)] \quad{}For $u,v\in V,$ there exists $f\left(z\right)\in\mathbb{C}_{\chi\left(G\right)}\left[z\right]$
such that
\begin{gather*}
f\left(z_{1}/z_{2}\right)Y_{W}\left(u,z_{1}\right)Y_{W}\left(v,z_{2}\right)\in\text{Hom}\left(W,W\left(\left(z_{1},z_{2}\right)\right)\right),\\
f\left((z_{2}+z_{0})/z_{2}\right)Y_{W}\left(Y\left(u,z_{0}\right)v,z_{2}\right)=\left(f\left(z_{1}/z_{2}\right)Y_{W}\left(u,z_{1}\right)Y_{W}\left(v,z_{2}\right)\right)|_{z_{1}=z_{2}+z_{0}}.
\end{gather*}
\end{enumerate}
\end{definition}

\begin{remark} If the automorphism group $G=\left\{ 1\right\} $ is the trivial group in the definition, one has the usual module for vertex algebra.  And if replacing the associate $\phi=z_{2}+z_{0}$  in (iii) by $\phi=z_{2}e^{z_{0}}$, one gets the notion of a $\phi$-coordinated $V$-module defined in Definition \ref{def:vaphimod} . \end{remark}


If $V=(V, Y(\cdot, z), {\mathbf{1}})$ is a $\mathbb{Z}$-graded vertex algebra,
we define a linear map
\begin{eqnarray*}
Y\left[\cdot,z\right]:V  \rightarrow  \mathrm{End}\left(V\right)\left[\left[z,z^{-1}\right]\right],\quad
v  \mapsto  Y\left(e^{zL(0)}v,e^{z}-1\right).
\end{eqnarray*}
Then $\left(V,Y\left[\cdot,z\right],{\mathbf{1}}\right)$ also carries
a vertex algebra structure \cite{Z,Li7}. Note that if $G$ is an
automorphism group of the vertex algebra $V$ preserving the $\mathbb{Z}$-grading of
$V$, then it is also an automorphism group of $\left(V,Y\left[\cdot,v\right],\mathbf{1}\right)$.
The following result is a generalization of \cite[Proposition 5.8]{Li7}.

\begin{proposition}\label{prop:phimodvsmod} Let $V=(V, Y(\cdot, z), {\mathbf{1}})$ be a $\mathbb{Z}$-graded
vertex algebra, $G$ an automorphism group of $V$ preserving the
$\mathbb{Z}$-grading of $V$ and $\chi$ a linear character of $G$.
Then the $\left(G,\chi\right)$-equivariant quasi $V$-modules $\left(W,Y_{W}\right)$
are exactly the $\left(G,\chi\right)$-equivariant $\phi$-coordinated
quasi modules $\left(W,Y_{W}^{\phi}\right)$ for the vertex algebra $\left(V,Y\left[\cdot,z\right],\mathbf{1}\right)$
with
\begin{align*}
Y_{W}^{\phi}\left(v,z\right)=Y_{W}\left(z^{L(0)}v,z\right),\quad\forall\ v\in V.
\end{align*}
\end{proposition}
\begin{proof}
Assume first that $\left(W,Y_{W}\right)$ is a $\left(G,\chi\right)$-equivariant
quasi $V$-module. Then for $v\in V$ and $g\in G$ we have
\begin{align*}
Y_{W}^{\phi}\left(v,\chi\left(g\right)z\right)=Y_{W}\left(\left(\chi(g)z\right){}^{L(0)}v,\chi(g)z\right)=Y_{W}\left(z^{L(0)}gv,z\right)=Y_{W}^{\phi}\left(gv,z\right).
\end{align*}
This together with \cite[Proposition 5.8]{Li7} shows that $\left(W,Y_{W}^{\phi}\right)$
is a $\left(G,\chi\right)$-equivariant $\phi$-coordinated quasi
module for $\left(V,Y[\cdot,z],\mathbf{1}\right)$. Conversely, let
$\left(W,Y_{W}^{\phi}\right)$ be a $\left(G,\chi\right)$-equivariant
$\phi$-coordinated quasi module for $\left(V,Y[\cdot,z],\mathbf{1}\right)$.
Then we have $Y_{W}\left(v,z\right)\in\mathrm{Hom}\left(W,W\left(\left(z\right)\right)\right)$
for $v\in V$, $Y_{W}\left({\mathbf{1}},z\right)=Y_{W}^{\phi}\left(z^{-L(0)}{\mathbf{1}},z\right)=Y_{W}^{\phi}\left({\mathbf{1}},z\right)=1_{W}$
and
\begin{align*}
Y_{W}\left(\chi\left(g\right){}^{-L(0)}gv,z\right)=Y_{W}^{\phi}\left(\left(\chi\left(g\right)z\right){}^{-L(0)}gv,z\right)=Y_{W}^{\phi}\left(\left(\chi\left(g\right)z\right){}^{-L(0)}v,\chi\left(g\right)z\right)=Y_{W}\left(v,\chi\left(g\right)z\right),
\end{align*}
for $v\in V$ and $g\in G$. For $u,v\in V$, let $f\left(z\right)$
be a nonzero polynomial in $\mathbb{C}_{\chi\left(G\right)}\left[z\right]$
such that Definition \ref{def:vaphimod} (iii) holds. Then we have
\begin{align*}
f\left(z_{1}/z_{2}\right)Y_{W}\left(u,z_{1}\right)Y_{W}\left(v,z_{2}\right)=f\left(z_{1}/z_{2}\right)Y_{W}^{\phi}\left(z_{1}^{-L(0)}u,z_{1}\right)Y_{W}^{\phi}\left(z_{2}^{-L(0)}v,z_{2}\right)\in\mathrm{Hom}\left(W,W\left(\left(z_{1},z_{2}\right)\right)\right).
\end{align*}
Furthermore, we have $f\left(z_{1}/z_{2}\right)Y_{W}^{\phi}\left(\left(z_{2}+w_{0}\right){}^{-L(0)}u,z_{1}\right)Y_{W}\left(z_{2}^{-L(0)}v,z_{2}\right)
\in\mathrm{Hom}\left(W,W\left(\left(z_{1},z_{2}\right)\right)\left[\left[w_{0}\right]\right]\right)$
and
\begin{align*}
f\left(z_{1}/z_{2}\right)Y_{W}^{\phi}\left(\left(z_{2}+w_{0}\right){}^{-L(0)}u,z_{1}\right)Y_{W}^{\phi}\left(z_{2}^{-L(0)}v,z_{2}\right)|_{z_{1}=z_{2}e^{z_{0}}}=f\left(e^{z_{0}}\right)Y_{W}^{\phi}\left(Y\left[\left(z_{2}+w_{0}\right)^{-L(0)}u,z_{0}\right]z_{2}^{-L(0)}v,z_{2}\right).
\end{align*}
By applying the substitution $z_{0}=\log\left(1+w_{0}/z_{2}\right)$ in the previous equation,
we obtain from the left hand side
\begin{align*}
 & \left(f(z_{1}/z_{2})Y_{W}^{\phi}\left(\left(z_{2}+w_{0}\right){}^{-L(0)}u,z_{1}\right)Y_{W}^{\phi}\left(z_{2}^{-L(0)}v,z_{2}\right)|_{z_{1}=z_{2}e^{z_{0}}}\right)|_{z_{0}=\log(1+w_{0}/z_{2})}\\
= & f\left(z_{1}/z_{2}\right)Y_{W}^{\phi}\left(z_{1}^{-L(0)}u,z_{1}\right)Y_{W}^{\phi}\left(z_{2}^{-L(0)}v,z_{2}\right)|_{z_{1}=z_{2}+w_{0}}\\
= & f\left(z_{1}/z_{2}\right)Y_{W}\left(u,z_{1}\right)Y_{W}\left(v,z_{2}\right)|_{z_{1}=z_{2}+w_{0}},
\end{align*}
while by using the fact $z_{2}^{L(0)}Y\left(u,z_{0}\right)z_{2}^{-L(0)}=Y\left(z_{2}^{L(0)}u,z_{0}z\right)$
(cf. \cite{FHL}), we obtain from the right hand side
\begin{align*}
 & \left(f\left(e^{z_{0}}\right)Y_{W}^{\phi}\left(Y\left[\left(z_{2}+w_{0}\right){}^{-L(0)}u,z_{0}\right]z_{2}^{-L(0)}v,z_{2}\right)\right)|_{z_{0}=\log(1+w_{0}/z_{2})}\\
= & \left(f\left(e^{z_{0}}\right)Y_{W}\left(z_{2}^{L(0)}Y\left(e^{z_{0}L(0)}\left(z_{2}+w_{0}\right){}^{-L(0)}u,e^{z_{0}}-1\right)z_{2}^{-L(0)}v,z_{2}\right)\right)|_{z_{0}=\log\left(1+w_{0}/z_{2}\right)}\\
= & \left(f\left(e^{z_{0}}\right)Y_{W}\left(Y\left(z_{2}^{L(0)}e^{z_{0}L(0)}\left(z_{2}+w_{0}\right){}^{-L(0)}u,\left(e^{z_{0}}-1\right)z_{2}\right)v,z_{2}\right)\right)|_{z_{0}=\log\left(1+w_{0}/z_{2}\right)}\\
= & f\left(z_{2}+w_{0}/z_{2}\right)Y_{W}\left(Y\left(u,w_{0}\right)v,z_{2}\right).
\end{align*}
This proves the proposition.
\end{proof}


\begin{proposition}\label{prop:modisoforvoa} Assume that $V$ is
a vertex operator algebra, $G$ is a group of automorphisms of $V$
and $\chi$ a linear character of $G$. Then the following two categories
are isomorphic:
\begin{enumerate}
\item[(i)] the category of $\left(G,\chi\right)$-equivariant $\phi$-coordinated
quasi $V$-modules $\left(W,Y_{W}^{\phi}\right)$;
\item[(ii)] the category of $\left(G,\chi\right)$-equivariant quasi $V$-modules
$\left(W,Y_{W}\right)$.
\end{enumerate}
Moreover, if $G=\langle\sigma\rangle$ is a cyclic group of finite order $N$, and $\chi\left(\sigma\right)=e^{-\frac{2\pi\sqrt{-1}}{N}}$,
then these two categories are also isomorphic to the following category
\begin{enumerate}
\item[(iii)] the category of (weak) $\sigma$-twisted $V$-modules $\left(W,Y_{W}^{t}\right)$.
\end{enumerate}
And furthermore, if $v\in V$ is a primary vector, i.e., $L(n)v=0$ for
all $n\ge1$, then we have the following identities
\begin{align}
Y_{W}^{\phi}\left(z^{-L(0)}v,z\right)=Y_{W}\left(v,z\right)=Y_{W}^{t}\left(\left(Nz^{N-1}\right){}^{L(0)}v,z^{N}\right).\label{eq:isoonmod1}
\end{align}
\end{proposition}
\begin{proof}
Let $\left(V,Y\left(\cdot,z\right),{\mathbf{1}},\omega\right)$ be
a vertex operator algebra of central charge $\ell$. Set $\tilde{\omega}=\omega-\frac{\ell}{24}{\mathbf{1}}$.
It was proved in \cite{Z} that $\left(V,Y\left[\cdot,z\right],{\mathbf{1}},\tilde{\omega}\right)$
carries a vertex operator algebra structure. And, an explicit
isomorphism $f$ from $\left(V,Y\left(\cdot,z\right),{\mathbf{1}},\omega\right)$
to $V\left(V,Y\left[\cdot,z\right],{\mathbf{1}},\tilde{\omega}\right)$
was constructed therein. This together with Proposition \ref{prop:phimodvsmod}
implies that the $\left(G,\chi\right)$-equivariant $\phi$-coordinated
quasi $V$-modules $\left(W,Y_{W}^{\phi}\right)$ are exactly the
$\left(G,\chi\right)$-equivariant quasi $V$-modules $\left(W,Y_{W}\right)$
with
\begin{align*}
Y_{W}^{\phi}\left(v,z\right)=Y_{W}\left(z^{L(0)}f(v),z\right),\quad\text{for}\ v\in V.
\end{align*}
In particular, if $v$ is a primary vector, then from \cite{Z} we
have $f(v)=v$ and hence $Y_{W}^{\phi}\left(z^{-L(0)}v,z\right)=Y_{W}\left(v,z\right)$.

Finally, if $G=\langle\sigma\rangle$ is a finite cyclic
group of order $N$ and $\chi(\sigma)=e^{-\frac{2\pi\sqrt{-1}}{N}}$, then it was
proved in \cite{Li5} that the $\left(G,\chi\right)$-equivariant
quasi $V$-modules $\left(W,Y_{W}\right)$ are exactly the (weak)
$\sigma$-twisted $V$-modules $\left(W,Y_{W}^{t}\right)$ with
\begin{align*}
Y_{W}\left(v,z\right)=Y_{W}\left(\Phi\left(z\right)v,z^{N}\right),\quad\text{for}\ v\in V,
\end{align*}
where $\Phi\left(z\right)$ is an operator in $\mathrm{Hom}\left(V,V\left(\left(z\right)\right)\right)$
such that $\Phi\left(z\right)v=\left(Nz^{N-1}\right){}^{L(0)}v$ whenever
$v$ is primary. This completes the proof the proposition.
\end{proof}

\subsection{Equivalence of module categories for universal affine vertex
algebras}

In this subsection, we study the equivalence of certain module categories
for universal affine vertex algebras. Let $\mathfrak{b}$ be a (possibly infinite dimensional) Lie algebra
equipped with a nondegenerate invariant symmetric  bilinear form $\left\langle\cdot,\cdot\right\rangle$.
Then we have the affine Lie algebra
\begin{align*}
\widetilde{\mathcal{L}}\left(\mathfrak{b}\right)=\left(\mathbb{C}\left[t,t^{-1}\right]\otimes\mathfrak{b}\right)\oplus\mathbb{C}\text{k}\oplus\mathbb{C}\text{d},
\end{align*}
where $\text{k}$ is a central element and
\begin{align}
\left[t^{m}\otimes a,t^{n}\otimes b\right]=t^{m+n}\otimes\left[a,b\right]+\left\langle a,b\right\rangle \delta_{m+n,0}\text{k},\quad\left[\text{d},t^{m}\otimes a\right]=mt^{m}\otimes a,\label{eq:affrela}
\end{align}
for $m,n\in\mathbb{Z}$ and $a,b\in\mathfrak{b}$. Equip $\widetilde{\mathcal{L}}\left(\mathfrak{b}\right)$
with a $\mathbb{Z}$-grading structure with respect to the adjoint
action of $-\text{d}$, and form the following $\mathbb{Z}$-graded Lie subalgebra
of $\widetilde{\mathcal{L}}\left(\mathfrak{b}\right)$:
\[
\widehat{\mathcal{L}}\left(\mathfrak{b}\right)=\left(\mathbb{C}\left[t,t^{-1}\right]\otimes\mathfrak{b}\right)\oplus\mathbb{C}\text{k}.
\]
Let $\ell$ be a fixed complex number. View $\mathbb{C}$ as a $\left(\mathbb{C}\left[t\right]\otimes\mathfrak{b}\oplus\mathbb{C}\text{k}\right)$-module
with $\mathbb{C}\left[t\right]\otimes\mathfrak{b}$ acting trivially
and $\text{k}$ acting as scalar $\ell$. Then we have the induced
$\widehat{\mathcal{L}}(\mathfrak{b})$-module
\begin{eqnarray}
V_{\widehat{\mathcal{L}}\left(\mathfrak{b}\right)}\left(\ell,0\right)=\mathcal{U}\left(\widehat{\mathcal{L}}\left(\mathfrak{b}\right)\right)\otimes_{\mathcal{U}\left(\mathbb{C}[t]\otimes\mathfrak{b}\oplus\mathbb{C}\text{k}\right)}\mathbb{C},\label{vwhfsl}
\end{eqnarray}
which is naturally $\mathbb{N}$-graded by defining $\deg\mathbb{C}=0$.
Set ${\bf 1}=1\otimes1\in V_{\widehat{\mathcal{L}}\left(\mathfrak{b}\right)}\left(\ell,0\right)$
and identify $\mathfrak{b}$ as the degree-one subspace of $V_{\widehat{\mathcal{L}}(\mathfrak{b})}\left(\ell,0\right)$
through the linear map
\[
a\mapsto\left(t^{-1}\otimes a\right){\bf 1}\in V_{\widehat{\mathcal{L}}\left(\mathfrak{b}\right)}\left(\ell,0\right),
\]
for $ a\in\mathfrak{b}$. It is known (cf. \cite{FZ,Li1}) that there exists a vertex algebra
structure on $V_{\widehat{\mathcal{L}}(\mathfrak{b})}\left(\ell,0\right)$,
which is uniquely determined by the condition that ${\bf 1}$ is the
vacuum vector and
\[
Y\left(a,z\right)=a\left(z\right)=\sum_{m\in\Z}\left(t^{m}\otimes a\right)z^{-m-1},
\]
for $a\in\mathfrak{b}$.
The vertex algebra $V_{\widehat{\mathcal{L}}(\mathfrak{b})}\left(\ell,0\right)$
is often called \emph{the universal affine vertex algebra} associated
to $\mathfrak{b}$. Denote by $J_{\widehat{\mathcal{L}}(\mathfrak{b})}\left(\ell,0\right)$
the unique maximal graded $\widehat{\mathcal{L}}\left(\mathfrak{b}\right)$-submodule
of $V_{\widehat{\mathcal{L}}\left(\mathfrak{b}\right)}\left(\ell,0\right)$.
Then $J_{\widehat{\mathcal{L}}\left(\mathfrak{b}\right)}\left(\ell,0\right)$
is an ideal of the vertex algebra $V_{\widehat{\mathcal{L}}\left(\mathfrak{b}\right)}\left(\ell,0\right)$.
Define
\begin{align*}
L_{\widehat{\mathcal{L}}\left(\mathfrak{b}\right)}\left(\ell,0\right)=V_{\widehat{\mathcal{L}}\left(\mathfrak{b}\right)}\left(\ell,0\right)/J_{\widehat{\mathcal{L}}\left(\mathfrak{b}\right)}\left(\ell,0\right),
\end{align*}
which is a simple $\Z$-graded vertex algebra.

Assume that $G$ is an automorphism group of $\mathfrak{b}$ preserving
the bilinear form. It is easy to see that $G$ can be uniquely lifted
to an automorphism group of $V_{\widehat{\mathcal{L}}\left(\mathfrak{b}\right)}\left(\ell,0\right)$
that preserving the $\mathbb{Z}$-grading. Moreover, as $J_{\widehat{\mathcal{L}}\left(\mathfrak{b}\right)}\left(\ell,0\right)$
is $G$-stable, $G$ also induces an automorphism group of $L_{\widehat{\mathcal{L}}\left(\mathfrak{b}\right)}\left(\ell,0\right)$.
Assume further that for $a,b\in\mathfrak{b}$,
\begin{align}
\left[ga,b\right]=0, \quad\left\langle ga,b\right\rangle =0\label{conconvar}
\end{align}
for all but finitely many $ g\in G$.
For any linear character $\chi:G\rightarrow\C^{\ast}$, we define a quotient
space
\begin{align}
\widehat{\mathcal{L}}\left(\mathfrak{b},G\right)=\widehat{\mathcal{L}}\left(\mathfrak{b}\right)/{\text{Span}}\left\{ \chi\left(g\right){}^{m}\left(t^{m}\otimes ga\right)-t^{m}\otimes a\mid g\in G,m\in\mathbb{Z},a\in\mathfrak{b}\right\} .\label{eq:whlgb}
\end{align}
It was proved in \cite{Li4} that $\widehat{\mathcal{L}}\left(\mathfrak{b},G\right)$
carries a Lie algebra structure with Lie bracket given by
\begin{align*}
\left[\overline{t^{m}\otimes a},\overline{t^{n}\otimes b}\right]=\sum_{g\in G}\chi\left(g\right){}^{m}\left(\overline{t^{m+n}\otimes\left[ga,b\right]}+\delta_{m+n,0}\left\langle ga,b\right\rangle \bar{\text{k}}\right)
\end{align*}
for $a,b\in\mathfrak{b}$ and $m,n\in\mathbb{Z}$, where $\overline{t^{m}\otimes a}$
and $\overline{\text{k}}$ stand for the images of $t^{m}\otimes a$ and
$\text{k}$ in $\widehat{\mathcal{L}}\left(\mathfrak{b},G\right)$
respectively.

We say that an $\widehat{\mathcal{L}}\left(\mathfrak{b},G\right)$-module
$W$ is\emph{ restricted} if for any $a\in\mathfrak{b}$ and $w\in W$,
$\overline{t^{n}\otimes a}\cdot w=0$ for $n$ sufficiently large,
and is \emph{of level $\ell$} if $\overline{\text{k}}$ acts as the
scalar $\ell$. From \cite[Proposition 6.4]{CLTW} we have the following result.

\begin{proposition}\label{prop:affva1} Let $G,\chi$ be as above.
The following categories are isomorphic to each other:
\begin{enumerate}
\item[(i)] the category of $\left(G,\chi\right)$-equivariant $\phi$-coordinated
quasi $V_{\widehat{\mathcal{L}}\left(\mathfrak{b}\right)}\left(\ell,0\right)$-modules
$\left(W,Y_{W}^{\phi}\right)$;
\item[(ii)] the category of $\left(G,\chi\right)$-equivariant quasi $V_{\widehat{\mathcal{L}}\left(\mathfrak{b}\right)}\left(\ell,0\right)$-modules
$\left(W,Y_{W}\right)$;
\item[(iii)] the category of restricted $\widehat{\mathcal{L}}(\mathfrak{b},G)$-modules
$W$ of level $\ell$.
\end{enumerate}
And, the isomorphisms are determined by
\begin{align*}
Y_{W}^{\phi}\left(a,z\right)=zY_{W}\left(a,z\right)=\sum_{n\in\mathbb{Z}}\left(\overline{t^{n}\otimes a}\right)z^{-n},
\end{align*}
for $ a\in\mathfrak{b}$.
\end{proposition}

We assume now that $\sigma$ is an automorphism of $\mathfrak{b}$ with finite order $N$ and
preserving the bilinear form, and set $\omega_N=e^{\frac{2\pi\sqrt{-1}}{N}}$.  For $a\in\mathfrak{b}$ and $m\in\mathbb{Z}$,
set $a_{\left(m\right)}=\sum_{p=0}^{N-1}\omega_{N}^{-mp}\sigma^{p}(a)$
and $\mathfrak{b}_{\left(m\right)}=\left\{ a_{\left(m\right)}\mid a\in\mathfrak{b}\right\} $
(the $\sigma$-eigenspace of $\mathfrak{b}$ with eigenvalue $\omega_{N}^{m}$). Then we have the following Lie subalgebras of $\widetilde{\mathcal{L}}(\mathfrak{b})$, called $\sigma$-twisted
affine Lie algebras (cf. \cite{K1})
\begin{align*}
\widetilde{\mathcal{L}}\left(\mathfrak{b},\sigma\right)
=\left(\sum_{m\in\mathbb{Z}}\mathbb{C}t^{m}\otimes\mathfrak{b}_{(m)}\right)\oplus\mathbb{C}\text{k}\oplus\mathbb{C}\text{d}, \quad \widehat{\mathcal{L}}\left(\mathfrak{b},\sigma\right)=\left(\sum_{m\in\mathbb{Z}}\mathbb{C}t^{m}\otimes\mathfrak{b}_{(m)}\right)\oplus\mathbb{C}\text{k}.
\end{align*}
Let $\chi_{\omega_{N}}$ be the linear character of the cyclic group
$\langle\sigma\rangle$ such that $\chi_{\omega_{N}}\left(\sigma\right)=\omega_{N}^{-1}$.
We take $G=\left\langle \sigma\right\rangle $ and $\chi=\chi_{\omega_{N}}$ in \eqref{eq:whlgb}. Then  it is easy to  check that the linear map
given by
\begin{align}
\overline{t^{m}\otimes a}\mapsto t^{m}\otimes a_{(m)},\quad\overline{\text{k}}\mapsto N\text{k},\quad\text{for}\ a\in\mathfrak{b},m\in\mathbb{Z},\label{eq:twistedaffiso1}
\end{align}
is a Lie algebra isomorphism from $\widehat{\mathcal{L}}\left(\mathfrak{b},\left\langle \sigma\right\rangle \right)$
to $\widehat{\mathcal{L}}\left(\mathfrak{b},\sigma\right)$.

\begin{definition} Let $W$ be a module of $\widetilde{\mathcal{L}}\left(\mathfrak{b},\sigma\right)$
or $\widehat{\mathcal{L}}\left(\mathfrak{b},\sigma\right)$.
We say that $W$ is
\emph{ restricted} if for any $a\in\mathfrak{b}$ and $w\in W$,
$\left(t^n\otimes a_{(n)}\right)\cdot w=0$ for $n$ sufficiently large,
and is \emph{of level $\ell$} if $\text{k}$ acts as
scalar $\ell/N$.
\end{definition}

 From Proposition \ref{prop:affva1} and the isomorphism
\eqref{eq:twistedaffiso1}, one immediately obtains the following result.

\begin{proposition}\label{prop:affva2} The following categories
are isomorphic to each other:
\begin{enumerate}
\item[(i)] the category of $\left(\langle\sigma\rangle,\chi_{\omega_{N}}\right)$-equivariant
$\phi$-coordinated quasi $V_{\widehat{\mathcal{L}}\left(\mathfrak{b}\right)}\left(\ell,0\right)$-modules
$\left(W,Y_{W}^{\phi}\right)$;
\item[(ii)] the category of $\left(\langle\sigma\rangle,\chi_{\omega_{N}}\right)$-equivariant
quasi $V_{\widehat{\mathcal{L}}\left(\mathfrak{b}\right)}\left(\ell,0\right)$-modules
$\left(W,Y_{W}\right)$;
\item[(iii)] the category of restricted $\widehat{\mathcal{L}}\left(\mathfrak{b},\sigma\right)$-modules
$W$ of level $\ell$.
\end{enumerate}
And, the isomorphisms are determined by
\begin{align*}
Y_{W}^{\phi}\left(a,z\right)=zY_{W}\left(a,z\right)=a\left[z\right]:=\sum_{n\in\mathbb{Z}}\left(t^{n}\otimes a_{(m)}\right)z^{-n},
\end{align*}
for $ a\in\mathfrak{b}$.
\end{proposition}

Furthermore, by Proposition \ref{prop:affva2}, Lemma
\ref{lem:actofd} and the fact that $\left[\text{d},a\left[z\right]\right]=-z\frac{d}{dz}a\left[z\right]$
for $a\in\mathfrak{b}$,  we have

\begin{proposition} \label{prop:affva3} For any complex number $\ell$,
the restricted $\widetilde{\mathcal{L}}\left(\mathfrak{b},\sigma\right)$-modules
$W$ of level $\ell$ are exactly the $\left(\langle\sigma\rangle,\chi_{\omega_{N}}\right)$-equivariant
$\phi$-coordinated quasi $V_{\widehat{\mathcal{L}}\left(\mathfrak{b}\right)}\left(\ell,0\right)$-modules
$\left(W,Y_{W}^{\phi},d\right)$ with
\begin{align*}
\text{d}=-d, \quad a\left[z\right]=Y_{W}^{\phi}\left(a,z\right),
\end{align*}
for $a\in\mathfrak{b}$.
\end{proposition}

\begin{remark}\label{rem:tmodforaffva} We also have the following
variant of the twisted affine Lie algebra $\widehat{\mathcal{L}}\left(\mathfrak{b},\sigma\right)$
(cf.\,\cite{FLM}):
\begin{align*}
\widehat{\mathcal{L}}\left[\mathfrak{b},\sigma\right]=\left(\sum_{m\in\frac{1}{N}\mathbb{Z}}\mathbb{C}t^{m}\otimes\mathfrak{b}_{(Nm)}\right)\oplus\mathbb{C}\text{k},
\end{align*}
where the Lie bracket is given by \eqref{eq:affrela} with $m,n\in\frac{1}{N}\mathbb{Z}$.
Note that $\widehat{\mathcal{L}}\left(\mathfrak{b},\sigma\right)$
is isomorphic to $\widehat{\mathcal{L}}\left[\mathfrak{b},\sigma\right]$
with the mapping
\begin{align}
t^{m}\otimes a_{(m)}\mapsto t^{m/N}\otimes a_{(m)},\quad N\text{k}\mapsto\text{k},\label{eq:twistedaffiso2}
\end{align}
for $ a\in\mathfrak{b}, m\in\mathbb{Z}.$
It is well-known that the $\sigma$-twisted $V_{\widehat{\mathcal{L}}\left(\mathfrak{b}\right)}\left(\ell,0\right)$-modules
$\left(W,Y_{W}^{t}\right)$ are exactly the restricted $\widehat{\mathcal{L}}\left[\mathfrak{b},\sigma\right]$-modules
$W$ of level $\ell$ with $Y_{W}^{t}\left(a,z\right)=\sum_{m\in\mathbb{Z}}\left(t^{\frac{m}{N}}\otimes\frac{a_{(m)}}{N}\right)z^{-\frac{m}{N}-1}$
for $a\in\mathfrak{b}$ (cf. \cite{Li2}).
\end{remark}

\begin{remark}\label{rem:affineder} One can also define the notion
of a $\sigma$-twisted $V_{\widehat{\mathcal{L}}\left(\mathfrak{b}\right)}\left(\ell,0\right)$-module
$\left(W,Y_{W}^{t},d\right)$ such that $\left[d,Y_{W}^{t}\left(v,z\right)\right]=Y_{W}^{t}\left(\mathcal{D}v,z\right)$
for $v\in V_{\widehat{\mathcal{L}}\left(\mathfrak{b}\right)}\left(\ell,0\right)$
(see \cite{Li2} for example).\textcolor{red}{{} }Then such $\sigma$-twisted
$V_{\widehat{\mathcal{L}}\left(\mathfrak{b}\right)}\left(\ell,0\right)$-modules
are precisely the restricted modules of level $\ell$ for the algebra
$\widehat{\mathcal{L}}\left[\mathfrak{b},\sigma\right]\rtimes\mathbb{C}\left(\frac{d}{dt}\otimes1\right)$,
as $[d,Y_{W}^{t}(v,z)]=\frac{d}{dz}Y_{W}^{t}(v,z)$ for $v\in V_{\widehat{\mathcal{L}}\left(\mathfrak{b}\right)}\left(\ell,0\right)$
(cf. \cite{Li2}). 
And a similar result also holds for the equivariant quasi $V_{\widehat{\mathcal{L}}\left(\mathfrak{b}\right)}\left(\ell,0\right)$-modules\textcolor{red}{{}
}(cf. \cite{Li3}). \end{remark}

\subsection{Associating affine Kac-Moody algebras with vertex algebras}

In this subsection, we associate the nullity 1 EALAs (i.e. affine Kac-Moody algebras with
derivations) to vertex algebras via equivariant $\phi$-coordinated
quasi modules. As a by-product, we obtain a characterization of integrable restricted modules for
 affine Kac-Moody algebras, which will be used later to associate the nullity 2 EALAs with vertex algebras.

Let $\dot{\mathfrak{g}}$ be a finite-dimensional simple Lie algebra over the complex field, $\dot{\mathfrak{h}}$ a Cartan subalgebra
of $\dot{\mathfrak{g}}$, and $\dot{\Delta}$ the root system of $\dot{\mathfrak{g}}$
with respect to $\dot{\mathfrak{h}}$. Let $\left\langle \cdot,\cdot\right\rangle $
be the nondegenerate invariant symmetric bilinear form on $\dot{\mathfrak{g}}$
which is normalized so that the square length of long roots are equal
to 2. For $\dot{\alpha}\in\dot{\Delta}$, choose nonzero root vectors
$x_{\dot{\alpha}}\in\dot{\mathfrak{g}}_{\dot{\alpha}}$ such that
$\left\{ x_{\dot{\alpha}},\dot{\alpha}^{\vee},x_{-\dot{\alpha}}\right\} $
form an $\mathfrak{sl}_{2}$-triple, where $\dot{\alpha}^{\vee}\in\dot{\mathfrak{h}}$
denotes the coroot of $\alpha$.

Fix a simple root system $\dot{\Pi}=\left\{ \dot{\alpha}_{1},\dots,\dot{\alpha}_{l}\right\} $
of $\dot{\Delta}$, where $l=\dim\dot{\mathfrak{h}}$ is the rank
of $\dot{\mathfrak{g}}$. Let $\dot{\nu}$ be a diagram automorphism
of $\dot{\mathfrak{g}}$ of order $N$ ($N=1,2$ or $3$). By definition,
there exists a permutation $\dot{\nu}$ on the set $\dot{I}=\left\{ 1,2,\dots,l\right\} $
such that $\dot{\nu}\left(x_{\pm\dot{\alpha}_{i}}\right)=x_{\pm\dot{\alpha}_{\dot{\nu}(i)}}$
for $i\in\dot{I}$. Then the Lie algebra $\widetilde{\mathcal{L}}\left(\dot{\mathfrak{g}},\dot{\nu}\right)$
as defined in Section 3.2 is a Kac-Moody algebra of affine type, and
any affine Kac-Moody algebra has such a form \cite{K1}. We say that
an $\widetilde{\mathcal{L}}\left(\dot{\mathfrak{g}},\dot{\nu}\right)$-module
(or an $\widehat{\mathcal{L}}\left(\dot{\mathfrak{g}},\dot{\nu}\right)$-module)
is \emph{integrable} if all real root vectors $t^{m}\otimes x_{\dot{\alpha}(m)}$
for $\dot{\alpha}\in\dot{\Delta},m\in\mathbb{Z}$ act locally nilpotent
on it \cite{K1}. Now we give the main result of this section.

\begin{theorem}\label{thm:main1} Let $\ell$ be any complex number.
For any restricted $\widetilde{\mathcal{L}}\left(\dot{\mathfrak{g}},\dot{\nu}\right)$-module
$W$ of level $\ell$, there is a $\left(\langle\dot{\nu}\rangle,\chi_{\omega_{N}}\right)$-equivariant
$\phi$-coordinated quasi $V_{\widehat{\mathcal{L}}\left(\dot{\mathfrak{g}}\right)}\left(\ell,0\right)$-module
structure $(Y_{W}^{\phi},d)$ on $W$, which is uniquely determined
by
\begin{align*}
d=-\text{d}, \quad Y_{W}^{\phi}\left(a,z\right)=a\left[z\right]\left(=\sum_{n\in\mathbb{Z}}\left(t^{m}\otimes a_{(m)}\right)z^{-m}\right),
\end{align*}
for $ a\in\dot{\mathfrak{g}}$.
On the other hand, for any $\left(\langle\dot{\nu}\rangle,\chi_{\omega_{N}}\right)$-equivariant
$\phi$-coordinated quasi $V_{\widehat{\mathcal{L}}\left(\dot{\mathfrak{g}}\right)}\left(\ell,0\right)$-module
$(W,Y_{W}^{\phi},d)$, $W$ is a restricted $\widetilde{\mathcal{L}}\left(\dot{\mathfrak{g}},\dot{\nu}\right)$-module
of level $\ell$ with action given by
\begin{align*}
\text{d}=-d, \quad a\left[z\right]=Y_{W}^{\phi}\left(a,z\right),
\end{align*}
for $ a\in\dot{\mathfrak{g}}$.
Furthermore, when $\ell$ is a nonnegative integer, the integrable\textcolor{red}{{}
}restricted $\widetilde{\mathcal{L}}\left(\dot{\mathfrak{g}},\dot{\nu}\right)$-modules
of level $\ell$ exactly correspond to the $\left(\langle\dot{\nu}\rangle,\chi_{\omega_{N}}\right)$-equivariant
$\phi$-coordinated quasi $L_{\widehat{\mathcal{L}}\left(\dot{\mathfrak{g}}\right)}\left(\ell,0\right)$-modules.
\end{theorem}
\begin{proof}
The one-to-one correspondence between the category of restricted $\widetilde{\mathcal{L}}\left(\dot{\mathfrak{g}},\dot{\nu}\right)$-modules of level $\ell$ and the category of $\left(\langle\dot{\nu}\rangle,\chi_{\omega_{N}}\right)$-equivariant
$\phi$-coordinated quasi $V_{\widehat{\mathcal{L}}\left(\dot{\mathfrak{g}}\right)}\left(\ell,0\right)$-modules follows from
 Proposition \ref{prop:affva3}.

For the last assertion of the theorem, we note that $L_{\widehat{\mathcal{L}}\left(\dot{\mathfrak{g}}\right)}\left(\ell,0\right)$
is a rational vertex operator algebra, and each $a\in\dot{\mathfrak{g}}$
is a primary vector of weight $1$ in $L_{\widehat{\mathcal{L}}\left(\dot{\mathfrak{g}}\right)}\left(\ell,0\right)$.
Thus, by Proposition \ref{prop:modisoforvoa}, it is easy to see that the $\left(\langle\dot{\nu}\rangle,\chi_{\omega_{N}}\right)$-equivariant
$\phi$-coordinated quasi $L_{\widehat{\mathcal{L}}\left(\dot{\mathfrak{g}}\right)}\left(\ell,0\right)$-modules
$\left(W,Y_{W}^{\phi}\right)$ are exactly the $\dot{\nu}$-twisted
$L_{\widehat{\mathcal{L}}(\dot{\mathfrak{g}})}\left(\ell,0\right)$-modules
$\left(W,Y_{W}^{t}\right)$ with
\begin{align*}
Y_{W}^{\phi}\left(a,z\right)=Nz^{N}Y_{W}^{t}\left(a,z\right),
\end{align*}
for $ a\in\dot{\mathfrak{g}}$.
Recall the algebra $\widehat{\mathcal{L}}\left[\dot{\mathfrak{g}},\dot{\nu}\right]$
defined in Remark \ref{rem:tmodforaffva}. It is known (cf. \cite{Li2}) that the
integrable restricted $\widehat{\mathcal{L}}\left[\dot{\mathfrak{g}},\dot{\nu}\right]$-modules
of level $\ell$ are exactly the $\dot{\nu}$-twisted $L_{\widehat{\mathcal{L}}\left(\dot{\mathfrak{g}}\right)}\left(\ell,0\right)$-modules
with
\begin{align*}
Y_{W}^{t}\left(a,z\right)=\sum_{m\in\mathbb{Z}}\left(t^{\frac{m}{N}}\otimes\frac{a_{(m)}}{N}\right)z^{-\frac{m}{N}-1},
\end{align*}
for $ a\in\dot{\mathfrak{g}}$.
Via the isomorphism \eqref{eq:twistedaffiso2}, we obtain that the
integrable restricted $\widehat{\mathcal{L}}\left(\dot{\mathfrak{g}},\dot{\nu}\right)$-modules
of level $\ell$ are exactly the $\left(\langle\dot{\nu}\rangle,\chi_{\omega_{N}}\right)$-equivariant
$\phi$-coordinated quasi $L_{\widehat{\mathcal{L}}\left(\dot{\mathfrak{g}}\right)}\left(\ell,0\right)$-modules
$\left(W,Y_{W}^{\phi}\right)$ with
\begin{align*}
Y_{W}^{\phi}\left(a,z\right)=Nz^{N}Y_{W}^{t}\left(a,z\right)=\sum_{m\in\mathbb{Z}}\left(t^{\frac{m}{N}}\otimes a_{(m)}\right)z^{-m}=\sum_{m\in\mathbb{Z}}\left(t^{m}\otimes a_{(m)}\right)z^{-m}=a\left[z\right],
\end{align*}
for $ a\in\dot{\mathfrak{g}}$. Thus, the last assertion of the theorem follows from this and Lemma \ref{lem:actofd}.
\end{proof}
We write $\dot{\Delta}=\dot{\Delta}_{1}\cup\dot{\Delta}_{2}$, where $\dot{\Delta}_{2}=\left\{ \dot{\alpha}\in\dot{\Delta}\mid\langle\dot{\alpha},\dot{\nu}\left(\dot{\alpha}\right)\rangle=-1\right\}$, and $\dot{\Delta}_{1}=\dot{\Delta}\setminus\dot{\Delta}_{2}$. For each $\dot{\alpha}\in\dot{\Delta}$, we set $\epsilon_{\dot{\alpha}}=\frac{2}{\langle\dot{\alpha},\dot{\alpha}\rangle}\ (=1,2\, \text{or}\, 3)$,
and $
p_{\dot{\alpha}}\left(z\right)=\frac{1-z^s}{1-z}$ if $\dot{\alpha}\in\dot{\Delta}_{s}$ for $s=1,2.$

Note that if $\dot{\alpha}\in\dot{\Delta}_{2}$ (resp.\ $\dot{\Delta}_{1}$),
then the Dynkin diagram associated to the $\dot{\nu}$-orbit $\left\{ \dot{\nu}^{p}\left(\dot{\alpha}\right)\mid p=0,\dots,N-1\right\} $
of $\dot{\alpha}$ is of type $A_{2}$ (resp.\, a direct sum of type
$A_{1}$). Using this, one can check that
\begin{align}
p_{\dot{\alpha}}\left(z/w\right)\cdot\left[x_{\dot{\alpha}}\left[z\right],x_{\dot{\alpha}}\left[w\right]\right]=0, \label{eq:charintaffmod1}
\end{align}
for $ \dot{\alpha}\in\dot{\Delta}$.
As a by-product of Theorem \ref{thm:main1}, we obtain a
characterization of the integrable restricted $\widetilde{\mathcal{L}}\left(\dot{\mathfrak{g}},\dot{\nu}\right)$-modules.
We note that when $\dot{\nu}=\mathrm{Id}$, such a characterization is known (cf. \cite[Proposition 5.2.3]{Li1}).
\begin{proposition}\label{prop:charintaffmod} Let $W$ be a restricted
$\widetilde{\mathcal{L}}\left(\dot{\mathfrak{g}},\dot{\nu}\right)$-module
of level $\ell$. Then $W$ is integrable if and only if $\ell$
is a nonnegative integer, and for each $\dot{\alpha}\in\dot{\Delta}$,
\begin{align}
\left(\prod_{1\le i<j\le\epsilon_{\dot{\alpha}}\ell+1}p_{\dot{\alpha}}\left(z_{i}/z_{j}\right)\right)x_{\dot{\alpha}}\left[z_{1}\right]x_{\dot{\alpha}}\left[z_{2}\right]\cdots x_{\dot{\alpha}}\left[z_{\epsilon_{\dot{\alpha}}\ell+1}\right]|_{z_{1}=z_{2}\cdots=z_{\epsilon_{\dot{\alpha}}\ell+1}}=0\quad\text{on}\ W.\label{eq:charintaffmod2}
\end{align}
\end{proposition}
\begin{proof}
Note that if $W$ is integrable, then $\ell$ is a nonnegative integer
and the ideal $J_{\widehat{\mathcal{L}}(\mathfrak{b})}\left(\ell,0\right)$
is generated by the elements $\left(\left(x_{\dot{\alpha}}\right){}_{-1}\right){}^{\epsilon_{\dot{\alpha}}\ell+1}{\mathbf{1}}$
for $\dot{\alpha}\in\dot{\Delta}$ \cite{K1}. Let $W$ be a restricted
$\widetilde{\mathcal{L}}\left(\dot{\mathfrak{g}},\dot{\nu}\right)$-module
of level $\ell$, and hence by Theorem \ref{thm:main1} there is a $\left(\langle\dot{\nu}\rangle,\chi_{\omega_{N}}\right)$-equivariant
$\phi$-coordinated quasi $V_{\widehat{\mathcal{L}}(\dot{\mathfrak{g}})}\left(\ell,0\right)$-module
$(W,Y_{W}^{\phi},d)$. Furthermore, if $W$ is also integrable,
then $(W,Y_{W}^{\phi},d)$   becomes  a $\left(\langle\dot{\nu}\rangle,\chi_{\omega_{N}}\right)$-equivariant
$\phi$-coordinated quasi $L_{\widehat{\mathcal{L}}\left(\dot{\mathfrak{g}}\right)}\left(\ell,0\right)$-module.
This then implies that $\left(\left(x_{\dot{\alpha}}\right){}_{-1}\right){}^{\epsilon_{\dot{\alpha}}\ell+1}{\mathbf{1}}=0$
on $W$. Note that $\left(x_{\dot{\alpha}}\right){}_{n}\left(x_{\dot{\alpha}}\right)=0$
for $n\ge0$ in $V_{\widehat{\mathcal{L}}(\dot{\mathfrak{g}})}\left(\ell,0\right)$. Thus by Proposition \ref{prop:modnil} and \eqref{eq:charintaffmod1},
we see that \eqref{eq:charintaffmod2} holds.

Conversely, if \eqref{eq:charintaffmod2}
holds, again by Proposition \ref{prop:modnil} and \eqref{eq:charintaffmod1},
we see that $\left(\left(x_{\dot{\alpha}}\right){}_{-1}\right){}^{\epsilon_{\dot{\alpha}}\ell+1}{\mathbf{1}}=0$
on $W$, where $W$ is now viewed as a faithful $\left(\langle\dot{\nu}\rangle,\chi_{\omega_{N}}\right)$-equivariant
$\phi$-coordinated quasi module for $V_{\widehat{\mathcal{L}}(\dot{\mathfrak{g}})}\left(\ell,0\right)/\ker Y_{W}^{\phi}$. This
implies that $J_{\widehat{\mathcal{L}}\left(\mathfrak{b}\right)}\left(\ell,0\right)\subset\ker Y_{W}^{\phi}$,
and hence $(W,Y_{W}^{\phi},d)$ becomes a $\left(\langle\dot{\nu}\rangle,\chi_{\omega_{N}}\right)$-equivariant
$\phi$-coordinated quasi $L_{\widehat{\mathcal{L}}(\dot{\mathfrak{g}})}\left(\ell,0\right)$-module.
Finally, the integrability of $W$ follows from Theorem \ref{thm:main1}.
\end{proof}

\section{Twisted toroidal extended affine Lie algebras}

In this section, we recall the construction of the nullity $2$ twisted toroidal extended
affine Lie algebras arising from diagram automorphisms of affine Kac-Moody
algebras.

\subsection{Toroidal EALA $\widetilde{\mathfrak{g}}$ }

We first recall the definition of nullity $2$ toroidal extended affine
Lie algebras in this subsection \cite{BGK,B2}. Let $\mathcal{R}=\mathbb{C}\left[t_{0},t_{0}^{-1},t_{1},t_{1}^{-1}\right]$
be Laurent polynomial ring in commuting variables $t_{0}$ and $t_{1}.$
Denote by $\Omega_{\mathcal{R}}^{1}=\mathcal{R}dt_{0}\oplus\mathcal{R}dt_{1}$
the space of K$\ddot{a}$hler differentials on $\mathcal{R}.$ Define
the $1$-forms
\[
\text{k}_{0}=t_{0}^{-1}dt_{0,}\qquad\text{k}_{1}=t_{1}^{-1}dt_{1}.
\]
Then $\left\{ \text{k}_{0},\text{k}_{1}\right\} $ form a $\mathcal{R}$-basis
of $\Omega_{\mathcal{R}}^{1}.$ Let
\[
d\left(\mathcal{R}\right)=\left\{ df=\frac{\partial f}{\partial t_{0}}dt_{0}+\frac{\partial f}{\partial t_{1}}dt_{1}\mid f\in\mathcal{R}\right\}
\]
be the space of exact $1$-forms in $\Omega_{\mathcal{R}}^{1}$, and
set
\[
\mathcal{K}=\Omega_{\mathcal{R}}^{1}/d\left(\mathcal{R}\right).
\]
and
\begin{equation}
\text{k}_{m_{0},m_{1}}=\begin{cases}
\frac{1}{m_{1}}t_{0}^{m_{0}}t_{1}^{m_{1}}\text{k}_{0}\  & \text{if}\ m_{1}\ne0,\\
-\frac{1}{m_{0}}t_{0}^{m_{0}}\text{k}_{1}\  & \text{if}\ m_{0}\ne0,m_{1}=0,\\
0\  & \text{if}\ m_{0}\ne0,m_{1}=0.
\end{cases}\label{eq:notionink}
\end{equation}
for $m_{0},m_{1}\in\mathbb{Z}$. Then the set
\begin{align}
\mathbb{B}_{\mathcal{K}}=\left\{ \text{k}_{0},\text{k}_{1}\right\} \cup\left\{ \text{k}_{m_{0},m_{1}}\mid\left(m_{0},m_{1}\right)\in\mathbb{Z}^{2}\setminus\left\{ \left(0,0\right)\right\} \right\}
\end{align}
forms a basis of $\mathcal{K}$. The set $\mathbb{B}_{\mathcal{K}} $ can also be expressed as follows
\begin{equation}
\mathbb{B}_{\mathcal{K}}=\left\{ \text{k}_{0}\right\} \cup\left\{ t_{0}^{m_{0}}\text{k}_{1},\text{k}_{m_{0},m_{1}}\mid m_{0}\in\mathbb{Z},m_{1}\in\mathbb{Z}^{\ast}\right\}, \label{eq:basisK}
\end{equation}
which will be used later on.
Recall that $\dot{\mathfrak{g}}$ is a finite-dimensional simple Lie
algebra and $\langle\cdot,\cdot\rangle$ is the normalized bilinear form on
$\dot{\mathfrak{g}}$. Let
\[
\mathfrak{t}\left(\dot{\mathfrak{g}}\right)=\left(\mathcal{R}\otimes\dot{\mathfrak{g}}\right)\oplus\mathcal{K}
\]
be a central extension of the double-loop algebra $\mathcal{R}\otimes\dot{\mathfrak{g}}$
by $\mathcal{K}$ with Lie product
\begin{equation}
\left[t_{0}^{m_{0}}t_{1}^{m_{1}}\otimes x,t_{0}^{n_{0}}t_{1}^{n_{1}}\otimes y\right]=t_{0}^{m_{0}+n_{0}}t_{1}^{m_{1}+n_{1}}\otimes\left[x,y\right]+\left\langle x,y\right\rangle \sum_{i=0,1}m_{i}t_{0}^{m_{0}+n_{0}}t_{1}^{m_{1}+n_{1}}\text{k}_{i}, \label{eq:relationint1}
\end{equation}
for $x,y\in\dot{\mathfrak{g}}$ and $m_{0},m_{1},n_{0},n_{1}\in\mathbb{Z}.$
It is proved in \cite{MRY} that the Lie algebra $\mathfrak{t}\left(\dot{\mathfrak{g}}\right)$ is
the universal central extension of $\mathcal{R}\otimes\dot{\mathfrak{g}}$,
and it is often called \emph{the nullity 2 toroidal Lie algebra}.

Let
\[
\text{Der}\left(\mathcal{R}\right)=\mathcal{R}\frac{\partial}{\partial t_{0}}\oplus\mathcal{R}\frac{\partial}{\partial t_{1}}
\]
be the space of derivations of $\mathcal{R}.$ Set
\[
\text{d}_{0}=t_{0}\frac{\partial}{\partial t_{0}},\quad\text{d}_{1}=t_{1}\frac{\partial}{\partial t_{1}}.
\]
Then $\left\{ \text{d}_{0},\text{d}_{1}\right\} $ form a $\mathcal{R}$-basis of
$\text{Der}\left(\mathcal{R}\right)$, and the derivation Lie algebra $\text{Der}\left(\mathcal{R}\right)$ also
acts on $\mathcal{R}\otimes\dot{\mathfrak{g}}$ with

\[
\psi\left(f\otimes x\right)=\psi\left(f\right)\otimes x,
\]
for $ \psi\in\text{Der}\left(\mathcal{R}\right),f\in\mathcal{R},x\in\dot{\mathfrak{g}}$. One notes that
the $\text{Der}\left(\mathcal{R}\right)$-action on $\mathcal{R}\otimes\dot{\mathfrak{g}} $ can be uniquely extended to an action on the center $\mathcal{K} $
of the toroidal Lie algebra
$\mathfrak{t}\left(\dot{\mathfrak{g}}\right)$ with
\[
\psi\left(f\text{d}g\right)=\psi\left(f\right)\text{d}g+f\text{d}\psi\left(g\right),
\]
for $ \psi\in\text{Der}\left(\mathcal{R}\right), g\in\mathcal{R}$. Now, we
form the semi-direct product Lie algebra
\[
\mathcal{T}\left(\dot{\mathfrak{g}}\right)=\mathfrak{t}\left(\dot{\mathfrak{g}}\right)\rtimes\text{Der}\left(\mathcal{R}\right)
=\mathcal{R\otimes\dot{\mathfrak{g}}}\oplus\mathcal{K}\oplus\text{Der}\left(\mathcal{R}\right),
\]
which is often called \emph{the full toroidal Lie algebra} \cite{B1}.
Note that in $\mathcal{T}\left(\dot{\mathfrak{g}}\right)$ we have
\begin{align}
\left[t_{0}^{m_{0}}t_{1}^{m_{1}}\text{d}_{i},t_{0}^{n_{0}}t_{1}^{n_{1}}\otimes x\right] & =n_{i}\left(t_{0}^{m_{0}+n_{0}}t_{1}^{m_{1}+n_{1}}\otimes x\right),\label{eq:relationint2}\\
\left[t_{0}^{m_{0}}t_{1}^{m_{1}}\text{d}_{i},t_{0}^{n_{0}}t_{1}^{n_{1}}\text{k}_{j}\right] & =n_{i}t_{0}^{m_{0}+n_{0}}t_{1}^{m_{1}+n_{1}}\text{k}_{j}+\delta_{i,j}\sum_{r=0,1}m_{r}t_{0}^{m_{0}+n_{0}}t_{1}^{m_{1}+n_{1}}\text{k}_{r},\label{eq:relationint3}\\
\left[t_{0}^{m_{0}}t_{1}^{m_{1}}\text{d}_{i},t_{0}^{n_{0}}t_{1}^{n_{1}}\text{d}_{j}\right] & =n_{i}t_{0}^{m_{0}+n_{0}}t_{1}^{m_{1}+n_{1}}\text{d}_{j}-m_{j}t_{0}^{m_{0}+n_{0}}t_{1}^{m_{1}+n_{1}}\text{d}_{i},\label{eq:relationint4}
\end{align}
for $x\in\mathfrak{g},$ $m_{0},n_{0},m_{1},n_{1}\in\mathbb{Z}$, and
$i,j\in\left\{ 0,1\right\} .$

We define a Lie subalgebra $ \mathcal{S}$ of $\text{Der}\left(\mathcal{R}\right)$ as follows
\[
\mathcal{S}=\left\{ f_{0}\text{d}_{0}+f_{1}\text{d}_{1}\in \text{Der}\left(\mathcal{R}\right)\mid f_{0},f_{1}\in\mathcal{R},\text{d}_{0}\left(f_{0}\right)+\text{d}_{1}\left(f_{1}\right)=0\right\}.
\]
The elements in $ \mathcal{S}$ is often called \emph{skew derivations} over $\mathcal{R}$ (cf. \cite{BGK,N}), and also known
as \emph{divergence-zero derivations} (cf. \cite{B2}).
It is easy to see that, for $ m_{0},m_{1}\in\mathbb{Z}$
\[
\widetilde{\text{d}}_{m_{0},m_{1}}=m_{0}t_{0}^{m_{0}}t_{1}^{m_{1}}\text{d}_{1}-m_{1}t_{0}^{m_{0}}t_{1}^{m_{1}}\text{d}_{0},
\]
 are elements of $\mathcal{S},$ and
the set
\[
\mathbb{B}_{\mathcal{S}}=\left\{ \text{d}_{0},\text{d}_{1}\right\} \cup\left\{ \widetilde{\text{d}}_{m_{0},m_{1}}\mid\left(m_{0},m_{1}\right)\in\mathbb{Z}\times\mathbb{Z}\setminus\left\{ \left(0,0\right)\right\} \right\}
\]
forms a $\mathbb{C}$-basis of $\mathcal{S}$. Note that $\widetilde{\text{d}}_{m,0}=mt_0^{m}\text{d}_{1}$ for $m\in\mathbb{Z}$, we can rewrite $\mathbb{B}_{\mathcal{S}}$ as follows
\begin{align}
\mathbb{B}_{\mathcal{S}}=\left\{ \text{d}_{0}\right\} \cup\left\{ t_{0}^{m_{0}}\text{d}_{1},\widetilde{\text{d}}_{m_{0},m_{1}}\mid m_{0}\in\mathbb{Z},m_{1}\in\mathbb{Z}^{\ast}\right\}, \label{eq:basisS}
\end{align}
and the Lie product relations are given as follows
\begin{align}
 \left[\text{d}_{i},\widetilde{\text{d}}_{m_{0},m_{1}}\right]  =m_{i}\widetilde{\text{d}}_{m_{0},m_{1}},\quad
\left[\widetilde{\text{d}}_{m_{0},m_{1}},\widetilde{\text{d}}_{n_{0},n_{1}}\right]  =\left(m_{0}n_{1}-m_{1}n_{0}\right)\widetilde{\text{d}}_{m_{0}+n_{0},m_{1}+n_{1}}\label{eq:relationinS2}
\end{align}
for $i\in\left\{ 0,1\right\} $, and $m_{0},m_{1},n_{0},n_{1}\in\mathbb{Z}.$

Form the following subalgebra of $\mathcal{T}\left(\dot{\mathfrak{g}}\right)$:
\begin{align}
\widetilde{\mathfrak{g}}=\mathfrak{t}\left(\dot{\mathfrak{g}}\right)\rtimes\mathcal{S}=\mathcal{R}\otimes\dot{\mathfrak{g}}\oplus\mathcal{K}\oplus\mathcal{S},
\end{align}
which is often called \emph{the
nullity 2 toroidal extended affine Lie algebra} \cite{BGK,B2}.

From \eqref{eq:relationint2} and \eqref{eq:relationint3}, one has the following Lie product relations in $\widetilde{\mathfrak{g}}$
\begin{align}
\left[\widetilde{\text{d}}_{m_{0},m_{1}},t_{0}^{n_{0}}t_{1}^{n_{1}}\otimes x\right] & =\left(m_{0}n_{1}-m_{1}n_{0}\right)t_{0}^{m_{0}+n_{0}}t_{1}^{m_{1}+n_{1}}\otimes x,\label{eq:relationinS3}\\
\left[\text{\ensuremath{\widetilde{\text{d}}}}_{m_{0},m_{1}},\text{k}_{n_{0},n_{1}}\right] & =\left(m_{0}n_{1}-m_{1}n_{0}\right)\text{k}_{m_{0}+n_{0},m_{1}+n_{1}}+\delta_{m_{0}+n_{0},0}\delta_{m_{1}+n_{1},0}\left(m_{0}\text{k}_{0}+
m_{1}\text{k}_{0}\right), \label{eq:relationinS4}
\end{align}
for $m_{0},n_{0},m_{1},n_{1}\in\mathbb{Z}$, and $x\in \dot{\mathfrak{g}}$. And the Lie subalgebra
\[
\mathfrak{g}:=\left(\mathbb{C}\left[t_{1},t_{1}^{-1}\right]\otimes\dot{\mathfrak{g}}\right)\oplus\mathbb{C}\text{k}_{1}\oplus\mathbb{C}\text{d}_{1}
\]
of $\widetilde{\mathfrak{g}}$ is isomorphic to the
affine Kac-Moody algebra $\widetilde{\mathcal{L}}\left(\dot{\mathfrak{g}}\right)$.
We extend the normalized bilinear form $\left\langle \cdot,\cdot\right\rangle $
of $\dot{\mathfrak{g}}$ to a nondegenerate invariant symmetric bilinear
form on $\mathfrak{g}$ by defining
\[
\left\langle t_{1}^{m}\otimes x+a\text{k}_{1}+b\text{d}_{1},t_{1}^{n}\otimes y+a'\text{k}_{1}+b'\text{d}_{1}\right\rangle =\delta_{m+n,0}\left\langle x,y\right\rangle +ab'+ba',
\]
where $m,n\in\mathbb{Z},$ $x,y\in\dot{\mathfrak{g}}$, and $a,b,a',b'\in\mathbb{C}.$
It is easy to see from \eqref{eq:basisK} and \eqref{eq:basisS} that
the Lie algebra $\widetilde{\mathfrak{g}}$ is linearly spanned by the
set
\begin{align}
\left\{ t_{0}^{m}u,\ \text{k}_{m,n},\ \widetilde{\text{d}}_{m,n},\ \text{d}_{0},\ \text{k}_{0}\mid u\in\mathfrak{g},\ m\in\mathbb{Z},\ n\in\mathbb{Z}^{\ast}\right\}, \label{eq:basisoftildeg}
\end{align}
with a nondegenerate and invariant symmetric bilinear form $\left\langle \cdot,\cdot\right\rangle $ defined by
\begin{align}
\left\langle t_{0}^{m}x,t_{0}^{-m}y\right\rangle =\left\langle x,y\right\rangle ,\quad\left\langle \text{k}_{0},\text{d}_{0}\right\rangle =\left\langle \widetilde{\text{d}}_{m,n},\text{k}_{-m,-n}\right\rangle =1,\label{eq:defform}
\end{align}
for $x,y\in\mathfrak{g}$, $m\in\mathbb{Z}$, $n\in\mathbb{Z}^{\ast}$, and
 a self-centralizing ad-diagonalizable subalgebra
\[
\widetilde{\mathfrak{h}}=\dot{\mathfrak{h}}\oplus\mathbb{C}\text{k}_{0}\oplus\mathbb{C}\text{k}_{1}\oplus\mathbb{C}\text{d}_{0}\oplus\mathbb{C}\text{d}_{1}.
\]

\subsection{\label{Affine algebra}Diagram automorphisms of affine Kac-Moody algebras }

For any diagram automorphism $\mu$  of the affine Kac-Moody algebra $\mathfrak{g}$, we define an automorphism $\hat{\mu}$ for the full toroidal
Lie algebra $\mathcal{T}\left(\dot{\mathfrak{g}}\right)$, which  will be
used to construct twisted toroidal extended affine Lie algebras
in the next subsection.

Denote by
\[
\mathfrak{h}=\dot{\mathfrak{h}}\oplus\mathbb{C}\text{k}_{1}\oplus\mathbb{C}\text{d}_{1}
\]
the Cartan subalgebra of the affine Kac-Moody algebra $\mathfrak{g},$ and $\mathfrak{h}^{\ast}$ the dual space of $\mathfrak{h}$.
We identify $\dot{\mathfrak{h}}^{\ast}$ as a subspace of $\mathfrak{h}^{\ast}$
such that $\dot{\alpha}\left(\text{k}_{1}\right)=\dot{\alpha}\left(\text{d}_{1}\right)=0$
for $\dot{\alpha}\in\dot{\mathfrak{h}}^{\ast},$ and define the null root
$\delta_{1}\in\mathfrak{h}^{\ast}$ by setting
\[
\delta_{1}\left(\dot{\mathfrak{h}}\right)=\delta\left(\text{k}_{1}\right)=0,\quad\delta_{1}\left(\text{d}_{1}\right)=1.
\]
Then we have the root space decomposition
$
\mathfrak{g}=\mathfrak{h}\oplus\sum_{\alpha\in\Delta}\mathfrak{g}_{\alpha},
$
where
\[
\Delta=\left\{ \dot{\alpha}+m\delta_{1},n\delta_{1}\mid\dot{\alpha}\in{\dot{\Delta}},m\in\mathbb{Z},n\in\mathbb{Z}^{\ast}\right\} .
\]
Denote by
$
\Delta^{\times}=\left\{ \dot{\alpha}+m\delta_{1}\mid\dot{\alpha}\in\dot{\Delta},m\in\mathbb{Z}\right\}
$
the set of real roots in $\Delta.$ Recall that for each $\dot{\alpha}\in\dot{\Delta}$ and $\alpha=\dot{\alpha}+m\delta_{1}\in\Delta^{\times}$,
there are $\mathfrak{sl}_{2}$-triples $\left\{ x_{\dot{\alpha}},\dot{\alpha}^{\vee},x_{-\dot{\alpha}}\right\} $
in $\dot{\mathfrak{g}}$, and $\left\{ x_{\alpha},\alpha^{\vee},x_{-\alpha}\right\} $  in $\mathfrak{g}$, where 
\[
\alpha^{\vee}=\dot{\alpha}^{\vee}+\frac{2m}{\left\langle \dot{\alpha},\dot{\alpha}\right\rangle }\text{k}_{1}\, \quad x_{\alpha}=t_{1}^{m}\otimes x_{\dot{\alpha}}\in\mathfrak{g}_{\alpha}.
\]

Let $\dot{\theta}$ be the highest root in $\dot{{\Delta}}$
with respect to the simple root system $\dot{\Pi}$=$\left\{ \dot{\alpha}_{1},\cdots,\dot{\alpha}_{l}\right\} $
of $\dot{\Delta}$. Then $\Pi=\left\{ \alpha_{0},\cdots,\alpha_{l}\right\} $
forms a simple root system of $\Delta$ with $\alpha_{0}=\delta_{1}-\dot{\theta}$
and $\alpha_{i}=\dot{\alpha}_{i}$ for $1\le i\le l$. Denote by
\[
A=\left(a_{ij}\right)_{i,j\in I}=\left(2\frac{\left\langle \alpha_{i},\alpha_{j}\right\rangle }{\left\langle \alpha_{i},\alpha_{i}\right\rangle }\right)_{i,j\in I}
\]
the generalized Cartan matrix of the affine Kac-Moody algebra $\mathfrak{g},$ where $I=\left\{ 0,1,\cdots,l\right\} .$\textcolor{red}{{}
}Let $\mu$ be an automorphism of the generalized Cartan matrix $A$, which
by definition is a permutation of $I$ such that $a_{ij}=a_{\mu\left(i\right)\mu\left(j\right)}$
for $i,j\in I$. It is known (cf. \cite{KW}, \cite{FSS}) that there is a unique automorphism
of $\mathfrak{g}$, still denoted by $\mu$, such that
\begin{equation}
\mu\left(x_{\pm\alpha_{i}}\right)=x_{\pm\alpha_{\mu\left(i\right)}}, \ \ \left\langle \mu\left(x\right),\mu\left(y\right)\right\rangle =\left\langle x,y\right\rangle \label{diagram automorphism}
\end{equation}
for $i\in I$ and $x,y\in\mathfrak{g}$. Moreover, $\mu$ stabilizes
$\mathfrak{h}$ and has the same order as it is viewed as a permutation
of $I$.

\begin{remark} The automorphism $\mu$ of the affine Kac-Moody algebra $\mathfrak{g}$
determined by (\ref{diagram automorphism}) is called a \emph{diagram automorphism
}associated to an automorphism $\mu$ of the generalized Cartan matrix
$A$.

\end{remark}

From now on, we fix a diagram automorphism $\mu$ of $\mathfrak{g}$
with order $T$. Recall that $\mu\left(\mathfrak{h}\right)=\mathfrak{h}$. Then
there is an action of $\mu$ on $\mathfrak{h}^{\ast}$ defined by
\[
\mu\left(\alpha\right)\left(h\right)=\alpha\left(\mu^{-1}\left(h\right)\right)
\]
for $\alpha\in\mathfrak{h}^{\ast},$ $h\in\mathfrak{h}$. Let $Q=\mathbb{Z}\Pi=\mathbb{Z}\alpha_{0}\oplus\cdots\oplus\mathbb{Z}\alpha_{l}$
be the root lattice of $\mathfrak{g}$, and $\dot{Q}=\mathbb{Z}\dot{\Pi}$
be the root lattice of $\dot{\mathfrak{g}}$. Note that $\mu\left(\alpha_{i}\right)=\alpha_{\mu\left(i\right)}$
for $i\in I$ and hence $\mu(Q)=Q$. Furthermore, it is known that
$\mu\left(\delta_{1}\right)=\delta_{1}$ (cf. \cite{FSS}). Since $Q=\dot{Q}\oplus\mathbb{Z}\delta_{1}$,
for each $\dot{\alpha}\in\dot{Q},$ we can write
\[
\mu\left(\dot{\alpha}\right)=\dot{\mu}\left(\dot{\alpha}\right)+\rho_{\mu}\left(\dot{\alpha}\right)\delta_{1},
\]
where $\dot{\mu}\left(\dot{\alpha}\right)\in\dot{Q}$ {and} $ \rho_{\mu}\left(\dot{\alpha}\right)\in\mathbb{Z}$.
One can easily check that
\begin{alignat*}{1}
\dot{\mu}:\dot{Q}\to\dot{Q},\quad\dot{\alpha}\mapsto\dot{\mu}\left(\dot{\alpha}\right)
\end{alignat*}
defines an automorphism of $\dot{Q}$, and the map
\begin{alignat*}{1}
\rho_{\mu}:\dot{Q}\to\mathbb{Z},\quad\dot{\alpha}\mapsto\rho_\mu\left(\dot{\alpha}\right)
\end{alignat*}
is a homomorphism of abelian groups.

Now identity $\dot{\mathfrak{h}}$ with $\dot{\mathfrak{h}}^{\ast}$
via $\left\langle \cdot,\cdot\right\rangle $ so that the coroot
$
\dot{\alpha}^{\vee}=\frac{2\dot{\alpha}}{\left\langle \dot{\alpha},\dot{\alpha}\right\rangle}$ for $\dot{\alpha}\in\dot{\Delta}.
$
Then by $\mathbb{C}$-linearity we obtain two the linear maps $\dot{\mu}:\dot{\mathfrak{h}}\to\dot{\mathfrak{h}}$
and $\rho_{\mu}:\dot{\mathfrak{h}}\to\mathbb{C}$ such that
\begin{align}
\mu(h)=\dot{\mu}(h)+\rho_{\mu}(h)\text{k}_{1}
\end{align}
for $h\in\dot{\mathfrak{h}}$. Furthermore, we can extend $\dot{\mu}$
to an automorphism of the Lie algebra $\dot{\mathfrak{g}}$ by the
following rule (cf. \cite{CJKT}):
\begin{align}
\mu\left(x_{\dot{\alpha}}\right)=t_{1}^{\rho_{\mu}(\dot{\alpha})}\otimes\dot{\mu}\left(x_{\dot{\alpha}}\right)\in\mathfrak{g}_{\dot{\mu}(\dot{\alpha})+\rho_{\mu}(\dot{\alpha})\delta_{1}}{\color{red}}\label{eq:defdotmu}
\end{align}
for $\dot{\alpha}\in\dot{\Delta}$ (recall that $\mu\left(\dot{\alpha}\right)=\dot{\mu}\left(\dot{\alpha}\right)+\rho_{\mu}\left(\dot{\alpha}\right)\delta_{1}$
). Then we have the following results:

\begin{lemma} \label{lem:muaction} For $\dot{\alpha}\in\dot{\Delta}$
and $n\in\mathbb{Z}$, we have
\begin{align}
\mu\left(t_{1}^{n}\otimes x_{\dot{\alpha}}\right) & =t_{1}^{\rho_{\mu}\left(\dot{\alpha}\right)+n}\otimes\dot{\mu}\left(x_{\dot{\alpha}}\right),
\quad \mu\left(\text{k}_{1}\right)  =\text{k}_{1},
\label{eq:charmu1}\\
\mu\left(t_{1}^{n}\otimes\dot{\alpha}^{\vee}\right) & =t_{1}^{n}\otimes\dot{\mu}\left(\dot{\alpha}^{\vee}\right)+\delta_{n,0}\rho_{\mu}\left(\dot{\alpha}^{\vee}\right)\text{k}_{1},\label{eq:charmu2},\\
\mu\left(\text{d}_{1}\right) & =\text{d}_{1}+\boldsymbol{h}-\frac{\left\langle \boldsymbol{h},\boldsymbol{h}\right\rangle }{2}\text{k}_{1},\label{eq:charmu4}
\end{align}
where $\boldsymbol{h}\in\dot{\mathfrak{h}}$ is determined by
\begin{equation}
\dot{\mu}\left(\dot{\gamma}\right)\left(\boldsymbol{h}\right)=-\rho_{\mu}\left(\dot{\gamma}\right)\quad\text{for}\ \dot{\gamma}\in\dot{Q}.\label{eq:defh}
\end{equation}
Furthermore, for any $h\in\dot{\mathfrak{h}}$ we have
\begin{align}
\dot{\mu}^{T}\left(h\right)=h,\quad\sum_{p=0}^{T-1}\rho_{\mu}\left(\dot{\mu}^{p}\left(h\right)\right)=0, \quad\sum_{p=0}^{T-1}
\dot{\mu}^{p}\left(\boldsymbol{h}\right)=0=\sum_{p=1}^{T-1}\left(T-p\right)\left\langle \dot{\mu}^{p}(\boldsymbol{h}),\boldsymbol{h}\right\rangle +\frac{T\left\langle \boldsymbol{h},\boldsymbol{h}\right\rangle }{2}.\label{eq:muN=00003D00003D00003D1}
\end{align}

\end{lemma}
\begin{proof}
The equalities \eqref{eq:charmu1} and \eqref{eq:charmu2} were proved
in \cite[Proposition 2.2]{CJKT}. Assume that $\mu\left(\text{d}_{1}\right)=\boldsymbol{h}+a\text{k}_{1}+b\text{d}_{1}$
with $\boldsymbol{h}\in\dot{\mathfrak{h}}$ and $a,b\in\mathbb{C}.$
Since

\[
1=\left\langle \text{d}_{1},\text{k}_{1}\right\rangle =\left\langle \mu\left(\text{d}_{1}\right), \ \ \mu\left(\text{k}_{1}\right)\right\rangle =\left\langle \boldsymbol{h}+a\text{k}_{1}+b\text{d}_{1},\text{k}_{1}\right\rangle =b,
\]
we obtain $b=1$. Then for any $\dot{\alpha}\in\dot{Q}$,
\[
0=\left\langle \mu\left(\text{d}_{1}\right),\ \ \mu\left(\dot{\alpha}\right)\right\rangle =\left\langle \boldsymbol{h}+a\text{k}_{1}+\text{d}_{1},\ \ \dot{\mu}\left(\dot{\alpha}\right)+\rho_{\mu}\left(\dot{\alpha}\right)\text{k}_{1}\right\rangle =\left\langle \boldsymbol{h},\dot{\mu}\left(\dot{\alpha}\right)\right\rangle +\rho_{\mu}\left(\dot{\alpha}\right),
\]
and hence $\boldsymbol{h}$ is determined by the condition (\ref{eq:defh}).
And by using the fact that
$\left\langle \mu\left(\text{d}_{1}\right),\mu\left(\text{d}_{1}\right)\right\rangle =0$,
one obtains $a=-\frac{\left\langle \boldsymbol{h},\boldsymbol{h}\right\rangle }{2}$,
which implies \eqref{eq:charmu4}. Finally, \eqref{eq:muN=00003D00003D00003D1}
follows from the equalities \eqref{eq:charmu2}, \eqref{eq:charmu4}
and the fact that $\mu$ has order $T$.
\end{proof}
Now we are ready to define an automorphism of the full toroidal Lie
algebra $\mathcal{T}\left(\dot{\mathfrak{g}}\right)$ from  the diagram
automorphism $\mu$ of $\mathfrak{g}$.

\begin{proposition}\label{prop:defhatmu} The following assignment, for $\dot{\alpha}\in\dot{\Delta}$,
$m_{0},m_{1}\in\mathbb{Z}$ and $i=0,1,$
\begin{eqnarray*}
t_{0}^{m_{0}}t_{1}^{m_{1}}\otimes x_{\dot{\alpha}} & \mapsto & t_{0}^{m_{0}}t_{1}^{m_{1}+\rho_{\mu}\left(\dot{\alpha}\right)}\otimes\dot{\mu}\left(x_{\dot{\alpha}}\right),\\
t_{0}^{m_{0}}t_{1}^{m_{1}}\otimes\dot{\alpha}^{\vee} & \mapsto & t_{0}^{m_{0}}t_{1}^{m_{1}}\otimes\dot{\mu}\left(\dot{\alpha}^{\vee}\right)+\rho_{\mu}\left(\dot{\alpha}^{\vee}\right)t_{0}^{m_{0}}t_{1}^{m_{1}}\text{k}_{1},\\
t_{0}^{m_{0}}t_{1}^{m_{1}}\text{k}_{i} & \mapsto & t_{0}^{m_{0}}t_{1}^{m_{1}}\text{k}_{i},\quad t_{0}^{m_{0}}t_{1}^{m_{1}}\text{d}_{0}  \mapsto t_{0}^{m_{0}}t_{1}^{m_{1}}\text{d}_{0},\\
t_{0}^{m_{0}}t_{1}^{m_{1}}\text{d}_{1} & \mapsto & t_{0}^{m_{0}}t_{1}^{m_{1}}\text{d}_{1}+t_{0}^{m_{0}}t_{1}^{m_{1}}\otimes\boldsymbol{h}-\frac{\left\langle \boldsymbol{h},\boldsymbol{h}\right\rangle }{2}t_{0}^{m_{0}}t_{1}^{m_{1}}\text{k}_{1},
\end{eqnarray*}
defines an automorphism, denoted by $\hat{\mu}$, of the full toroidal Lie
algebra $\mathcal{T}\left(\dot{\mathfrak{g}}\right)$ with order $T$.
\end{proposition}
\begin{proof}
 For $m_{0},m_{1},n_{0},n_{1}\in\mathbb{Z}$,
by applying \eqref{eq:relationint1}-\eqref{eq:relationint4}, we have
\begin{align*}
 & \left[\hat{\mu}\left(t_{0}^{m_{0}}t_{1}^{m_{1}}\text{d}_{1}\right),\hat{\mu}\left(t_{0}^{n_{0}}t_{1}^{n_{1}}\text{d}_{1}\right)\right]\\
=\  & \left[t_{0}^{m_{0}}t_{1}^{m_{1}}\text{d}_{1},t_{0}^{n_{0}}t_{1}^{n_{1}}\text{d}_{1}\right]-\frac{\left\langle \boldsymbol{h},\boldsymbol{h}\right\rangle }{2}\left(\left[t_{0}^{m_{0}}t_{1}^{m_{1}}\text{d}_{1},t_{0}^{n_{0}}t_{1}^{n_{1}}\text{k}_{1}\right]+\left[t_{0}^{m_{0}}t_{1}^{m_{1}}\text{k}_{1},t_{0}^{n_{0}}t_{1}^{n_{1}}\text{d}_{1}\right]\right)\\
 & +\left[t_{0}^{m_{0}}t_{1}^{m_{1}}\text{d}_{1},t_{0}^{n_{0}}t_{1}^{n_{1}}\otimes\boldsymbol{h}\right]+\left[t_{0}^{m_{0}}t_{1}^{m_{1}}\otimes\boldsymbol{h},t_{0}^{n_{0}}t_{1}^{n_{1}}\text{d}_{1}\right]+\left[t_{0}^{m_{0}}t_{1}^{m_{1}}\otimes\boldsymbol{h},t_{0}^{n_{0}}t_{1}^{n_{1}}\otimes\boldsymbol{h}\right]\\
=\  & (n_{1}-m_{1})\left(t_{0}^{m_{0}+n_{0}}t_{1}^{m_{1}+n_{1}}\text{d}_{1}+t_{0}^{m_{0}+n_{0}}t_{1}^{m_{1}+n_{1}}\otimes\boldsymbol{h}-\frac{\left\langle \boldsymbol{h},\boldsymbol{h}\right\rangle }{2}t_{0}^{m_{0}+n_{0}}t_{1}^{m_{1}+n_{1}}\text{k}_{1}\right)\\
 & -\frac{\left\langle \boldsymbol{h},\boldsymbol{h}\right\rangle }{2}\left(\sum_{i=1,2}m_{i}t_{0}^{m_{0}+n_{0}}t_{1}^{m_{1}+n_{1}}\text{k}_{i}-\sum_{i=1,2}n_{i}t_{0}^{m_{0}+n_{0}}t_{1}^{m_{1}+n_{1}}\text{k}_{i}\right)+\left\langle \boldsymbol{h},\boldsymbol{h}\right\rangle \sum_{i=1,2}m_{i}t_{0}^{m_{0}+n_{0}}t_{1}^{m_{1}+n_{1}}\text{k}_{i}\\
=\  & \hat{\mu}\left(\left[t_{0}^{m_{0}}t_{1}^{m_{1}}\text{d}_{1},t_{0}^{n_{0}}t_{1}^{n_{1}}\text{d}_{1}\right]\right).
\end{align*}
One can check $\hat{\mu}$ preserves all other relations by a similar argument as above. And it follows from \eqref{eq:muN=00003D00003D00003D1} that the order of the automorphism $\hat{\mu}$ is equal to $T$.
\end{proof}

\subsection{\label{Lie algebras}Subalgebra of $\widetilde{\mathfrak{g}}$ fixed by the automorphism $\tilde{\mu}$}

From the definition of the automorphism $\hat{\mu}$ of the  full toroidal
Lie algebra $\mathcal{T}\left(\dot{\mathfrak{g}}\right)$ given in the previous subsection, it is easy to see that $\hat{\mu}\left(\widetilde{\mathfrak{g}}\right)=\widetilde{\mathfrak{g}}$. 
 Fix a $\mathbb{Z}$-grading $\widetilde{\mathfrak{g}}=\oplus_{m\in\mathbb{Z}}\widetilde{\mathfrak{g}}_{\left(m\right)}$
 by the adjoint action of $-\text{d}_{0}$,
i.e., $\widetilde{\mathfrak{g}}_{\left(m\right)}$=$\left\{ x\in\widetilde{\mathfrak{g}}\mid\left[\text{d}_{0},x\right]=-mx\right\} $.
Set $\omega=\omega_{T}=e^{2\pi\sqrt{-1}/T}$ and let $\omega^{-\text{d}_{0}}$ be the
automorphism of $\widetilde{\mathfrak{g}}$ defined by $\omega^{-\text{d}_{0}}\left(a\right)=\omega^{m}a$
for $a\in\widetilde{\mathfrak{g}}_{\left(m\right)}$.
Therefore,
\[
\tilde{\mu}=\omega^{-\text{d}_{0}}\circ\hat{\mu}|_{\widetilde{\mathfrak{g}}}\] defines an automorphism for the toroidal EALA $\widetilde{\mathfrak{g}}$.
 Now we investigate the Lie subalgebra $\widetilde{\mathfrak{g}}\left[\mu\right]$ of  $ \widetilde{\mathfrak{g}}$ fixed
by the automorphism $\tilde{\mu}$.


Firstly, for $x\in\mathfrak{g}_{\alpha}$ with $\alpha\in\Delta^{\times}\cup\left\{ 0\right\} $,
$h\in\dot{\mathfrak{h}}$, $m\in\mathbb{Z}$ and $n\in \mathbb{Z}^\ast$, we have
\begin{gather}
\tilde{\mu}\left(t_{0}^{m}x\right)=\omega^{-m}t_{0}^{m}\mu\left(x\right),\quad\tilde{\mu}\left(t_{0}^{m}t_{1}^{n}\otimes h\right)=\omega^{-m}\left(t_{0}^{m}t_{1}^{n}\otimes\dot{\mu}\left(h\right)-\rho_{\mu}\left(h\right)m\text{k}_{m,n}\right),
\quad\tilde{\mu}\left(\text{k}_{0}\right)=\text{k}_{0},\nonumber \\
\tilde{\mu}\left(\text{d}_{0}\right)=\text{d}_{0},\quad\tilde{\mu}\left(\text{k}_{m,n}\right)=\omega^{-m}\text{k}_{m,n},\quad\tilde{\mu}\left(\widetilde{\text{d}}_{m,n}\right)=\omega^{-m}\left(\widetilde{\text{d}}_{m,n}+mt_{0}^{m}t_{1}^{n}\otimes\boldsymbol{h}+\frac{\langle\boldsymbol{h},\boldsymbol{h}\rangle}{2}m^{2}\text{k}_{m,n}\right).\label{eq:deftildemu}
\end{gather}
Moreover, we recall that the diagram automorphism $\mu$ preserves the Cartan subalgebra
$\mathfrak{h}$ and the bilinear form $\left\langle \cdot,\cdot\right\rangle $.
Therefore, by \eqref{eq:defform}, \eqref{eq:defh} and \eqref{eq:deftildemu}, we have the following results:

\begin{lemma}\label{lem:tmuinv} 
$\tilde{\mu}\left(\widetilde{\mathfrak{h}}\right)=\widetilde{\mathfrak{h}}$, and $\left\langle \tilde{\mu}\left(x\right), \tilde{\mu}\left(y\right)\right\rangle =\left\langle x,y\right\rangle $
for $x,y\in\widetilde{\mathfrak{g}}$. \end{lemma}

Denoted by $\widetilde{\mathfrak{g}}\left[\mu\right]$
and $\widetilde{\mathfrak{h}}\left[\mu\right]$ the Lie subalgebras of
$\widetilde{\mathfrak{g}}$ and $\mathfrak{\widetilde{\mathfrak{h}}}$ respectively
fixed by $\tilde{\mu}$. We consider the root space decomposition of $\widetilde{\mathfrak{g}}\left[\mu\right]$
with respect to its abelian subalgebra $\widetilde{\mathfrak{h}}\left[\mu\right]$.
Denote by $\left(\mathfrak{h}^{\ast}\right)^{\mu}$ the subspace of
$\mathfrak{h}^{\ast}$ that is fixed by $\mu$ and denote by
\begin{gather*}
\pi_{\mu}:\mathfrak{h}^{\ast}\to\left(\mathfrak{h}^{\ast}\right)^{\mu},\quad\alpha\mapsto\check{\alpha}=\frac{1}{T}\sum_{p=0}^{T-1}\mu^{p}\left(\alpha\right)
\end{gather*}
the canonical projection. Since $\check{\alpha}\left(h-\mu\left(h\right)\right)=0$
for $\alpha\in\mathfrak{h}^{\ast}$ and $h\in\mathfrak{h}$, we may
identify $\left(\mathfrak{h}^{\mu}\right)^{\ast}$ with $\left(\mathfrak{h}^{\ast}\right)^{\mu},$
where $\mathfrak{h}^{\mu}$ is the subspace of $\mathfrak{h}$ fixed
by $\mu$. Note that $\widetilde{\mathfrak{h}}\left[\mu\right]=\mathfrak{h}^{\mu}\oplus\mathbb{C}\text{k}_{0}\oplus\mathbb{C}\text{d}_{0}$.
We view $\left(\mathfrak{h}^{\mu}\right)^{\ast}$ as a subspace of
$\widetilde{\mathfrak{h}}\left[\mu\right]^{\ast}$ such that $\alpha\left(\text{k}_{0}\right)=\alpha\left(\text{d}_{0}\right)=0$
for $\alpha\in\left(\mathfrak{h}^{\mu}\right)^{\ast}$, and define
$\delta_{0}\in\widetilde{\mathfrak{h}}\left[\mu\right]^{\ast}$ by
\begin{equation}
\delta_{0}\left(\mathfrak{h}^{\mu}\right)=\delta_{0}\left(\text{k}_{0}\right)=0,\ \ \delta_{0}\left(\text{d}_{0}\right)=1.\label{def of delta_0}
\end{equation}
In particular, for each $\alpha\in\Delta$, $\check{\alpha}\in\left(\mathfrak{h}^{\ast}\right)^{\mu}=\left(\mathfrak{h}^{\mu}\right)^{\ast}$
is an element in $\widetilde{\mathfrak{h}}\left[\mu\right]^{\ast}$.

For $\alpha\in\widetilde{\mathfrak{h}}\left[\mu\right]^{\ast},$ set
\[
\widetilde{\mathfrak{g}}\left[\mu\right]_{\alpha}=\left\{ x\in\widetilde{\mathfrak{g}}\left[\mu\right]\mid\left[h,x\right]=\alpha\left(h\right)x,\ h\in\widetilde{\mathfrak{h}}\left[\mu\right]\right\} .
\]
Then we have the following root space decomposition
\[
\widetilde{\mathfrak{g}}\left[\mu\right]=\widetilde{\mathfrak{g}}\left[\mu\right]_{0}\oplus\sum_{\alpha\in\tilde{\Delta}_{\mu}}\widetilde{\mathfrak{g}}\left[\mu\right]_{\alpha},
\]
where $\tilde{\Delta}_{\mu}=\left\{ \alpha\in\tilde{\mathfrak{h}}\left[\mu\right]^{\ast}\backslash\left\{ 0\right\} \mid\widetilde{\mathfrak{g}}\left[\mu\right]_{\alpha}\not=0\right\} .$
Let
\begin{eqnarray}
\eta_{\mu}:\widetilde{\mathfrak{g}}\to\widetilde{\mathfrak{g}}\left[\mu\right],\quad x\mapsto\sum_{p=0}^{T-1}\tilde{\mu}^{p}\left(x\right)\label{eq:defetamu}
\end{eqnarray}
be a projection from $\widetilde{\mathfrak{g}}$ to $\widetilde{\mathfrak{g}}\left[\mu\right]$.
It can be readily seen that
\begin{align}
\eta_{\mu}\left(t_{0}^{m}u\right)\in\widetilde{\mathfrak{g}}\left[\mu\right]_{\check{\alpha}+m\delta_{0}},\quad\eta_{\mu}\left(\text{k}_{m,n}\right),\eta_{\mu}\left(\tilde{\text{d}}_{m,n}\right)\in\widetilde{\mathfrak{g}}\left[\mu\right]_{m\delta_{0}+n\delta_{1}},\quad\text{k}_{0},\text{d}_{0}\in\widetilde{\mathfrak{g}}\left[\mu\right]_{0}\label{eq:rootvectors}
\end{align}
for $u\in\mathfrak{g}_{\alpha}$ with $\alpha\in\Delta\cup\left\{ 0\right\} $,
$m\in\mathbb{Z}$ and $n\in\mathbb{Z}^{\ast}$. Then we have

\begin{lemma}\label{lem:selfcenh} The root system $\tilde{\Delta}_{\mu}\subset\left(\pi_{\mu}\left(\Delta\right)+\mathbb{Z}\delta_{1}\right)\cup\mathbb{Z}\delta_{1}$
and $\mathfrak{\widetilde{\mathfrak{g}}}\left[\mu\right]_{0}=\mathfrak{\widetilde{h}}\left[\mu\right]$.
\end{lemma}
\begin{proof}
The first assertion follows from \eqref{eq:basisoftildeg} and \eqref{eq:rootvectors}.
For the second one, note that if $\alpha=\sum_{i\in I}a_{i}\alpha_{i}$
is a root in $\Delta$, then $\check{\alpha}=\frac{1}{T}\sum_{i\in I}\sum_{p=0}^{T-1}a_{i}\alpha_{\mu^{p}(i)}$
is clearly nonzero. This together with \eqref{eq:rootvectors} implies
that
\[
\widetilde{\mathfrak{g}}\left[\mu\right]_{0}=\eta_{\mu}\left(\mathfrak{h}\right)\oplus\mathbb{C}\text{k}_{0}\oplus\mathbb{C}\text{d}_{0}=\widetilde{\mathfrak{h}}\left[\mu\right].
\]
\end{proof}
Denote by $\mathfrak{t}\left(\dot{\mathfrak{g}},\mu\right)$ the subalgebra
of $\mathfrak{t}\left(\dot{\mathfrak{g}}\right)$ fixed by $\tilde{\mu}$
(noting that $\tilde{\mu}\left(\mathfrak{t}\left(\dot{\mathfrak{g}}\right)\right)=\mathfrak{t}\left(\dot{\mathfrak{g}}\right)$).
We have the following result from \cite{CJKT}. \begin{proposition}\label{prop:centralclosed}
The Lie algebra $\mathfrak{t}\left(\dot{\mathfrak{g}},\mu\right)$
is centrally closed. \end{proposition}

In the rest of this subsection, we recall a characterization for the subset  of the root system $ \tilde{\Delta}_{\mu}$ given in \cite{CJKT}:
\[
\tilde{\Delta}_{\mu}^{\times}=\left\{ \check{\alpha}+m\delta_{0}\in\tilde{\Delta}_{\mu}\mid\alpha\in\Delta,m\in\mathbb{Z}\ \text{and}\ \left\langle \check{\alpha},\check{\alpha}\right\rangle \not=0\right\}.
\]
 For every $i\in I=\left\{ 0,1,\cdots,l\right\} $,
denote by $\mathcal{O}\left(i\right)$ the orbit containing $i$ under
the action of the group $\left\langle \mu\right\rangle $. We say
$\mu$ is \emph{transitive} if $\mathcal{O}\left(i\right)=I$ for
each $i\in I$. Observe that a diagram automorphism on $\mathfrak{g}$
is transitive if and only if $\mathfrak{g}$ is of type $A_{l}^{\left(1\right)}$,
 and it has order $l+1$. Note that in this case we have $\pi_{\mu}\left(\alpha_{i}\right)=\delta_{1}$
for any $i\in I$, and hence $\tilde{\Delta}_{\mu}^{\times}=\emptyset.$

From now on we assume that $\mu$ is nontransitive, and in this case, it is known that the folded matrix $\check{A}=\left(2\frac{\left\langle \check{\alpha}_{i},\check{\alpha}_{j}\right\rangle }{\left\langle \check{\alpha}_{i},\check{\alpha}_{i}\right\rangle }\right)_{i,j\in\check{I}}$
of $A$ associated to $\mu$ is also an affine generalized Cartan
matrix (cf. \cite{FSS,ABP}), where $\check{I}=\left\{ i\in I\mid \mu^{k}\left(i\right)\ge i\ \text{for\ }k\in\mathbb{Z}\right\}$ is a set
of representative elements in $I$.
Denote by $\check{\Delta}$ and $\check{W}$ respectively the root system and the
Weyl group associated to the folded matrix $\check{A}$.
Then $\left\{ \check{\alpha}_{i}\right\} _{i\in\check{I}}$ is a simple
root base of $\check{\Delta}$. Furthermore, we have the following result from
\cite[Lemma 12.15]{ABP}.

\begin{lemma}\label{lem:defsi} For each $i\in I$, one and only one of
the following statement holds:

(a) The elements $\alpha_{p},$ for $p\in\mathcal{O}\left(i\right)$, are pairwise
orthogonal;

(b) $\mathcal{O}\left(i\right)=\left\{ i,\mu\left(i\right)\right\} $,
and $a_{i\mu\left(i\right)}=-1=a_{\mu\left(i\right)i}.$

\end{lemma}

Recall that $T$ is the order of the diagram automorphism $\mu$.
For every $i\in I$, we set
\begin{align}
T_{i}=T/\left|\mathcal{O}\left(i\right)\right|, \ \text{and}\ s_{i}=\begin{cases}
1, & \text{if}\ \text{(a)\ holds\ in Lemma \ref{lem:defsi};}\\
2, & \text{if\ (b)\ holds\ in Lemma \ref{lem:defsi}. }
\end{cases}\label{eq:defsi}
\end{align}
The following result was proved in \cite[Proposition 5.1]{CJKT}.

\begin{proposition}\label{prop:nonisoroots} If $\mu$ is nontransitive,
then
\begin{alignat*}{1}
\tilde{\Delta}_{\mu}^{\times}=\left\{ \check{w}\check{\alpha}_{i}+T_{i}m\delta_{0}\mid\check{w}\in\check{W},i\in\check{I},m\in\mathbb{Z}\right\} \cup\left\{ 2\check{w}\check{\alpha}_{i}+\left(\frac{T}{2}+mT\right)\delta_{0}\mid\check{w}\in\check{W},i\in\check{I}\ \text{with\ }s_{i}=2,m\in\mathbb{Z}\right\} .
\end{alignat*}
\end{proposition}

\begin{remark}\label{rem:ears} If $\mu$ is nontransitive, then
$\tilde{\Delta}_{\mu}^{\times}$ is a nullity 2 reduced extended affine
root system introduced by Saito (cf. \cite{Sa}). Conversely, any nullity
2 reduced extended affine root system is of such a form (cf. \cite{ABP}).
Furthermore, we will see in the next section that the triple $\left(\widetilde{\mathfrak{g}}\left[\mu\right],\widetilde{\mathfrak{h}}\left[\mu\right],\langle\cdot,\cdot\rangle\right)$
is a nullity $2$ EALA, and
which we call a \emph{nullity $2$ twisted toroidal extended affine
Lie algebra}.

\end{remark}

\section{Nullity 2 EALA of maximal type}

In this section, we first recall the definition of extended affine Lie algebra (cf. \cite{AABGP,N}), and the classification of nullity $2$ extended
affine Lie algebras of maximal type based on Allison-Berman-Pianzola's
work \cite{ABP}.

Let
$\mathcal{E}$ be a Lie algebra equipped with a nontrivial finite-dimensional
self-centralizing ad-diagonalizable subalgebra $\mathcal{H}$ and
a nondegenerate invariant symmetric bilinear form $\left(\cdot |\cdot \right).$
Let $\mathcal{E}=\mathcal{H}\oplus\sum_{\alpha\in\Phi}\mathcal{E}_{\alpha}$
be the root space decomposition  with respect to
$\mathcal{H},$ where $\Phi$ is the corresponding root system. The
form $\left(\cdot |\cdot \right)$ is also nondegenerate when it is restricted
to $\mathcal{H}=\mathcal{E}_{0}$. Hence it induces a nondegenerate
symmetric bilinear form on $\mathcal{H}^{\ast}$. Set
\[
\Phi^{\times}=\left\{ \alpha\in\Phi\mid\left(\alpha|\alpha\right)\not=0\right\}, \quad\Phi^{0}=\left\{ \alpha\in\Phi\mid\left(\alpha|\alpha\right)=0\right\} .
\]
Let $\mathcal{E}_{c}$ be the subalgebra of $\mathcal{E}$ generated
by the root spaces $\mathcal{E}_{\alpha},\alpha\in\Phi^{\times}$,
which is called the \emph{core} of $\mathcal{E}$.

\begin{definition}\label{def:eala} $\mathcal{E}:=\left(\mathcal{E},\mathcal{H},(\cdot |\cdot)\right)$ is called an \emph{extended affine Lie
algebra} (EALA for short) if

(1) $\text{ad}\left(x\right)$ is locally nilpotent for $x\in\mathcal{E}_{\alpha},\alpha\in\Phi^{\times};$

(2) $\Phi^{\times}$ cannot be decomposed as a union of two orthogonal
nonempty subsets;

(3) The centralizer of $\mathcal{E}_{c}$ in $\mathcal{E}$ is contained
in $\mathcal{E}_{c}$;

(4) $\Phi$ is a discrete subset of $\mathcal{H}^{\ast}$ with respect
to its Euclidean topology.

\end{definition}

The axiom (4) implies that the subgroup $\left\langle \Phi^{0}\right\rangle $
of $\mathcal{H}^{\ast}$ generated by $\Phi^{0}$ is a free abelian
group of finite rank, and this rank is called the \emph{nullity} of
$\mathcal{E}$. Indeed, the nullity $0$ EALAs are exactly the finite dimensional
simple Lie algebras, while nullity $1$ EALAs are exactly the affine Kac-Moody
algebras (cf. \cite{ABGP}). For the purpose of classifying EALAs, the following
notion was introduced in \cite{BGK}.

\begin{definition} An EALA is of \emph{maximal type }if its core
is centrally closed. \end{definition}

Note that the nullity $0$ and nullity $1$ EALAs are all of maximal
type. In general, an EALA may not be of maximal type. However, the
maximal EALAs appear to be the most interesting ones as from them
one can know the structure of other EALAs (see \cite[Remark 3.73]{BGK}
and \cite{N} for details). In what follows, we give two classes of
nullity $2$ EALAs of maximal type.

\begin{proposition} \label{prop:EALA-1} Let $\mu$ be a nontransitive
diagram automorphism of an untwisted affine Kac-Moody algebra $\mathfrak{g}$.
Then the triple $\left(\widetilde{\mathfrak{g}}\left[\mu\right],\widetilde{\mathfrak{h}}\left[\mu\right],\left\langle \cdot,\cdot\right\rangle \right)$, defined in the previous section,
is a nullity 2 EALA of maximal type, and $\mathfrak{t}\left(\dot{\mathfrak{g}},\mu\right)$
is the core of $\widetilde{\mathfrak{g}}\left[\mu\right]$. \end{proposition}
\begin{proof}
Recall from Lemma \ref{lem:selfcenh} that the ad-diagonalizable subalgebra
$\widetilde{\mathfrak{h}}\left[\mu\right]$ of $\widetilde{\mathfrak{g}}\left[\mu\right]$
is self-centralized, and it follows from Lemma \ref{lem:tmuinv}
that the invariant form $\left\langle \cdot,\cdot\right\rangle $
restricted to $\widetilde{\mathfrak{g}}\left[\mu\right]$ is still
nondegenerate. One can easily check that $\mathfrak{t}\left(\dot{\mathfrak{g}},\mu\right)$
is generated by the elements $\eta_{\mu}\left(t_{0}^{m}x_{\pm\alpha_{i}}\right)$
for $m\in\mathbb{Z}$ and $i\in I$. This together with Proposition
\ref{prop:nonisoroots} implies that the core of $\widetilde{\mathfrak{g}}\left[\mu\right]$
is $\mathfrak{t}\left(\dot{\mathfrak{g}},\mu\right)$. Now we check
the axioms (1)-(4) in Definition \ref{def:eala}. The axiom (1) follows
from Proposition \ref{prop:nonisoroots}, while the axiom (2) is implied
by Remark \ref{rem:ears} as the root system defined by Saito is irreducible
(see \cite[Definition 1]{Sa}). The axioms (3) and (4) are obvious.
Finally, Proposition \ref{prop:centralclosed} implies the maximality
of $\widetilde{\mathfrak{g}}\left[\mu\right]$.
\end{proof}

For the second class of nullity $2$ EALAs of maximal type. We let $q\in \mathbb{C}^{\ast}$ be generic, i.e, $q$ is not a root of unity, and $\mathbb{C}_{q}:=\mathbb{C}_{q}\left[t_{0}^{\pm1},t_{1}^{\pm1}\right]$
 the quantum torus associated to $q$ such that $t_{0}t_{1}=qt_{1}t_{0}$.
For any positive integer $N\ge2$, denoted by $\mathfrak{gl}_{N}\left(\mathbb{C}_{q}\right)$
the general linear Lie algebra over $\mathbb{C}_{q}$, and $\mathfrak{sl}_{N}\left(\mathbb{C}_{q}\right)$
the derived subalgebra of $\mathfrak{gl}_{N}\left(\mathbb{C}_{q}\right)$.
For $1\le i,j\le N$, $a\in\mathbb{C}_{q}$, we write $E_{i,j}a$
the matrix whose only possible nonzero entry is the $\left(i,j\right)$-entry
which is $a$. We consider the central extension
of the Lie algebra $\mathfrak{gl}_{N}(\mathbb{C}_{q})$:
\begin{align*}
\mathfrak{\widehat{gl}}_{N}\left(\mathbb{C}_{q}\right):=\mathfrak{gl}_{N}\left(\mathbb{C}_{q}\right)\oplus\mathbb{C}\text{k}_{0}\oplus\mathbb{C}\text{k}_{1},
\end{align*}
where $\text{k}_{0},\text{k}_{1}$ are central
elements, and
\begin{equation}
\begin{split}\label{commutator1}\left[E_{i,j}t_{0}^{m_{0}}t_{1}^{m_{1}},E_{k,l}t_{0}^{n_{0}}t_{1}^{n_{1}}\right]=\  & \delta_{j,k}q^{m_{1}n_{0}}E_{i,l}t_{0}^{m_{0}+n_{0}}t_{1}^{m_{1}+n_{1}}-\delta_{i,l}q^{n_{1}m_{0}}E_{k,j}t_{0}^{m_{0}+n_{0}}t_{1}^{m_{1}+n_{1}}\\
 & +\delta_{j,k}\delta_{i,l}\delta_{m_{0}+n_{0},0}\delta_{m_{1}+n_{1},0}q^{m_{1}n_{0}}\left(m_{0}\text{k}_{0}+m_{1}\text{k}_{1}\right),
\end{split}
\end{equation}
for $1\le i,j,k,l\le N$ and $m_{0},m_{1},n_{0},n_{1}\in\mathbb{Z}$. Moreover, we
define two derivations $\text{d}_{0},\text{d}_{1}$ acting on $\mathfrak{\widehat{gl}}_{N}\left(\mathbb{C}_{q}\right)$
by
\begin{align*}
[\text{d}_{r}, E_{i,j}t_{0}^{m_{0}}t_{1}^{m_{1}}]=m_{r}E_{i,j}t_{0}^{m_{0}}t_{1}^{m_{1}},\quad [\text{d}_{r}, \text{k}_{s}]=0=[\text{d}_{r}, \text{d}_{s}],
\end{align*}
for $r,s\in \{0,1\}$ and for $1\le i,j\le N,m_{0},m_{1}\in\mathbb{Z}$. Therefore,
we obtain a Lie algebra
\begin{eqnarray*}
\mathfrak{\widetilde{gl}}_{N}\left(\mathbb{C}_{q}\right):=\mathfrak{\widehat{gl}}_{N}\left(\mathbb{C}_{q}\right)\oplus\mathbb{C}\text{d}_{0}\oplus\mathbb{C}\text{d}_{1}.
\end{eqnarray*}

Set
\begin{eqnarray*}
\widehat{\mathfrak{sl}}_{N}\left(\mathbb{C}_{q}\right)=\left[\mathfrak{\widehat{gl}}_{N}\left(\mathbb{C}_{q}\right),\mathfrak{\widehat{gl}}_{N}\left(\mathbb{C}_{q}\right)\right]=\mathfrak{sl}_{N}\left(\mathbb{C}_{q}\right)\oplus\mathbb{C}\text{k}_{0}\oplus\mathbb{C}\text{k}_{1},
\end{eqnarray*}
the derived subalgebra of $\mathfrak{\widehat{gl}}_{N}\left(\mathbb{C}_{q}\right)$.
It is known that $\widehat{\mathfrak{sl}}_{N}\left(\mathbb{C}_{q}\right)$
is central closed and $\mathfrak{\widehat{gl}}_{N}\left(\mathbb{C}_{q}\right)=\widehat{\mathfrak{sl}}_{N}\left(\mathbb{C}_{q}\right)\oplus\mathbb{C}I_{N}$
(cf. \cite{BGK}), where $I_{N}$ is the identity matrix. Furthermore, we
define
\begin{eqnarray*}
\widetilde{\mathfrak{sl}}_{N}\left(\mathbb{C}_{q}\right):=\widehat{\mathfrak{sl}}_{N}\left(\mathbb{C}_{q}\right)\oplus\mathbb{C}\text{d}_{0}
\oplus\mathbb{C}\text{d}_{1},\quad\widetilde{\mathcal{H}}:=\sum_{i=1}^{N-1}\mathbb{C}\left(E_{i,i}-E_{i+1,i+1}\right)\oplus\mathbb{C}\text{k}_{0}
\oplus\mathbb{C}\text{k}_{1}\oplus\mathbb{C}\text{d}_{0}\oplus\mathbb{C}\text{d}_{1},
\end{eqnarray*}
and a bilinear form $\left\langle \cdot,\cdot\right\rangle $ of
 $\widetilde{\mathfrak{sl}}_{N}\left(\mathbb{C}_{q}\right)$ such that
\begin{align*}
\left\langle E_{i,j}t_{0}^{m_{0}}t_{1}^{m_{1}},E_{j,i}t_{0}^{-m_{0}}t_{1}^{-m_{1}}\right\rangle =1=\left\langle \text{d}_{r},\text{k}_{r}\right\rangle,
\end{align*}
for $1\le i,j\le N,m_{0},m_{1}\in\mathbb{Z}$, $r\in\{0,1\}$, and trivial for others. The following
result is from \cite{BGK}:

\begin{proposition} \label{prop:EALA-2} Let $N\ge2$ be a positive
integer and $q\in\mathbb{C}^{\ast}$ be generic. Then the triple $\left(\widetilde{\mathfrak{sl}}_{N}\left(\mathbb{C}_{q}\right),\widetilde{\mathcal{H}},\left\langle \cdot,\cdot\right\rangle \right)$
is a nullity 2 EALA of maximal type, and $\widehat{\mathfrak{sl}}_{N}\left(\mathbb{C}_{q}\right)$
is the core of $\mathfrak{\widetilde{sl}}_{N}(\mathbb{C}_{q})$. \end{proposition}

For any given EALA $\mathcal{E}$, we call $\mathcal{E}_{cc}:=\mathcal{E}_{c}/{\text{Z}}\left(\mathcal{E}_{c}\right)$
\emph{the centerless core} of $\mathcal{E}$, where ${\text{Z}}\left(\mathcal{E}_{c}\right) $ is the center of the core $ \mathcal{E}_{c}$. One notes that the centerless
core of the affine Kac-Moody algebra $\mathfrak{g}$ is isomorphic to the loop algebra $\overline{\mathfrak{g}}=\mathbb{C}\left[t_{1},t_{1}^{-1}\right]\otimes\dot{\mathfrak{g}}.$
Denote by $\bar{\mu}$ the automorphism of $\overline{\mathfrak{g}}$
induced by $\mu$. Then the centerless core of $\widetilde{\mathfrak{g}}\left[\mu\right]$
is isomorphic to the $\bar{\mu}$-twisted loop algebra
\[
\mathcal{L}\left(\overline{\mathfrak{g}},\bar{\mu}\right)=\sum_{n\in\mathbb{Z}}t_{0}^{n}\otimes\overline{\mathfrak{g}}_{\left(m\right)}\subset\mathbb{C}\left[t_{0},t_{0}^{-1}\right]\otimes\overline{\mathfrak{g}}=\mathcal{R}\otimes\dot{\mathfrak{g}},
\]
where $\overline{\mathfrak{g}}_{\left(m\right)}=\left\{ x\in\overline{\mathfrak{g}}\mid\bar{\mu}\left(x\right)=\omega^{m}x\right\} $.

It is clear that the centerless core of $\widetilde{\mathfrak{sl}}_{N}\left(\mathbb{C}_{q}\right)$ with $q$ generic
is isomorphic to $\mathfrak{sl}_{N}\left(\mathbb{C}_{q}\right)$. The following classification of centerless cores of nullity 2 EALAs
was given in \cite{ABP}:

\begin{proposition} \label{prop:centerlesscore} The centerless core
of a nullity 2 EALA is either isomorphic to $\mathfrak{sl}_{N}\left(\mathbb{C}_{q}\right)$
for some positive integer $N\ge2$ and generic $q\in\mathbb{C}^{\ast}$,
or to $\mathcal{L}\left(\overline{\mathfrak{g}},\bar{\mu}\right)$
for some nontransitive diagram automorphism $\mu$ of an untwisted
affine Kac-Moody algebra $\mathfrak{g}$. \end{proposition}

Two EALAs are said to be \emph{ equivalent} if their cores are isomorphic.
Then we have the following result.

\begin{theorem}\label{thm:2EALAs} Up to equivalence, a nullity $2$
extended affine Lie algebra of maximal type either has the form $\widetilde{\mathfrak{sl}}_{N}\left(\mathbb{C}_{q}\right)$
for some positive integer $N\ge2$ and generic $q\in\mathbb{C}^{\ast}$,
or has the form $\widetilde{\mathfrak{g}}\left[\mu\right]$ for some
nontransitive diagram automorphism $\mu$ of an untwisted affine Kac-Moody
algebra $\mathfrak{g}$. \end{theorem}
\begin{proof}
Let $\mathcal{E}$ be a nullity $2$ EALA of maximal type. Then by
Proposition \ref{prop:centerlesscore}, either $\mathcal{E}_{cc}\cong\mathfrak{sl}_{N}\left(\mathbb{C}_{q}\right)$,
or $\mathcal{E}_{cc}\cong\mathcal{L}\left(\overline{\mathfrak{g}},\bar{\mu}\right)$.
Then the maximality of $\mathcal{E}$ forces that either $\mathcal{E}_{c}\cong\widehat{\mathfrak{sl}}_{N}\left(\mathbb{C}_{q}\right)$
or $\mathcal{E}_{c}\cong\mathfrak{t}\left(\dot{\mathfrak{g}},\mu\right)$.
Thus  it
follows from Propositions \ref{prop:EALA-1} and \ref{prop:EALA-2} that $\mathcal{E}$ is either equivalent to $\widetilde{\mathfrak{sl}}_{N}\left(\mathbb{C}_{q}\right)$
or to $\widetilde{\mathfrak{g}}\left[\mu\right]$.
\end{proof}

 \begin{remark}\label{rem:defaffinecocyle}
Following \cite{BGK,N}, a \emph{$\mathcal{K}$-valued affine cocycle
$\tau$ on $\mathcal{S}$} is an abelian $2$-cocycle $\tau:\mathcal{S}\times\mathcal{S}\rightarrow\mathcal{K}$
such that $\tau\left(\mathcal{S},\text{d}_{i}\right)=0$ for $i=0,1$,
and $\left\langle \tau\left(s_{1},s_{2}\right),s_{3}\right\rangle =\left\langle s_{1},\tau\left(s_{2},s_{3}\right)\right\rangle $
for $s_{1},s_{2},s_{3}\in\mathcal{S}$. For any affine cocycle $\tau$,
one can define a new Lie multiplication $[\cdot,\cdot]_{\tau}$ on $\widetilde{\mathfrak{g}}$
defined by
\begin{align*}
\left[x_{1}+s_{1},x_{2}+s_{2}\right]{}_{\tau}=\left[x_{1}+s_{1},x_{2}+s_{2}\right]+\tau\left(s_{1},s_{2}\right)
\end{align*}
for $x_{1},x_{2}\in\mathfrak{t}\left(\dot{\mathfrak{g}}\right)$ and
$s_{1},s_{2}\in\mathcal{S}$. We denote the resulting Lie algebra by
$\widetilde{\mathfrak{g}}_{\tau}$. It is easy to see that the action
\eqref{eq:deftildemu} also defines an automorphism $\tilde{\mu}$
of $\widetilde{\mathfrak{g}}_{\tau}$. Denoted by $\widetilde{\mathfrak{g}}\left[\mu\right]{}_{\tau}$
the subalgebra of $\widetilde{\mathfrak{g}}_{\tau}$ fixed by $\tilde{\mu}$.
Similar to the proof of Proposition \ref{prop:EALA-1}, one can check
that $\left(\widetilde{\mathfrak{g}}\left[\mu\right]_{\tau},\widetilde{\mathfrak{h}}\left[\mu\right],\left\langle \cdot,\cdot\right\rangle \right)$
is a nullity 2 EALA of maximal type if $\mu$ is nontransitive.
By the explicit construction of EALAs (of maximal type) given by Neher
in \cite{N}, one can prove that if an EALA is equivalent to $\widetilde{\mathfrak{g}}\left[\mu\right]$,
then it is isomorphic to $\widetilde{\mathfrak{g}}\left[\mu\right]_{\tau}$
for some affine cocycle $\tau$. And, if  an EALA is equivalent to $\widetilde{\mathfrak{sl}}_{N}\left(\mathbb{C}_{q}\right)$,
then it must be isomorphic to  $\widetilde{\mathfrak{sl}}_{N}\left(\mathbb{C}_{q}\right)$. \end{remark}

\begin{remark}\label{rem:deftaua} It was shown in \cite{MRY} that
there exist nontrivial affine cocycles. For example, for any complex
number $a$, the bilinear map $\tau_{a}:\mathcal{S}\times\mathcal{S}\rightarrow\mathcal{K}$
defined by
\[
\tau_{a}\left(\mathcal{S},\text{d}_{r}\right)=0,\quad\tau_{a}\left(\text{d}_{m_{0},m_{1}},\text{d}_{n_{0},n_{1}}\right)=a\left(m_{0}n_{1}-n_{0}m_{1}\right){}^{3}\text{k}_{m_{0}+n_{0},m_{1}+n_{1}}
\]
for $m_{0},m_{1},n_{0},n_{1}\in\mathbb{Z}$ and $r=0,1$, is an affine
cocycle. Then we have the Lie algebras $\widetilde{\mathfrak{g}}\left[\mu\right]_{\tau_{a}}$
with $\widetilde{\mathfrak{g}}\left[\mu\right]_{\tau_{0}}=\widetilde{\mathfrak{g}}\left[\mu\right]$.
It is conjectured that any $\mathcal{K}$-valued affine cocycle on
$\mathcal{S}$ has the form $\tau_{a}$ for some $a\in\mathbb{C}$.
This will imply that $\widetilde{\mathfrak{g}}\left[\mu\right]_{\tau_{a}}$
and $\widetilde{\mathfrak{sl}}_{N}\left(\mathbb{C}_{q}\right)$ exhaust
all nullity $2$ EALAs of maximal type up to isomorphism. \end{remark}

\section{Associating $\tilde{\mathfrak{g}}\left[\mu\right]$ with vertex algebras}

Let $\mu$ be a fixed diagram automorphism of the untwisted affine
Kac-Moody algebra $\mathfrak{g}$. In this section, we associate the twisted toroidal
EALA $\widetilde{\mathfrak{g}}\left[\mu\right]$ with vertex algebras
through equivariant $\phi$-coordinated quasi modules.

\subsection{Vertex algebras $V_{\widehat{\mathfrak{g}}}\left(\ell,0\right)$
and $L_{\widehat{\mathfrak{g}}}\left(\ell,0\right)$}

We recall the variant of the skew derivation algebra
$\mathcal{S}$ introduced in \cite{CLiT}:
\[
\widehat{\mathcal{S}}=\left\{ f_{0}\frac{\partial}{\partial t_{0}}+f_{1}\text{d}_{1}\mid f_{0},f_{1}\in\mathcal{R},\frac{\partial}{\partial t_{0}}\left(f_{0}\right)+\text{d}_{1}\left(f_{1}\right)=0\right\} \subset\mathrm{Der}(\mathcal{R}).
\]
For $m,n\in\mathbb{Z}$, set
\[
\widehat{\text{d}}_{n,m}=\left(n+1\right)t_{0}^{n}t_{1}^{m}\text{d}_{1}-mt_{0}^{n}t_{1}^{m}\text{d}_{0},
\]
It is easy to see that the set
\[
\mathbb{B}_{\widehat{\mathcal{S}}}=\left\{ t_{0}^{-1}\text{d}_{0},t_{0}^{-1}\text{d}_{1}\right\} \cup\left\{ \widehat{\text{d}}_{n,m}\mid\left(n,m\right)\in\mathbb{Z}\times\mathbb{Z}\setminus\left\{ \left(-1,0\right)\right\} \right\}
\]
\[
=\left\{ t_{0}^{-1}\text{d}_{0}\right\} \cup\left\{ t_{0}^{n}\text{d}_{1},\ \widehat{\text{d}}_{n,m}\mid n\in\mathbb{Z},m\in\mathbb{Z}^{\ast}\right\}
\]
forms a $\mathbb{C}$-basis of $\widehat{\mathcal{S}}$, and subject to the following relations
\begin{align}
 & \left[t_{0}^{-1}\text{d}_{0},\widehat{\text{d}}{}_{n,m}\right]=\left(n+1\right)\widehat{\text{d}}_{n-1,m},\quad\left[t_{0}^{-1}\text{d}_{1},\widehat{\text{d}}_{n,m}\right]=m\widehat{\text{d}}_{n-1,m},\label{eq:relationinhats1}\\
 & \left[\widehat{\text{d}}_{n,m},\widehat{\text{d}}_{n_{1},m_{1}}\right]=\left(\left(n+1\right)m_{1}-m\left(n_{1}+1\right)\right)\widehat{\text{d}}_{n+n_{1},m+m_{1}},\label{eq:relationinhats2}
\end{align}
for $n,m,n_{1},m_{1}\in\mathbb{Z}$. In view of this, we have the
following subalgebra of the full toroidal Lie algebra $\mathcal{T}(\dot{\mathfrak{g}})$:
\begin{align*}
\widehat{\mathfrak{g}}=\mathfrak{t}\left(\dot{\mathfrak{g}}\right)+\left[\widehat{\mathcal{S}},\widehat{\mathcal{S}}\right]+\mathbb{C}t_{0}^{-1}\text{d}_{1}.
\end{align*}

Note that $\widehat{\mathfrak{g}}$ is linearly spanned by the set
\begin{align}
\left\{ t_{0}^{n}u,\ \text{k}_{n,m},\ \text{k}_{0},\ \widehat{\text{d}}_{n,m}\mid u\in\mathfrak{g},\ n\in\mathbb{Z},\ m\in\mathbb{Z}^{\ast}\right\} .\label{eq:basisofhatg}
\end{align}
From \eqref{eq:relationint2} and \eqref{eq:relationint3}, we have
\begin{align}
 & \left[\widehat{\text{d}}_{i,m},t_{0}^{j}t_{1}^{n}\otimes x\right]=\left(\left(i+1\right)n-mj\right)t_{0}^{i+j}t_{1}^{m+n}\otimes x,\label{eq:relationinhatg1}\\
 & \left[\widehat{\text{d}}_{i,m},\text{k}_{j,n}\right]=\left(\left(i+1\right)\left(m+n\right)-m\left(i+j\right)\right)\text{k}_{i+j,m+n}+\delta_{m+n,0}\delta_{i+j,0}\left(\left(i+1\right)\text{k}_{0}+m\text{k}_{1}\right),\label{eq:relationinhatg2}
\end{align}
for $i,j,m,n\in\mathbb{Z}$ and $x\in\dot{\mathfrak{g}}$. Set
\begin{align}
\widehat{\mathfrak{g}}_{+} & =\text{Span}\left\{ t_{0}^{n}u,\text{k}_{n+1,m},\widehat{\text{d}}_{n-1,m}\mid u\in\mathfrak{g},\ n\in\mathbb{N},\ m\in\mathbb{Z}\right\} ,\label{eq:decofhatg}\\
\widehat{\mathfrak{g}}_{-} & =\text{Span}\left\{ t_{0}^{-n-1}u,\text{k}_{-n,m},\widehat{\text{d}}_{-n-2,m}\mid u\in\mathfrak{g},\ n\in\mathbb{N},\ m\in\mathbb{Z}\right\}.
\end{align}
Then both $\widehat{\mathfrak{g}}_{+}$ and $\widehat{\mathfrak{g}}_{-}$ are subalgebras of $\widehat{\mathfrak{g}}$. Furthermore, we have the
decomposition:
\begin{align}
\widehat{\mathfrak{g}}=\widehat{\mathfrak{g}}_{+}\oplus\mathbb{C}\text{k}_{0}\oplus\widehat{\mathfrak{g}}_{-}.
\end{align}

Let $\ell$ be a complex number. View $\mathbb{C}$ as a $\left(\widehat{\mathfrak{g}}_{+}+\mathbb{C}\text{k}_{0}\right)$-module
with $\widehat{\mathfrak{g}}_{+}$ acting trivially and with $\text{k}_{0}$
acting as scalar $\ell$. We form the induced $\widehat{\mathfrak{g}}$-module
\begin{align}
V_{\widehat{\mathfrak{g}}}\left(\ell,0\right)=\mathcal{U}\left(\widehat{\mathfrak{g}}\right)\otimes_{\mathcal{U}\left(\widehat{\mathfrak{g}}_{+}+\mathbb{C}\text{k}_{0}\right)}\mathbb{C}.\label{universal-va}
\end{align}
Let $\mathcal{A}$ be a vector space with a basis $\left\{ K_{n},D_{n}\mid n\in\mathbb{Z}^{\ast}\right\} $,
and set
\[
\mathcal{A}_{\mathfrak{g}}=\mathfrak{g}\oplus\mathcal{A}.
\]
Form the generating functions $a\left(z\right)$, $a\in\mathcal{A}_{\mathfrak{g}}$
in $\widehat{\mathfrak{g}}\left[\left[z,z^{-1}\right]\right]$ as
follows:
\begin{gather*}
u\left(z\right)=\sum_{n\in\mathbb{Z}}\left(t_{0}^{n}u\right)z^{-n-1},\quad D_{m}\left(z\right)=\sum_{n\in\mathbb{Z}}\widehat{\text{d}}_{n,m}z^{-n-2},\quad K_{m}\left(z\right)=\sum_{n\in\mathbb{Z}}\text{k}_{n,m}z^{-n},
\end{gather*}
for $u\in\mathfrak{g}$ and $m\in\text{\ensuremath{\mathbb{Z}^{\ast}.}}$
Set $\mathbf{1}=1\otimes1\in V_{\widehat{\mathfrak{g}}}\left(\ell,0\right)$.
Identify $\mathcal{A}_{\mathfrak{g}}$ as a subspace of $V_{\widehat{\mathfrak{g}}}\left(\ell,0\right)$
through the linear map
\begin{align}
u\mapsto\left(t_{0}^{-1}u\right)\mathbf{1},\quad K_{n}\mapsto\text{k}_{0,n}\mathbf{1},\quad D_{n}\mapsto\hat{\text{d}}_{-2,n}\mathbf{1},\label{embaginv}
\end{align}
for $u\in\mathfrak{g},\,n\in\mathbb{Z}^{\ast}$. It is proved in \cite{CLiT}
that there is a unique vertex algebra structure on $V_{\widehat{\mathfrak{g}}}\left(\ell,0\right)$
with $Y\left(a,z\right)=a\left(z\right)$ for $a\in\mathcal{A}_{\mathfrak{g}}$,
and $\mathbf{1}$ the vacuum vector.

\begin{remark}\label{rem:t0-1d0}Note that $\mathfrak{t}\left(\dot{\mathfrak{g}}\right)\oplus\widehat{\mathcal{S}}=\widehat{\mathfrak{g}}\oplus\mathbb{C}t_{0}^{-1}\text{d}_{0}$
with
\[
\left[-t_{0}^{-1}\text{d}_{0},a(z)\right]=\frac{d}{dz}a\left(z\right),\quad a\in\mathcal{A}_{\mathfrak{g}}.
\]
This implies that $V_{\widehat{\mathfrak{g}}}\left(\ell\right)$ is
a $\mathfrak{t}\left(\dot{\mathfrak{g}}\right)\oplus\widehat{\mathcal{S}}$-module
with $-t_{0}^{-1}\text{d}_{0}$ acts as the canonical derivation $\mathcal{D}$.
In particular, from \cite{LLi} it follows that an ideal of vertex
algebra $V_{\widehat{\mathfrak{g}}}\left(\ell,0\right)$ is the same
as a $\mathfrak{t}\left(\dot{\mathfrak{g}}\right)\oplus\widehat{\mathcal{S}}$-submodule.
\end{remark}

When $\ell$ is a nonnegative integer, denoted by $J_{\widehat{\mathfrak{g}}}\left(\ell,0\right)$
the $\widehat{\mathfrak{g}}$-submodule of $V_{\widehat{\mathfrak{g}}}\left(\ell,0\right)$
generated by the vectors
\begin{align}
\left(t_{0}^{-1}x_{\pm\alpha_{i}}\right){}^{\epsilon_{i}\ell+1}\mathbf{1},\quad i\in I,\label{defjbl}
\end{align}
where $\epsilon_{i}=\frac{2}{\langle\alpha_{i},\alpha_{i}\rangle}$.
It is straightforward to check that $J_{\widehat{\mathfrak{g}}}\left(\ell,0\right)$
is $t_{0}^{-1}\text{d}_{0}$-stable (cf.  \cite[Lemma 3.13]{CLiT}), and hence by Remark \ref{rem:t0-1d0}
it is an ideal of the vertex algebra $V_{\widehat{\mathfrak{g}}}\left(\ell,0\right)$.
Let
\begin{align}
L_{\widehat{\mathfrak{g}}}\left(\ell,0\right)=V_{\widehat{\mathfrak{g}}}\left(\ell,0\right)/J_{\widehat{\mathfrak{g}}}\left(\ell,0\right)
\end{align}
be a quotient vertex algebra of $V_{\widehat{\mathfrak{g}}}\left(\ell,0\right)$.

\subsection{Conformal algebra $\mathcal{C}_{\mathfrak{g}}$ }

In order to associate the twisted toroidal EALA $\widetilde{\mathfrak{g}}[\mu]$ with the vertex algebras $V_{\widehat{\mathfrak{g}}}\left(\ell,0\right)$
and $L_{\widehat{\mathfrak{g}}}\left(\ell,0\right)$,
we define a $G_{\mu}$-conformal algebra $\mathcal{C}_{\mathfrak{g}}$
such that $\widehat{\mathcal{C}}_{\mathfrak{g}}\cong\widehat{\mathfrak{g}}$,
and $\widetilde{\mathcal{C}}_{\mathfrak{g}}\left[G_{\mu}\right]\cong\widetilde{\mathfrak{g}}\left[\mu\right]$.
As a vector space, we set
\[
\mathcal{C}_{\mathfrak{g}}=\left(\mathbb{C}\left[\partial\right]\otimes\mathcal{A}_{\mathfrak{g}}\right)\oplus\mathbb{C}\text{k}_{0},
\]
and define $\partial$ to be a linear transformation on $\mathcal{C}_{\mathfrak{g}}$ such that
\[
\partial\left(\partial^{m}\otimes x\right)=\partial^{m+1}\otimes x,\quad\partial\left(\text{k}_{0}\right)=0
\]
for $m\in\mathbb{N},$ $x\in\mathcal{A}_{\mathfrak{g}}.$ Let
\begin{eqnarray*}
Y^{-}:\mathcal{C}_{\mathfrak{g}}\to\text{Hom}\left(\mathcal{C}_{\mathfrak{g}},z^{-1}\mathcal{C}_{\mathfrak{g}}\left[z^{-1}\right]\right),\quad a\mapsto\sum_{i\in\mathbb{N}}a_{i}z^{-i-1}
\end{eqnarray*}
be the unique linear map such that the  property (\ref{derivation property of Y-})
holds, and the nontrivial $i$-products on $\mathcal{A}_{\mathfrak{g}}\oplus {\mathbb C}\text{k}_{0}$
are as follows:
\begin{gather*}
\left(t_{1}^{m}\otimes u\right)_{0}\left(t_{1}^{n}\otimes v\right)=t_{1}^{m+n}\otimes\left[u,v\right]+\left\langle u,v\right\rangle m\left(\partial\otimes K_{m+n}\right)+\delta_{m+n,0}\left\langle u,v\right\rangle m\text{k}_{1},\\
\left(t_{1}^{m}\otimes u\right)_{1}\left(t_{1}^{n}\otimes v\right)=\left(m+n\right)\left\langle u,v\right\rangle K_{m+n}+\delta_{m+n,0}\left\langle u,v\right\rangle \text{k}_{0},\\
\left(D_{r}\right)_{0}\left(t_{1}^{m}\otimes u\right)=m\left(\partial\otimes\left(t_{1}^{r+m}\otimes u\right)\right),\quad\left(t_{1}^{m}\otimes u\right)_{0}\left(D_{r}\right)=r\left(\partial\otimes\left(t_{1}^{r+m}\otimes u\right)\right),\\
\left(D_{r}\right)_{1}\left(t_{1}^{m}\otimes u\right)=\left(t_{1}^{m}\otimes u\right)_{1}\left(D_{r}\right)=\left(r+m\right)t_{1}^{r+m}\otimes u,\\
\left(D_{r}\right)_{0}\left(K_{s}\right)=r\left(\partial\otimes K_{r+s}\right)+\delta_{r+s,0}r\text{k}_{1},\\
\left(K_{s}\right)_{0}\left(D_{r}\right)=s\left(\partial\otimes K_{r+s}\right)+\delta_{r+s,0}\left(-r\text{k}_{1}+\partial\otimes\text{k}_{0}\right),\\
\left(D_{r}\right)_{1}\left(K_{s}\right)=\left(K_{s}\right)_{1}\left(D_{r}\right)=\left(r+s\right)K_{r+s}+\delta_{r+s,0}\text{k}_{0},\\
\left(\text{d}_{1}\right)_{0}\left(t_{1}^{m}\otimes u\right)=-\left(t_{1}^{m}\otimes u\right)_{0}\left(\text{d}_{1}\right)=mt_{1}^{m}\otimes u,\quad\left(\text{d}_{1}\right)_{1}\left(\text{k}_{1}\right)=\left(\text{k}_{1}\right)_{1}\left(\text{d}_{1}\right)=\text{k}_{0},\\
\left(\text{d}_{1}\right)_{0}\left(K_{r}\right)=-\left(K_{r}\right)_{0}\left(\text{d}_{1}\right)=rK_{r},\quad\left(\text{d}_{1}\right)_{0}\left(D_{r}\right)=-\left(D_{r}\right)_{0}\left(\text{d}_{1}\right)=D_{r},\\
\left(D_{r}\right)_{0}\left(D_{s}\right)=r\partial\otimes D_{r+s}+\delta_{r+s,0}\left(-r\partial^{2}\otimes\text{d}_{1}\right),\quad\left(D_{r}\right)_{1}\left(D_{s}\right)=\left(r+s\right)D_{r+s},
\end{gather*}
where $u,v\in\dot{\mathfrak{g}},$ $m,n\in\mathbb{Z},$ $r,s\in\mathbb{Z}^{\ast}$,
and we have used the convention $K_{0}=0=D_{0}$.

\begin{proposition}\label{prop:conalgcg} The vector space $\mathcal{C}_{\mathfrak{g}}$,
together with the linear maps $\partial$ and $Y^{-},$ as defined
above, is a conformal algebra, and the linear map $\hat{i}_{\mathfrak{g}}:\widehat{\mathcal{C}}_{\mathfrak{g}}\to\widehat{\mathfrak{g}}$
defined by
\begin{eqnarray*}
u\left(m\right)\mapsto t_{0}^{m}u,\ K_{n}\left(m\right)\mapsto\text{k}_{m+1,n},\ D_{n}\left(m\right)\mapsto\widehat{\text{d}}_{m-1,n},\ \text{k}_{0}\left(m\right)\mapsto\delta_{m,-1}\text{k}_{0},
\end{eqnarray*}
for $u\in\mathfrak{g},$ $m\in\mathbb{Z}$ and $n\in\mathbb{Z}^{\ast}$,
is an isomorphism of Lie algebras. Furthermore, the linear map $\tilde{i}_{\mathfrak{g}}:\widetilde{\mathcal{C}}_{\mathfrak{g}}\to\widetilde{\mathfrak{g}}$
defined by
\begin{eqnarray}
u\left[m\right]\mapsto t_{0}^{m}u,\ K_{n}\left[m\right]\mapsto\text{k}_{m,n},\ D_{n}\left[m\right]\mapsto\widetilde{\text{d}}_{m,n},\ \text{k}_{0}\left[m\right]\mapsto\delta_{m,0}\text{k}_{0},\ \mathbf{d}\mapsto-\text{d}_{0}
\end{eqnarray}
for $u\in\mathfrak{g},$ $m\in\mathbb{Z}$ and $n\in\mathbb{Z}^{\ast}$,
is also an isomorphism of Lie algebras. \end{proposition}
\begin{proof}
By definition, $\widehat{\mathcal{C}}_{\mathfrak{g}}$ is linearly
spanned by the elements $a\left(m\right),\text{k}_{0}\left(m\right)$
for $a\in\mathcal{A}_{\mathfrak{g}}$, $m\in\Z$. Note that $\text{k}_{0}\left(m\right)=0$
if $m\ne-1$. Using these with \eqref{eq:basisofhatg}, it is easy
to see that the map $\hat{i}_{\mathfrak{g}}$ is an isomorphism of
vector spaces. Then $\widehat{\mathcal{C}}_{\mathfrak{g}}$ admits
a Lie algebra structure transferring from $\widehat{\mathfrak{g}}$ via the linear isomorphism $\hat{i}_{\mathfrak{g}}$.
By using the relations \eqref{eq:relationint1}, \eqref{eq:relationinhats1},
\eqref{eq:relationinhats2}, \eqref{eq:relationinhatg1}, \eqref{eq:relationinhatg2}
and the facts that $\widehat{\text{d}}_{n,0}=\left(n+1\right)t_{0}^{n}\text{d}_{1}$
and $\text{k}_{m,0}=-\frac{1}{m}t_{0}^{m}\text{k}_{1}$ for $n\in\mathbb{Z}$
and $m\in\mathbb{Z}^{\ast}$, one can check that the Lie brackets
on $\widehat{\mathcal{C}}_{\mathfrak{g}}$ coincide with that in \eqref{eq:relationinhatc}
with the $i$-products defined above. Thus, by Lemma \ref{lem:hatc},
$\left(\mathcal{C}_{\mathfrak{g}},\partial,Y_{-}\right)$ is a conformal algebra
with $\hat{i}_{\mathfrak{g}}$ a Lie algebra isomorphism.

For the second assertion of the proposition, we note that $\widetilde{\mathcal{C}}_{\mathfrak{g}}$
is linearly spanned by the elements $a\left[m\right],\text{k}_{0}\left[m\right]$,
$\mathbf{d}$ for $a\in\mathcal{A}_{\mathfrak{g}}$ and $m\in\Z$.
Also note that $\text{k}_{0}\left[n\right]=0$ if $n\ne0$. These
together with \eqref{eq:basisoftildeg} give that $\tilde{i}_{\mathfrak{g}}$
is a linear isomorphism. Furthermore, by comparing the Lie relation
\eqref{eq:relationintildec} in $\widetilde{\mathcal{C}}_{\mathfrak{g}}$
and the Lie relations \eqref{eq:relationint1},
\eqref{eq:relationinS2}, \eqref{eq:relationinS3}, and \eqref{eq:relationinS4}
in $\widetilde{\mathfrak{g}}$, it is straightforward to check that
$\tilde{i}_{\mathfrak{g}}$ is a Lie algebra homomorphism, as required.
\end{proof}
Now we define an automorphism group $G_{\mu}$ on $\mathcal{C}_{\mathfrak{g}}$
so that $\widetilde{\mathcal{C}}_{\mathfrak{g}}\left[G_{\mu}\right]\cong\widetilde{\mathfrak{g}}\left[\mu\right]$.
We first define a linear transformation $R_{\mu}$ on $\mathcal{C}_{\mathfrak{g}}$
by
\begin{eqnarray*}
 &  & R_{\mu}\left(\partial^{m}\otimes x\right)=\partial^{m}\otimes\mu\left(x\right),\ R_{\mu}\left(\text{k}_{0}\right)=\text{k}_{0},\ R_{\mu}\left(\partial^{m}\otimes K_{n}\right)=\partial^{m}\otimes K_{n},\\
 &  & R_{\mu}\left(\partial^{m}\otimes\left(t_{1}^{n}\otimes h\right)\right)=\partial^{m}\otimes\left(t_{1}^{n}\otimes\dot{\mu}\left(h\right)\right)+\rho_{\mu}\left(h\right)\partial^{m+1}\otimes K_{n},\\
 &  & R_{\mu}\left(\partial^{m}\otimes D_{n}\right)=\partial^{m}\otimes D_{n}-\partial^{m+1}\otimes\left(t_{1}^{n}\otimes\boldsymbol{h}\right)+\frac{\left\langle \boldsymbol{h},\boldsymbol{h}\right\rangle }{2}\partial^{m+2}\otimes K_{n},
\end{eqnarray*}
for $x\in\mathfrak{g}_{\alpha}$ with $\alpha\in\Delta^{\times}\cup\left\{ 0\right\} $,
$h\in\dot{\mathfrak{h}},$ $m\in\mathbb{Z}$ and $n\in\mathbb{Z}^{\ast}$.

\begin{lemma}The linear transformation $R_{\mu}$, as defined above,
is an automorphism of $\mathcal{C}_{\mathfrak{g}}$ with order $T$.
\end{lemma}
\begin{proof}
Note that $R_{\mu}\circ\partial=\partial\circ R_{\mu}$ on $\mathcal{C}_{\mathfrak{g}}$.
Then we have a linear map $\widehat{R}_{\mu}:\widehat{\mathcal{C}}_{\mathfrak{g}}\rightarrow\widehat{\mathcal{C}}_{\mathfrak{g}}$
determined by \eqref{eq:defhatvarphi}. Via the isomorphism $\hat{i}_{\mathfrak{g}}:\widehat{\mathcal{C}}_{\mathfrak{g}}\rightarrow\widehat{\mathfrak{g}}$
given in Proposition \ref{prop:conalgcg}, $\widehat{R}_{\mu}$ induces
a linear map on $\widehat{\mathfrak{g}}$ determined by
\begin{equation}
\begin{split}\label{eq:defhatmu}\widehat{R}_{\mu}\left(t_{0}^{m}x\right)=t_{0}^{m}\mu\left(x\right),\quad\widehat{R}_{\mu}\left(t_{0}^{m}t_{1}^{m}\otimes h\right)=t_{0}^{m}t_{1}^{m}\otimes\dot{\mu}\left(h\right)-\rho_{\mu}\left(h\right)m\text{k}_{m,n},\\
\widehat{R}_{\mu}\left(\text{k}_{m,n}\right)=\text{k}_{m,n},\quad\widehat{R}_{\mu}\left(\text{d}_{m,n}\right)
=\text{d}_{m,n}+\left(m+1\right)t_{0}^{m}t_{1}^{n}\otimes\boldsymbol{h}+\frac{\left\langle \boldsymbol{h},\boldsymbol{h}\right\rangle }{2}\left(m+1\right)m\text{k}_{m,n},
\end{split}
\end{equation}
for $x\in\mathfrak{g}_{\alpha}$ with $\alpha\in\Delta^{\times}\cup\left\{ 0\right\} ,$
$h\in\dot{\mathfrak{h}},$ $m\in\mathbb{Z}$ and $n\in\mathbb{Z}^{\ast}.$
Moreover, from Proposition \ref{prop:defhatmu} we have a
Lie automorphism $\hat{\mu}$ of $\mathcal{T}\left(\dot{\mathfrak{g}}\right)$,
which preserves $\widehat{\mathfrak{g}}$. And one can check
that $\hat{\mu}=\widehat{R}_{\mu}$ on $\widehat{\mathfrak{g}}$.
Thus $\widehat{R}_{\mu}$ is an automorphism of the Lie algebra $\widehat{\mathcal{C}}_{\mathfrak{g}}$
with order $T$. The assertion of the lemma then follows from Lemma \ref{lem:conformalaut}.
\end{proof}
Set $G_{\mu}=\langle R_{\mu}\rangle$, an automorphism group of $\mathcal{C}_{\mathfrak{g}}$,
and let $\chi_{\omega}$ be the linear character of $G_{\mu}$ defined
by $\chi_{\omega}\left(R_{\mu}\right)=\omega^{-1}$. Associated to
the $G_{\mu}$-conformal algebra $\mathcal{C}_{\mathfrak{g}}$ and
the character $\chi_{\omega}$, there is a Lie algebras $\widetilde{\mathcal{C}}_{\mathfrak{g}}\left[G_{\mu}\right]$
by \eqref{eq:deftcg}.\textcolor{red}{{} }Recalling the surjective map
$\eta_{\mu}:\widetilde{\mathfrak{g}}\rightarrow\widetilde{\mathfrak{g}}\left[\mu\right]$
defined in \eqref{eq:defetamu}, we have:

\begin{proposition}\label{prop:tgisotcg} The following assignment
\begin{eqnarray}
\overline{u\left[m\right]}\mapsto\eta_{\mu}\left(t_{0}^{m}u\right),\ \overline{K_{n}\left[m\right]}\mapsto\eta_{\mu}\left(\text{k}_{m,n}\right),\ \overline{D_{n}\left[m\right]}\mapsto\eta_{\mu}\left(\widetilde{\text{d}}_{m,n}\right),\ \overline{\text{k}_{0}\left[m\right]}\mapsto\delta_{m,0}T\text{k}_{0},\ \mathbf{d}\mapsto-\text{d}_{0}
\end{eqnarray}
for $u\in\mathfrak{g},$ $m\in\mathbb{Z}$ and $n\in\mathbb{Z}^{\ast}$,
determines an isomorphism from the Lie algebra $\widetilde{\mathcal{C}}_{\mathfrak{g}}\left[G_{\mu}\right]$
to the Lie algebra $\widetilde{\mathfrak{g}}$$\left[\mu\right]$. \end{proposition}
\begin{proof}
Corresponding to the automorphism $R_{\mu}$ of ${\mathcal{C}}_{\mathfrak{g}}$, there is an automorphism $\widetilde{R}_{\mu}$
of $\widetilde{\mathcal{C}}_{\mathfrak{g}}$ (see \eqref{eq:deftildeg}).
Via the isomorphism $\tilde{i}_{\mathfrak{g}}:\widetilde{\mathcal{C}}_{\mathfrak{g}}\rightarrow\widetilde{\mathfrak{g}}$
given in Proposition \ref{prop:conalgcg}, $\widetilde{R}_{\mu}$
induces an automorphism of $\widetilde{\mathfrak{g}}$. It is straightforward
to check that this automorphism of $\widetilde{\mathfrak{g}}$ coincides
with $\tilde{\mu}$ (see \eqref{eq:deftildemu}). Since $G_{\mu}$
is a cyclic group of order $T$, it follows from Remark \ref{rem:isoofcovandfix}
that $\widetilde{\mathcal{C}}_{\mathfrak{g}}\left[G_{\mu}\right]$
is isomorphic to the subalgebra of $\widetilde{\mathcal{C}}_{\mathfrak{g}}$
fixed by $\widetilde{R}_{\mu}$. This implies that $\widetilde{\mathcal{C}}_{\mathfrak{g}}\left[G_{\mu}\right]$
is isomorphic to $\widetilde{\mathfrak{g}}[\mu]$ with the isomorphism
given in the proposition.
\end{proof}

\subsection{The correspondence theorem for $\widetilde{\mathfrak{g}}\left[\mu\right]$}

In this subsection we generalize the correspondence theorem (Theorem
\ref{thm:main1}) for affine Kac-Moody algebras to the nullity $2$
twisted toroidal extended affine Lie algebras $\widetilde{\mathfrak{g}}\left[\mu\right]$.

Form the following generating functions $a^{\mu}\left[z\right]$,
$a\in\mathcal{A}_{\mathfrak{g}}$ in $\widetilde{\mathfrak{g}}\left[\mu\right]\left[\left[z,z^{-1}\right]\right]$:
\begin{gather*}
u^{\mu}\left[z\right]=\sum_{n\in\mathbb{Z}}\eta_{\mu}\left(t_{0}^{n}u\right)z^{-n},\quad D_{m}^{\mu}\left[z\right]=\sum_{n\in\mathbb{Z}}\eta_{\mu}\left(\widetilde{\text{d}}_{n,m}\right)z^{-n},\quad K_{m}^{\mu}\left[z\right]=\sum_{n\in\mathbb{Z}}\eta_{\mu}\left(\text{k}_{n,m}\right)z^{-n},
\end{gather*}
for $u\in\mathfrak{g}$ and $m\in\text{\ensuremath{\mathbb{Z}^{\ast}.}}$
Note that all the components of these generating functions together
with $\text{k}_{0}$, $\text{d}_{0}$ span the algebra $\widetilde{\mathfrak{g}}\left[\mu\right]$.
As in the affine Kac-Moody algebra case, we formulate the following definition.

\begin{definition}\label{def:resintgmumod} We say that a $\widetilde{\mathfrak{g}}\left[\mu\right]$-module
$W$ is\emph{ restricted} if for any $a\in\mathcal{A}_{\mathfrak{g}}$,
$a^{\mu}\left[z\right]\in\mathrm{Hom}\left(W,W\left(\left(z\right)\right)\right)$. And $W$
 is said \emph{ of level $\ell\in\mathbb{C}$} if the central element
$\text{k}_{0}$ acts as the scalar $\ell/T$. Furthermore, if $\mu$
is nontransitive, we say that $W$ is \emph{integrable} if for any
$\alpha\in\tilde{\Delta}_{\mu}^{\times}$, $\widetilde{\mathfrak{g}}\left[\mu\right]{}_{\alpha}$
acts locally nilpotent on $W$. \end{definition}

For each $i\in I=\{0,1,\cdots, l\}$, one recalls the positive integer $T_{i},s_{i}$ defined
in \eqref{eq:defsi}, and set
\[
p_{i}\left(z\right)=\frac{1-z^{s_iT_{i}}}{1-z^{T_{i}}}.
\]
Then we have the following analogue of Proposition \ref{prop:charintaffmod}.

\begin{proposition}\label{prop:charinttgmumod} Assume that $\mu$
is nontransitive. Then for any $i\in I$,
\begin{align}
p_{i}\left(z_{1}/z_{2}\right)\left[x_{\pm\alpha_{i}}^{\mu}\left[z_{1}\right],x_{\pm\alpha_{i}}^{\mu}\left[z_{2}\right]\right]=0.\label{eq:pizxx=00003D00003D00003D0}
\end{align}
Furthermore, if $W$ is a restricted $\widetilde{\mathfrak{g}}\left[\mu\right]$-module
of level $\ell$, then $W$ is integrable if and only if $\ell$
is a nonnegative integer and for any $i\in I$,
\begin{align}
\left(\prod_{1\le i<j\le\epsilon_{i}\ell+1}p_{i}\left(z_{i}/z_{j}\right)\right)x_{\pm\alpha_{i}}^{\mu}\left[z_{1}\right]x_{\pm\alpha_{i}}^{\mu}\left[z_{2}\right]\cdots x_{\pm\alpha_{i}}^{\mu}\left[z_{\epsilon_{i}\ell+1}\right]|_{z_{1}=z_{2}\cdots=z_{\epsilon_{i}\ell+1}}=0\quad\text{on}\ W.\label{eq:charinttgmumod}
\end{align}
\end{proposition}
\begin{proof}
For each $i\in I$, denoted by $\widetilde{\mathfrak{g}}\left[\mu\right]{}_{i}$
the subalgebra of $\widetilde{\mathfrak{g}}\left[\mu\right]$ generated
by the elements $t_{0}^{m}x_{\pm\alpha_{i}}$, $\text{d}_{0}$ for
$m\in\mathbb{Z}$. We first show that $\widetilde{\mathfrak{g}}\left[\mu\right]{}_{i}$
is isomorphic to the affine Kac-Moody algebra of type $A_{s_{i}}^{(s_{i})}$.
For $k=1,2$, we denote by $\theta_{k}$ the order $k$ diagram automorphism
of the simple Lie algebra $\mathfrak{sl}_{k+1}$.
Then for each $i\in I$, we have an affine Kac-Moody algebra $\widetilde{\mathcal{L}}\left(\mathfrak{sl}_{s_{i}+1},\theta_{s_{i}}\right)$
of type $A_{s_{i}}^{(s_{i})}$ (see Section 3.2). Recall that $\eta_{\mu}\left(t_{0}^{m}x_{\pm\alpha_{i}}\right)=\sum_{p=0}^{T-1}t_{0}^{m}x_{\pm\alpha_{\mu(i)}}$
for $i\in I$ and $m\in\mathbb{Z}$. In particular, we have $\eta_{\mu}\left(t_{0}^{m}x_{\pm\alpha_{i}}\right)\ne0$
if and only if $m\in T_{i}\mathbb{Z}$. By Lemma \ref{lem:defsi},
it is straightforward to see that the assignment ($m\in \mathbb{Z}$ and $\beta$ a fixed simple root of $\mathfrak{sl}_{s_{i}+1}$)
\begin{align*}
\eta_{\mu}\left(t_{0}^{T_{i}m}x_{\pm\alpha_{i}}\right)\mapsto t^{m}\otimes\left(x_{\pm\beta}\right){}_{(m)},\quad\frac{T}{s_{i}}\text{k}_{0}\mapsto\text{k},\quad\text{d}_{0}\mapsto T_{i}\text{d}
\end{align*}
determines an isomorphism from the Lie algebra $\widetilde{\mathfrak{g}}\left[\mu\right]{}_{i}$
to the Lie algebra $\widetilde{\mathcal{L}}\left(\mathfrak{sl}_{s_{i}+1},\theta_{s_{i}}\right)$.
This together with \eqref{eq:charintaffmod1} implies the first assertion
of the proposition.

For the second assertion, let $W$ be a restricted $\widetilde{\mathfrak{g}}\left[\mu\right]$-module
of level $\ell$. For each $i\in I$, via the isomorphism $\widetilde{\mathfrak{g}}\left[\mu\right]{}_{i}\cong\widetilde{\mathcal{L}}\left(\mathfrak{sl}_{s_{i}+1},\theta_{s_{i}}\right)$,
$W$ becomes a restricted $\widetilde{\mathcal{L}}\left(\mathfrak{sl}_{s_{i}+1},\theta_{i}\right)$-module
of level $\ell$. Note that the $\widetilde{\mathfrak{g}}\left[\mu\right]$-module
$W$ is integrable if and only if the elements $t_{0}^{m}x_{\pm\alpha_{i}}$,
$i\in I,m\in\mathbb{Z}$ act locally nilpotent (as they generate the
core $\mathfrak{t}\left(\dot{\mathfrak{g}},\mu\right)$). Thus, $W$
is integrable if and only if $W$ is integrable as an $\widetilde{\mathcal{L}}\left(\mathfrak{sl}_{s_{i}+1},\theta_{i}\right)$-module
for all $i\in I$. Thus the assertion follows from Proposition \ref{prop:charintaffmod}.
\end{proof}

Note that $\text{k}_0$ is a central element in $\mathcal{C}_{\mathfrak{g}}$.
Thus, for any $\ell\in \mathbb{C}$, the $\widehat{\mathcal{C}}_{\mathfrak{g}}$-submodule $\langle\text{k}_{0}-\ell\rangle$ of
$V_{\mathcal{C}_{\mathfrak{g}}}$ generated by $\text{k}_{0}-\ell$ is an ideal of $V_{\mathcal{C}_{\mathfrak{g}}}$ (as a vertex algebra).
Recall the isomorphism $\hat{i}_{\mathfrak{g}}:\widehat{\mathcal{C}}_{\mathfrak{g}}\to\widehat{\mathfrak{g}}$
given in Proposition \ref{prop:conalgcg}. One can readily check that
$\hat{i}_{\mathfrak{g}}\left(\widehat{\mathcal{C}}_{\mathfrak{g}}^{+}\right)=\widehat{\mathfrak{g}}_{+}$
and $\hat{i}_{\mathfrak{g}}\left(\widehat{\mathcal{C}}_{\mathfrak{g}}^{-}\right)=\widehat{\mathfrak{g}}{}_{-}\oplus \mathbb{C}\text{k}_0$
(see \eqref{eq:polardecofC} and \eqref{eq:decofhatg}). This implies
that $V_{\widehat{\mathfrak{g}}}\left(\ell,0\right)$
is isomorphic to the quotient vertex algebra $V_{\mathcal{C}_{\mathfrak{g}}}/\langle\text{k}_{0}-\ell\rangle$.
 Recall also that the automorphism group $G_{\mu}=\langle R_{\mu}\rangle$
of $\mathcal{C}_{\mathfrak{g}}$ can be uniquely lifted to an automorphism
group of its universal enveloping vertex algebra $V_{\mathcal{C}_{\mathfrak{g}}}$.
As $R_{\mu}\left(\text{k}_{0}\right)=\text{k}_{0}$, $G_{\mu}$ is
naturally an automorphism group of $V_{\widehat{\mathfrak{g}}}\left(\ell,0\right)$.
Furthermore, we have

\begin{lemma}For each nonnegative integer $\ell$, $J_{\widehat{\mathfrak{g}}}\left(\ell,0\right)$
is a $R_{\mu}$-stable ideal of $V_{\widehat{\mathfrak{g}}}\left(\ell,0\right)$.
\end{lemma}
\begin{proof}
The assertion follows from the fact that
\[
R_{\mu}\left(\left(t_{0}^{-1}x_{\pm\alpha_{i}}\right){}^{\epsilon_{i}\ell+1}\mathbf{1}\right)=\left(t_{0}^{-1}\mu\left(x_{\pm\alpha_{i}}\right)\right){}^{\epsilon_{i}\ell+1}\mathbf{1}=\left(t_{0}^{-1}x_{\pm\alpha_{\mu(i)}}\right){}^{\epsilon_{\mu(i)}\ell+1}\mathbf{1}\in J(\ell),
\]
where $\epsilon_{\mu(i)}=\epsilon_{i}$ for $i\in I$.
\end{proof}
In view of the above lemma, $G_{\mu}$ is also an automorphism group
of the vertex algebra $L_{\widehat{\mathfrak{g}}}\left(\ell,0\right)$.
Now we state one of the main results of the paper.

\begin{theorem}\label{thm:main2} Let $\ell$ be a complex number.
For any restricted $\widetilde{\mathfrak{g}}\left[\mu\right]$-module
$W$ of level $\ell$, there is a $\left(G_{\mu},\chi_{\omega}\right)$-equivariant
$\phi$-coordinated quasi $V_{\widehat{\mathfrak{g}}}\left(\ell,0\right)$-module
structure $(Y_{W}^{\phi},d)$ on $W$, which is uniquely determined
by
\begin{align*}
d=-\text{d}_{0}, \ \quad Y_{W}^{\phi}\left(a,z\right)=a^{\mu}\left[z\right],
\end{align*}
for $ a\in\mathcal{A}_{\mathfrak{g}}.$
On the other hand, for any $\left(G_{\mu},\chi_{\omega}\right)$-equivariant
$\phi$-coordinated quasi $V_{\widehat{\mathfrak{g}}}\left(\ell,0\right)$-module
$\left(W,Y_{W}^{\phi},d\right)$, $W$ is a restricted $\widetilde{\mathfrak{g}}\left[\mu\right]$-module
of level $\ell$ with action given by
\begin{align*}
\text{d}_{0}=-d, \ \quad a^{\mu}\left[z\right]=Y_{W}^{\phi}\left(a,z\right),
\end{align*}
for $ a\in\mathcal{A}_{\mathfrak{g}}$.
Furthermore, if $\ell$ is a nonnegative integer and $\mu$ is nontransitive,
then the integrable restricted $\widetilde{\mathfrak{g}}\left[\mu\right]$-modules
of level $\ell$ are exactly the $\left(G_{\mu},\chi_{\omega}\right)$-equivariant
$\phi$-coordinated quasi $L_{\widehat{\mathfrak{g}}}\left(\ell,0\right)$-modules
$(W,Y_{W}^{\phi},d)$. \end{theorem}
\begin{proof}
By Proposition \ref{prop:modnil} and Proposition \ref{prop:tgisotcg},
the restricted $\widetilde{\mathfrak{g}}\left[\mu\right]$-modules
are exactly the $\left(G_{\mu},\chi_{\omega}\right)$-equivariant
$\phi$-coordinated quasi $V_{\mathcal{C}_{\mathfrak{g}}}$-modules.
Thus the fact that $V_{\widehat{\mathfrak{g}}}\left(\ell,0\right)\cong V_{\mathcal{C}_{\mathfrak{g}}}/\left\langle \text{k}_{0}-\ell\right\rangle $
implies the first assertion of the theorem.

For the second part of the theorem, we assume that $\ell$ is a nonnegative integer, $\mu$ is nontransitive and $W$ is an integrable
restricted $\widetilde{\mathfrak{g}}\left[\mu\right]$-module of level
$\ell$. Note that for each
$i$, we have $\left[x_{\pm\alpha_{i}}\left(z_{1}\right),x_{\pm\alpha_{i}}\left(z_{2}\right)\right]=0$
on $\widehat{\mathfrak{g}}$. This implies that for any $n\in\mathbb{N}$
and $i\in I$,
\begin{align}
\left(x_{\pm\alpha_{i}}{}\right)_{n}\left(x_{\pm\alpha_{i}}\right)=0\quad\text{on}\quad V_{\widehat{\mathfrak{g}}}\left(\ell,0\right).\label{eq:xalinxali=00003D00003D00003D0}
\end{align}

Viewing $W$ as a faithful $\left(G_{\mu},\chi_{\omega}\right)$-equivariant
$\phi$-coordinated quasi $V_{\widehat{\mathfrak{g}}}\left(\ell,0\right)/\ker Y_{W}^{\phi}$-module,
by Proposition \ref{prop:charinttgmumod}, we have
\begin{align}
\left(\prod_{1\le i<j\le\epsilon_{i}\ell+1}p_{i}(z_{i}/z_{j})\right)Y_{W}^{\phi}\left(x_{\pm\alpha_{i}},z_{1}\right)Y_{W}^{\phi}\left(x_{\pm\alpha_{i}},z_{2}\right)\cdots Y_{W}^{\phi}\left(x_{\pm\alpha_{i}},z_{\epsilon_{i}\ell+1}\right)|_{z_{1}=z_{2}\cdots=z_{\epsilon_{i}\ell+1}}=0\quad\text{on}\ W\label{eq:main1eq1}
\end{align}
for all $i\in I$. This together with \eqref{eq:xalinxali=00003D00003D00003D0}
and Proposition \ref{prop:modnil} proves that $\left(\left(x_{\pm\alpha_{i}}\right){}_{-1}\right){}^{\epsilon_{i}\ell+1}\mathbf{1}\in\ker Y_{W}^{\phi}$
for $i\in I$. Thus, we have $J_{\widehat{\mathfrak{g}}}\left(\ell,0\right)\subset\ker Y_{W}^{\phi}$
and $W$ becomes a $\left(G_{\mu},\chi_{\omega}\right)$-equivariant
$\phi$-coordinated quasi $L_{\widehat{\mathfrak{g}}}\left(\ell,0\right)$-module.

Conversely, let $\left(W,Y_{W}^{\phi},d\right)$ be a $\left(G_{\mu},\chi_{\omega}\right)$-equivariant
$\phi$-coordinated quasi $L_{\widehat{\mathfrak{g}}}\left(\ell,0\right)$-module.
Then it is also a $\left(G_{\mu},\chi_{\omega}\right)$-equivariant $\phi$-coordinated
quasi $V_{\widehat{\mathfrak{g}}}\left(\ell,0\right)$-module, and such
that for any $i\in I$, $\left(\left(x_{\pm\alpha_{i}}\right){}_{-1}\right){}^{\epsilon_{i}\ell+1}\mathbf{1}$
acts trivially on $W$. Recall from the first part of the theorem that $W$ is then a $\widetilde{\mathfrak{g}}[\mu]$-module
with $a^{\mu}\left[z\right]=Y_{W}^{\phi}\left(a,z\right)$ for $a\in\mathcal{A}_{\mathfrak{g}}$.
Combining this with \eqref{eq:pizxx=00003D00003D00003D0}, we obtain
\begin{align*}
\left(\prod_{1\le i<j\le\epsilon_{i}\ell+1}p_{i}\left(z_{i}/z_{j}\right)\right)Y_{W}^{\phi}\left(x_{\pm\alpha_{i}},z_{1}\right)Y_{W}^{\phi}\left(x_{\pm\alpha_{i}},z_{2}\right)\cdots Y_{W}^{\phi}\left(x_{\pm\alpha_{i}},z_{\epsilon_{i}\ell+1}\right)\in\mathrm{Hom}\left(W,W\left(\left(z_{1},\dots,z_{\epsilon_{i}\ell+1}\right)\right)\right).
\end{align*}
Then again by \eqref{eq:xalinxali=00003D00003D00003D0} and Proposition
\ref{prop:modnil}, we see that \eqref{eq:main1eq1} holds. This implies, viewing $W$ as a restricted $\widetilde{\mathfrak{g}}[\mu]$-module
of level $\ell$, that \eqref{eq:charinttgmumod} holds. Thus, by Proposition
\ref{prop:charinttgmumod}, $W$ is an integrable $\widetilde{\mathfrak{g}}[\mu]$-module.
This completes the proof of the theorem.
\end{proof}
\begin{remark} For any complex number $a$, two vertex algebras $V_{\widehat{\mathfrak{t}}(\mathfrak{g},a)^{o}}\left(\ell\right)$
and $L_{\widehat{\mathfrak{t}}(\mathfrak{g},a)^{o}}(\ell)$ were constructed in \cite{CLiT}, and such that
$V_{\widehat{\mathfrak{t}}\left(\mathfrak{g},0\right){}^{o}}\left(\ell\right)=V_{\widehat{\mathfrak{g}}}\left(\ell,0\right)$
and $L_{\widehat{\mathfrak{t}}(\mathfrak{g},0)^{o}}\left(\ell\right)=L_{\widehat{\mathfrak{g}}}\left(\ell,0\right)$. By a similarly argument as above, one can prove that  the Lie algebra $\widetilde{\mathfrak{g}}\left[\mu\right]{}_{\tau_{a}}$ can be associated
with the vertex algebras $V_{\widehat{\mathfrak{t}}(\mathfrak{g},a)^{o}}\left(\ell\right)$
and $L_{\widehat{\mathfrak{t}}(\mathfrak{g},a)^{o}}\left(\ell\right)$
via their equivariant $\phi$-coordinated quasi modules, where $\widetilde{\mathfrak{g}}\left[\mu\right]{}_{\tau_{a}}$ and the affine cocycle $\tau_{a}$ are defined in Remark \ref{rem:deftaua}.
\end{remark}

\section{Associating $\widetilde{\mathfrak{sl}}_{N}\left(\mathbb{C}_{q}\right)$
with vertex algebras}

Let $N\ge2$ be a positive integer and $q\in {\mathbb{C}}^{\ast}$ a generic complex number.
In this section we prove an analog of Theorem \ref{thm:main2} for
the extended affine Lie algebra $\widetilde{\mathfrak{sl}}_{N}\left(\mathbb{C}_{q}\right)$.

First, we define the following generating functions in $\widetilde{\mathfrak{sl}}_{N}\left(\mathbb{C}_{q}\right)[[z,z^{-1}]]$:
\begin{align*}
\left(E_{i,j}t_{1}^{m}\right)\left[z\right]&=\sum_{n\in\Z}\left(E_{i,j}t_{0}^{n}t_{1}^{m}\right)z^{-n},\quad H_{k}\left[z\right]=\sum_{n\in\mathbb{Z}}\left(E_{k,k}-E_{k+1,k+1}\right)t_{0}^{n}z^{-n},\\
H_{N}\left[z\right]&=\sum_{n\in\mathbb{Z}}\left(E_{N,N}t_{0}^{n}-q^{-n}E_{1,1}t_{0}^{n}-\delta_{n,0}{\text{k}}_{1}\right)z^{-n},
\end{align*}
where $1\le i,j\le N,\ m\in\mathbb{Z}$ with $\left(i-j,m\right)\ne\left(0,0\right)$
and $1\le k\le N-1$.
Note that all the coefficients of these generating
functions, together with $\text{k}_{0}$, $\text{d}_{0}$ and $\text{d}_{1}$,
form a basis of $\widetilde{\mathfrak{sl}}_{N}\left(\mathbb{C}_{q}\right)$.

\begin{definition}\label{def:resintslcqmod} We say that
an $\widetilde{\mathfrak{sl}}_{N}\left(\mathbb{C}_{q}\right)$-module
$W$ is \emph{restricted} if $\left(E_{i,j}t_{1}^{m}\right)\left[z\right]$, $ H_{k}\left[z\right]\in\mathrm{Hom}\left(W,W\left(\left(z\right)\right)\right)$ for $1\le i,j,k\le N,\ m\in\mathbb{Z}$
with $\left(i-j,m\right)\ne(0,0)$.
$W$ is said to be \emph{of level }$\ell\in\mathbb{C}$ if the central element
$\text{k}_{0}$ acts as the scalar $\ell$. Furthermore, we say that
$W$ is \emph{integrable} if
$E_{i,j}t_{0}^{n}t_{1}^{m}$ acts locally nilpotent on $W$ for $1\le i\ne j\le N,\ m,n\in\mathbb{Z}$. \end{definition}

The following result is from \cite[Proposition 3.13]{CLiTW}.

\begin{proposition}\label{prop:charintslmod} Let $W$ be a restricted
$\widetilde{\mathfrak{sl}}_{N}\left(\mathbb{C}_{q}\right)$-module
of level $\ell$. Then $W$ is integrable if and only if $\ell$ is
a nonnegative integer and
\[
\left(E_{i,j}t_{1}^{m}\right)\left[z\right]^{\ell+1}=0\ \ \text{on}\ W
\]
for $1\le i\ne j\le N$ and $m\in\Z$. \end{proposition}

Let $\mathfrak{gl}_{\infty}$ be the algebra of all doubly infinite
complex matrices\textcolor{red}{{} }with only finitely many nonzero
entries. For $m,n\in\mathbb{Z}$, let $E_{m,n}$ denote the unit matrix
whose only nonzero entry is the $(m,n)$-entry which is equal to $1$. Equip
$\mathfrak{gl}_{\infty}$ with a nondegenerate, invariant and symmetric bilinear form $\langle\cdot,\cdot\rangle$
defined by
\begin{align*}
\left\langle E_{i,j},E_{k,l}\right\rangle =\delta_{j,k}\,\delta_{i,l},
\end{align*}
for $ i,j,k,l\in\mathbb{Z}.$
 Let $\mathfrak{sl}_{\infty}=\left[\mathfrak{gl}_{\infty},\mathfrak{gl}_{\infty}\right]$ be
the derived subalgebra of $\mathfrak{gl}_{\infty}$. Then $\left\langle \cdot,\cdot\right\rangle $
is also nondegenerate on $\mathfrak{sl}_{\infty}$. And
associated to the pair $\left(\mathfrak{sl}_{\infty},\langle\cdot,\cdot\rangle\right)$,
we have the corresponding affine Lie algebra $\widehat{\mathcal{L}}\left(\mathfrak{sl}_{\infty}\right)$,
the universal affine vertex algebra $V_{\widehat{\mathcal{L}}\left(\mathfrak{sl}_{\infty}\right)}\left(\ell,0\right)$,
and the simple affine vertex algebra $L_{\widehat{\mathcal{L}}\left(\mathfrak{sl}_{\infty}\right)}\left(\ell,0\right)
=V_{\widehat{\mathcal{L}}\left(\mathfrak{sl}_{\infty}\right)}\left(\ell,0\right)/J_{\widehat{\mathcal{L}}\left(\mathfrak{sl}_{\infty}\right)}
\left(\ell,0\right)$ (see the subsection 3.2).
The following result is given in \cite[Lemma 3.11]{CLiTW}.

\begin{lemma}\label{lem:genofjslinfty} If $\ell$ is a nonnegative
integer, then $J_{\widehat{\mathcal{L}}\left(\mathfrak{sl}_{\infty}\right)}\left(\ell,0\right)$,
as  $\widehat{\mathcal{L}}\left(\mathfrak{sl}_{\infty}\right)$-module,
is generated by the set of vectors
\begin{align}
\{\left(t^{-1}\otimes E_{mN+i,nN+j}\right){}^{\ell+1}{\mathbf{1}}\mid \ \text{for}\ 1\le i\ne j\le N,\ m,n\in\Z\}.\label{genofj}
\end{align}
\end{lemma}

Let $\sigma_{N}$ be the automorphism of the algebra $\mathfrak{gl}_{\infty}$
defined by
\begin{align}
\sigma_{N}\left(E_{m,n}\right)=E_{m+N,n+N},
\end{align}
for $ m,n\in\mathbb{Z}$.
Restrict $\sigma_{N}$ to the subalgebra $\mathfrak{sl}_{\infty}$, we see that $\sigma_{N}$
is also an automorphism of $\mathfrak{sl}_{\infty}$ that preserves the
bilinear form $\left\langle \cdot,\cdot\right\rangle $. Denoted by
$G_{N}=\langle\sigma_{N}\rangle$ the automorphism group of $\mathfrak{sl}_{\infty}$
generated by $\sigma_{N}$. As pointed out in Section 3.2, $G_{N}$ can be
extended uniquely to an automorphism group of the vertex algebras $V_{\widehat{\mathcal{L}}\left(\mathfrak{sl}_{\infty}\right)}\left(\ell,0\right)$
and $L_{\widehat{\mathcal{L}}\left(\mathfrak{sl}_{\infty}\right)}\left(\ell,0\right)$.
Let $\chi_{q}:G_{N}\rightarrow\mathbb{C}^{\times}$ be the linear
character defined by $\chi_{q}\left(\sigma_{N}^{n}\right)=q^{n}$
for $n\in\mathbb{Z}$.

Define a $\Z$-grading $\mathfrak{gl}_{\infty}=\oplus_{n\in\mathbb{Z}}\mathfrak{gl}_{\infty(n)}$ on $\mathfrak{gl}_{\infty}$
 by assigning
\begin{align}
\deg E_{mN+i,nN+j}=n-m, \label{eq:defpingrading}
\end{align}
for $ m,n\in\mathbb{Z},\ 1\le i,j\le N$.
Note that $\mathfrak{sl}_{\infty}$ is a graded subalgebra of $\mathfrak{gl}_{\infty}$.
Denoted by $\mathcal{P}$ the derivation of $\mathfrak{sl}_{\infty}$
defined by
$\mathcal{P}(a)=na$ if $a\in\mathfrak{sl}_{\infty}$ with $\deg a=n$.
Note that for $a,b\in\mathfrak{sl}_{\infty}$, one has $\langle\mathcal{P}a,b\rangle+\langle a,\mathcal{P}b\rangle=0$.
This allows us to lift the  derivation $\mathcal{P}$
of $\mathfrak{sl}_{\infty}$ to be a derivation of the affine Lie algebra $\widehat{\mathcal{L}}\left(\mathfrak{sl}_{\infty}\right)$
with
\begin{align}
\mathcal{P}\left(\text{k}\right)=0, \ \quad\mathcal{P}\left(t^{n}\otimes a\right)=t^{n}\otimes\mathcal{P}(a),
\end{align}
for $ n\in\mathbb{Z},a\in\mathfrak{sl}_{\infty}$.
As $\mathcal{P}\left(t^{-1}\mathbb{C}\left[t^{-1}\right]\otimes\mathfrak{sl}_{\infty}\right)\subset t^{-1}\mathbb{C}\left[t^{-1}\right]\otimes\mathfrak{sl}_{\infty}$,
$\mathcal{P}$ is also a derivation of the associative algebra  $\mathcal{U}\left(t^{-1}\mathbb{C}\left[t^{-1}\right]\otimes\mathfrak{sl}_{\infty}\right)$.
Via the isomorphism $\mathcal{U}\left(t^{-1}\mathbb{C}\left[t^{-1}\right]\otimes\mathfrak{sl}_{\infty}\right)\cong V_{\widehat{\mathcal{L}}\left(\mathfrak{sl}_{\infty}\right)}\left(\ell,0\right)$,
$\mathcal{P}$ becomes a derivation of $V_{\widehat{\mathcal{L}}\left(\mathfrak{sl}_{\infty}\right)}\left(\ell,0\right)$
(as a vertex algebra). Furthermore, if $\ell$ is a nonnegative integer, then by
Lemma \ref{lem:genofjslinfty}, we see that $\mathcal{P}$ also preserves
the submodule $J_{\widehat{\mathcal{L}}\left(\mathfrak{sl}_{\infty}\right)}\left(\ell,0\right)$.
Therefore, it descends to a derivation of $L_{\widehat{\mathcal{L}}\left(\mathfrak{sl}_{\infty}\right)}\left(\ell,0\right)$.


\begin{definition} A \textit{$\left(G_{N},\chi_{q}\right)$-equivariant
$\phi$-coordinated quasi $V_{\widehat{\mathcal{L}}\left(\mathfrak{sl}_{\infty}\right)}\left(\ell,0\right)$-module}
$\left(W,Y_{W}^{\phi},d,p\right)$ is a $\left(G_{N},\chi_{q}\right)$-equivariant
$\phi$-coordinated quasi module $\left(W,Y_{W}^{\phi},d\right)$
equipped with a linear transformation $p$ on $W$ such that
$$\left[p,Y_{W}^{\phi}(v,z)\right]
=Y_{W}^{\phi}\left(\mathcal{P}v,z\right),
$$
for $v\in V_{\widehat{\mathcal{L}}\left(\mathfrak{sl}_{\infty}\right)}\left(\ell,0\right)$.
Similarly, when $\ell$ is a nonnegative integer, we can define the
notion of $\left(G_{N},\chi_{q}\right)$-equivariant $\phi$-coordinated
quasi $L_{\widehat{\mathcal{L}}\left(\mathfrak{sl}_{\infty}\right)}\left(\ell,0\right)$-module
$\left(W,Y_{W}^{\phi},d,p\right)$. \end{definition}


\begin{theorem}\label{thm:main3} Let $\ell$ be a complex number. If $W$ is a restricted $\widetilde{\mathfrak{sl}}_{N}\left(\mathbb{C}_{q}\right)$-module
of level $\ell$, then there is a $\left(G_{N},\chi_{q}\right)$-equivariant
$\phi$-coordinated quasi $V_{\widehat{\mathcal{L}}\left(\mathfrak{sl}_{\infty}\right)}\left(\ell,0\right)$-module
structure $\left(Y_{W}^{\phi},d,p\right)$ on $W$ uniquely
determined by
\begin{align*}
p=-\text{d}_{1},\ d=-\text{d}_{0},\ Y_{W}^{\phi}\left(E_{mN+i,j},z\right)=\left(E_{i,j}t_{1}^{m}\right)\left[z\right], \ \ Y_{W}^{\phi}\left(E_{k,k}-E_{k+1,k+1},z\right)=H_{k}\left[z\right]
\end{align*}
for $1\le i,j,k\le N$, $m\in\Z$ with $\left(i-j,m\right)\ne(0,0)$.
On the other hand, if $\left(W,Y_{W}^{\phi},d,p\right)$ is a $\left(G_{N},\chi_{q}\right)$-equivariant
$\phi$-coordinated quasi $V_{\widehat{\mathcal{L}}\left(\mathfrak{sl}_{\infty}\right)}\left(\ell,0\right)$-module, then $W$ is a restricted $\widetilde{\mathfrak{sl}}_{N}\left(\mathbb{C}_{q}\right)$-module
of level $\ell$ with action given by
\begin{align*}
\text{d}_{1}=-p,\ \text{d}_{0}=-d,\ \left(E_{i,j}t_{1}^{m}\right)\left[z\right]=Y_{W}^{\phi}\left(E_{mN+i,j},z\right), \ \ H_{k}\left[z\right]=Y_{W}^{\phi}\left(E_{k,k}-E_{k+1,k+1},z\right)
\end{align*}
for $1\le i,j,k\le N$, $m\in\Z$ with $\left(i-j,m\right)\ne(0,0)$.

Furthermore, if $\ell$ is a nonnegative integer, then the integrable
restricted $\widetilde{\mathfrak{sl}}_{N}\left(\mathbb{C}_{q}\right)$-modules
of level $\ell$ are exactly the $\left(G_{N},\chi_{q}\right)$-equivariant
$\phi$-coordinated quasi $L_{\widehat{\mathcal{L}}\left(\mathfrak{sl}_{\infty}\right)}\left(\ell,0\right)$-modules
$\left(W,Y_{W}^{\phi},d,p\right)$. \end{theorem}
\begin{proof}
Recall that $\widetilde{\mathfrak{sl}}_{N}\left(\mathbb{C}_{q}\right)=\widehat{\mathfrak{sl}}_{N}\left(\mathbb{C}_{q}\right)\oplus\mathbb{C}\text{d}_{0}\oplus\mathbb{C}\text{d}_{1}$.
It was proved in \cite[Proposition 3.8]{CLiTW} that the restricted $\widehat{\mathfrak{sl}}_{N}\left(\mathbb{C}_{q}\right)$-modules
$W$ of level $\ell$ are exactly the $\left(G_{N},\chi_{q}\right)$-equivariant
quasi $V_{\widehat{\mathcal{L}}\left(\mathfrak{sl}_{\infty}\right)}\left(\ell,0\right)$-modules
$\left(W,Y_{W}\right)$ with
\begin{align*}
z^{-1}\left(E_{i,j}t_{1}^{m}\right)\left[z\right]=Y_{W}\left(E_{mN+i,j},z\right), \ \ z^{-1}H_{k}\left[z\right]=Y_{W}\left(E_{k,k}-E_{k+1,k+1},z\right)
\end{align*}
for $1\le i,j,k\le N$, $m\in\Z$ with $\left(i-j,m\right)\ne\left(0,0\right)$.
It then follows from Proposition \ref{prop:affva1} that the restricted
$\widehat{\mathfrak{sl}}_{N}\left(\mathbb{C}_{q}\right)$-modules
$W$ of level $\ell$ are exactly the $\left(G_{N},\chi_{q}\right)$-equivariant
$\phi$-coordinated quasi $V_{\widehat{\mathcal{L}}\left(\mathfrak{sl}_{\infty}\right)}\left(\ell,0\right)$-modules
$\left(W,Y_{W}^{\phi}\right)$ with
\begin{align*}
\left(E_{i,j}t_{1}^{m}\right)\left[z\right]=Y_{W}^{\phi}\left(E_{mN+i,j},z\right), \ \ H_{k}\left[z\right]=Y_{W}^{\phi}\left(E_{k,k}-E_{k+1,k+1},z\right).
\end{align*}
Furthermore, by \eqref{eq:defpingrading} and Lemma \ref{lem:actofd}
we have
\begin{align*}
\left[-\text{d}_{1},\left(E_{i,j}t_{1}^{m}\right)\left[z\right]\right] & =-m\left(E_{i,j}t_{1}^{m}\right)\left[z\right]=-mY_{W}^{\phi}\left(E_{mN+i,j},z\right)=Y_{W}^{\phi}\left(\mathcal{P}\left(E_{mN+i,j}\right),z\right),\\
\left[-\text{d}_{1},H_{k}\left[z\right]\right] & =0=Y_{W}^{\phi}\left(\mathcal{P}\left(E_{k,k}-E_{k+1,k+1}\right),z\right),\\
\left[-\text{d}_{0},\left(E_{i,j}t_{1}^{m}\right)\left[z\right]\right] & =z\frac{d}{dz}\left(E_{i,j}t_{1}^{m}\right)\left[z\right]=z\frac{d}{dz}Y_{W}^{\phi}\left(E_{mN+i,j},z\right)=Y_{W}^{\phi}\left(\mathcal{D}\left(E_{mN+i,j}\right),z\right),\\
\left[-\text{d}_{0},H_{k}\left[z\right]\right] & =z\frac{d}{dz}H_{k}\left[z\right]=z\frac{d}{dz}Y_{W}^{\phi}\left(E_{k,k}-E_{k+1,k+1},z\right)=Y_{W}^{\phi}\left(\mathcal{D}\left(E_{k,k}-E_{k+1,k+1}\right),z\right).
\end{align*}
Therefore, we have finished the  proof for the first part of the theorem.

 To prove the second part of the theorem, we suppose that $\ell$ is a nonnegative integer and $W$ an integrable restricted $\widetilde{\mathfrak{sl}}_{N}\left(\mathbb{C}_{q}\right)$-module
of level $\ell$. Then for $1\le i\ne j\le N$ and $m,n\in\mathbb{Z}$,
 we have $\left(E_{mN+i,nN+j}\right){}_{r}\left(E_{mN+i,nN+j}\right)=0$
for $r\ge0$ in $V_{\widehat{\mathcal{L}}\left(\mathfrak{sl}_{\infty}\right)}\left(\ell,0\right)$. And
 as a $\left(G_{N},\chi_{q}\right)$-equivariant
$\phi$-coordinated quasi $V_{\widehat{\mathcal{L}}\left(\mathfrak{sl}_{\infty}\right)}\left(\ell,0\right)$-module
we have
\begin{align*}
Y_{W}^{\phi}\left(E_{mN+i,nN+j},z\right)=Y_{W}^{\phi}\left(\sigma_{N}^{n}\left(E_{(m-n)N+i,j}\right),z\right)=Y_{W}^{\phi}\left(E_{(m-n)N+i,j},\chi_{q}(\sigma_{N})^{n}z\right)=\left(E_{i,j}t_{1}^{m-n}\right)\left(q^{n}z\right).
\end{align*}
This implies that
\[
\left[Y_{W}^{\phi}\left(E_{mN+i,nN+j},z_{1}\right),Y_{W}^{\phi}\left(E_{mN+i,nN+j},z_{2}\right)\right]=\left[\left(E_{i,j}t_{1}^{m-n}\right)\left(q^{n}z_{1}\right),\left(E_{i,j}t_{1}^{m-n}\right)\left(q^{n}z_{2}\right)\right]=0.
\]
Thus, by Propositions \ref{prop:modnil} and \ref{prop:charintslmod},
we have $\left(\left(E_{mN+i,nN+j}\right){}_{-1}\right){}^{\ell+1}{\mathbf{1}}=0$
in $V_{\widehat{\mathcal{L}}\left(\mathfrak{sl}_{\infty}\right)}\left(\ell,0\right)/\ker Y_{W}^{\phi}$.
This together with Lemma \ref{lem:genofjslinfty} implies that $J_{\widehat{\mathcal{L}}\left(\mathfrak{sl}_{\infty}\right)}\left(\ell,0\right)\subset\ker Y_{W}^{\phi}$,
and hence $\left(W,Y_{W}^{\phi},d,p\right)$ is a $\left(G_{N},\chi_{q}\right)$-equivariant
$\phi$-coordinated quasi $L_{\widehat{\mathcal{L}}\left(\mathfrak{sl}_{\infty}\right)}\left(\ell,0\right)$-module.
Conversely, let $\left(W,Y_{W}^{\phi},d,p\right)$ be a $\left(G_{N},\chi_{q}\right)$-equivariant
$\phi$-coordinated quasi $L_{\widehat{\mathcal{L}}\left(\mathfrak{sl}_{\infty}\right)}\left(\ell,0\right)$-module.
Then it is a restricted $\widetilde{\mathfrak{sl}}_{N}\left(\mathbb{C}_{q}\right)$-module
of level $\ell$. Again by Lemma \ref{lem:genofjslinfty}, Proposition
\ref{prop:modnil} and Proposition \ref{prop:charintslmod}, one deduces that
$W$ is integrable as required.
\end{proof}

\section*{Acknowledgement}
This work was supported by the National Natural Science Foundation of China 11971396, 11971397,  and  the Fundamental Research Funds for the Central Universities 20720190069, 20720200067. The authors would like to thank Institute for Advanced Study in Mathematics, Zhejiang. Part of this work was carried out while the first and the third author were visiting the institute.

\vskip10pt

\textbf{F. Chen}{:
School of Mathematical Sciences, Xiamen University, Xiamen, 361005,
China; }\texttt{chenf@xmu.edu.cn}

\textbf{S. Tan}{:
School of Mathematical Sciences, Xiamen University, Xiamen, 361005,
China;} \texttt{tans@xmu.edu.cn}

\textbf{N. Yu}{:
School of Mathematical Sciences, Xiamen University, Xiamen, 361005,
China; } \texttt{ninayu@xmu.edu.cn}

\end{document}